\documentclass[12pt, twoside]{article}
\usepackage{amsfonts,mathrsfs,amsmath,amssymb,amsthm,upref,amscd}
\usepackage{xcolor}
\usepackage[all,cmtip]{xy}
\usepackage{setspace}
\usepackage{enumerate}
\usepackage[titletoc]{appendix}
\usepackage{times}
\usepackage{cite}
\usepackage{tikz}
\usepackage[colorinlistoftodos]{todonotes}
\usetikzlibrary{arrows}
\numberwithin{equation}{section}

\makeatletter 
\@mparswitchfalse
\makeatother
\normalmarginpar %for right-handed notes and lines, or 
\def\titlerunning#1{\gdef\titrun{#1}}
\makeatletter
\def\author#1{\gdef\autrun{\def\and{\unskip, }#1}\gdef\@author{#1}}
\def\address#1{{\def\and{\\\hspace*{18pt}}\renewcommand{\thefootnote}{}%
\footnote {#1}}%
\markboth{\autrun}{\titrun}}
\makeatother
\def\email#1{e-mail: #1}
\def\subjclass#1{{\renewcommand{\thefootnote}{}%
\footnote{\emph{Mathematics Subject Classification (2010):} #1}}}
\def\keywords#1{\par\medskip
\noindent\textbf{Keywords.} #1}

\allowdisplaybreaks

\theoremstyle{plain}
\newtheorem{Thm}{Theorem}[section]
\newtheorem{Lem}[Thm]{Lemma}
\newtheorem{question}{Question}
\newtheorem{Cor}[Thm]{Corollary}
\newtheorem{Prop}[Thm]{Proposition}
\newtheorem*{Thm*}{Theorem}
\newtheorem*{claim*}{Claim}
\theoremstyle{definition}
\newtheorem{Def}[Thm]{Definition}
\newtheorem*{Def*}{Definition}
\newtheorem*{Cor*}{Corollary}
\newtheorem{Rem}[Thm]{Remark}

\DeclareMathOperator{\tr}{tr}
\DeclareMathOperator{\real}{Re}
\DeclareMathOperator{\cat}{Cat}

\DeclareMathOperator{\vol}{vol}
\DeclareMathOperator{\dist}{dist}
\DeclareMathOperator{\supp}{supp}
\DeclareMathOperator{\im}{im}
\DeclareMathOperator{\coker}{coker}
\DeclareMathOperator{\tdeg}{deg}
\DeclareMathOperator{\ind}{ind}
\DeclareMathOperator{\cpl}{CL}
\DeclareMathOperator{\diag}{diag}

\newcommand{\equ}{equation}
\newcommand{\C}{\mathbb{C}}
\newcommand{\N}{\mathbb{N}}
\newcommand{\R}{\mathbb{R}}
\newcommand{\Z}{\mathbb{Z}}

\newcommand\eps{\varepsilon}

\newcommand\ca{\mathcal{A}}
\newcommand\cb{\mathcal{B}}
\newcommand\cc{\mathcal{C}}

\newcommand\cf{\mathcal{F}}
\newcommand\cg{\mathcal{G}}
\newcommand\ch{\mathcal{H}}

\newcommand\cj{\mathcal{J}}

\newcommand\cm{\mathcal{M}}
\newcommand\cn{\mathcal{N}}

\newcommand\cp{\mathcal{P}}

\newcommand\cs{\mathcal{S}}

\newcommand\cw{\mathcal{W}}

\newcommand\cz{\mathcal{Z}}
\newcommand{\inp}[2]{\left\langle#1,#2\right\rangle}

\def\mbs{\mathbb{S}}

\def\msd{\mathscr{D}}

\def\msf{\mathscr{F}}

\def\msh{\mathscr{H}}

\def\mfm{\mathfrak{M}}
\def\id{\text{Id}}
\def\ig{\textit{g}}
\def\ih{\textit{h}}

\def\ov{\overline}

\def\pa {\partial}

\def\De{\Delta}

\def\al{\alpha}
\def\bt{\beta}

\def\de{\delta}
\def\Ga{\Gamma}
\def\ga{\gamma}

\def\lm{\lambda}
\def\La{\Lambda}
\def\om{\omega}
\def\Om{\Omega}
\def\sa{\sigma}
\def\vr{\varepsilon}
\def\va{\varphi}

\begin{document}

\titlerunning{Solutions of  Spinorial Yamabe-type	Problems on $S^m$}

\title{Solutions of  Spinorial Yamabe-type	Problems on $S^m$: Perturbations and Applications}

\author{Takeshi Isobe \quad Tian Xu}%\footnote{Supported by the National Science Foundation of China (NSFC 11601370) and the Alexander von Humboldt Foundation of Germany}

\date{}

\maketitle

\address{T. Isobe: Graduate School of Economics, Hitotsubashi University, 2-1 Naka, Kunitachi, Tokyo 186-8601, Japan;
 \email{t.isobe@r.hit-u.ac.jp}
\and
T. Xu: Center for Applied Mathematics, Tianjin University, Nankai District, Tianjin 300072, China; \email{xutian@amss.ac.cn} 
\and
T. Isobe is supported by JSPS KAKENHI Grant Number 20K03674 and T. Xu is supported by the National Science Foundation of China (NSFC 11601370) and the Alexander von Humboldt Foundation of Germany}


\begin{abstract}
This paper is part of a program to establish the existence theory for the conformally invariant Dirac equation
\[
D_\ig\psi=f(x)|\psi|_\ig^{\frac2{m-1}}\psi 
\]
on a closed spin manifold $(M,\ig)$ of dimension $m\geq2$ with a fixed spin structure, where $f:M\to\R$ is a given function. The study on such nonlinear equation is motivated by its important applications in Spin Geometry: when $m=2$, a solution corresponds to an isometric immersion of the universal covering $\widetilde M$ into  $\R^3$ with prescribed mean curvature $f$; meanwhile, for general dimensions and $f\equiv constant$, a solution provides an upper bound estimate for the B\"ar-Hijazi-Lott invariant.

Comparing with the existing issues, the aim of this paper is to establish multiple existence results in a new geometric context, which have not been considered in the previous literature. Precisely, in order to examine the dependence of solutions of the aforementioned nonlinear Dirac equations on geometrical data, concrete analysis are made for two specific models on the manifold $(S^m,\ig)$: the geometric potential $f$ is a perturbation from constant with $\ig=\ig_{S^m}$ being the canonical round metric; and $f\equiv 1$ with the metric $\ig$ being a perturbation of $\ig_{S^m}$, that is not conformally flat somewhere on $S^m$. The proof is variational: solutions of these problems are found as critical points of their corresponding energy functionals. The emphasis is that the solutions are always degenerate: they appear as critical manifolds of positive dimension. This is very different from most situations in elliptic PDEs and classical critical point theory. As corollaries of the existence results, multiple distinct embedded spheres in $\R^3$ with a common mean curvature are constructed, and furthermore, a strict inequality estimate for the B\"ar-Hijazi-Lott invariant on $S^m$, $m\geq4$, is derived, which is the first result of this kind in the non-locally-conformally-flat setting.

\vspace{.5cm}
\end{abstract}

\tableofcontents

\section{Introduction}

Let $(M,\ig,\sa)$ be an $m$-dimensional closed spin manifold, $m\geq2$, with a fixed Riemannian metric $\ig$ and a fixed spin structure $\sa:P_{Spin}(M)\to P_{SO}(M)$. For the metric $\ig$, we obtain a spinor bundle $\mbs(M)\to M$, which is endowed with a natural hermitian metric $|\cdot|_{\ig}$. Sections of $\mbs(M)$ are usually called spinors. And the compatible covariant derivative on $\mbs(M)$, denoted by $\nabla^\mbs$, can be locally expressed as
\[
\nabla^\mbs\va_\al=\frac14\sum_{i,j=1}^m\ig(\nabla e_i,e_j)e_i\cdot_\ig e_j\cdot_\ig\va_\al
\]
where $\{e_i\}$ is a local orthonormal basis of $TM$ and $\{\va_\al\}$ is a local spinorial frame in $\mbs(M)$, the symbol $\nabla$ on the right hand side is the Levi-Civita connection on $(M,\ig)$ and ``$\cdot_\ig$" stands for the Clifford multiplication. With all these notations, let us introduce the Dirac operator $D_\ig: C^\infty(M,\mbs(M))\to C^\infty(M,\mbs(M))$ which is locally formulated as
\[
D_\ig\psi=\sum_{i=1}^{m}e_i\cdot_\ig\nabla^\mbs_{e_i}\psi. 
\]
for a spinor $\psi\in C^\infty(M,\mbs(M))$.

Under conformal transformations, the Dirac operator behaves analogously to the conformal Laplacian, which arises in the study of the classical Yamabe problem, see for instance \cite{Hij86, Hit74}. For this reason, a spinorial counterpart of the Yamabe invariant, which is known as the {\it B\"ar-Hijazi-Lott invariant} of the Dirac operator, was investigated in \cite{Ammann, Ammann2009}.

From the variational view point, the B\"ar-Hijazi-Lott invariant for $(M,\ig,\sa)$ is a positive real number $\lm_{min}^+(M,\ig,\sa)$ given by
\begin{\equ}\label{BHL-invariant}
\lm_{min}^+(M,\ig,\sa):=\inf\left\{\frac{\Big(
	\int_M |D_{\ig}\psi|_\ig^{\frac{2m}{m+1}}d\vol_{\ig}
	\Big)^{\frac{m+1}{m}}}{\int_M(D_{\ig}\psi,\psi)_\ig d\vol_{\ig}} \ \ \Bigg| \  \  \int_M(D_{\ig}\psi,\psi)_\ig d\vol_{\ig}> 0\right\}.
\end{\equ}
For $\tilde\ig=f^2\ig$ (with $f\in C^\infty(M)$ and $f>0$) and $\tilde\psi=f^{-\frac{m-1}2}\psi$ (via the identification of spinor bundles with respect to conformal metrics \cite{Hij86, Hit74}), we have that $\int_M(D_\ig\psi,\psi)_\ig d\vol_\ig=\int_M(D_{\tilde\ig}\tilde\psi,\tilde\psi)_{\tilde\ig}d\vol_{\tilde\ig}$ and $\int_M|D_\ig\psi|_\ig^{\frac{2m}{m+1}}d\vol_\ig=\int_M|D_{\tilde\ig}\tilde\psi|_{\tilde\ig}^{\frac{2m}{m+1}}d\vol_{\tilde\ig}$. Hence, $\lm_{min}^+(M,\ig,\sa)$ is actually a conformal invariant. Furthermore, the B\"ar-Hijazi-Lott invariant is tightly connected to and behaves much like the classical Yamabe invariant. In fact, it follows from Hijazi's inequality \cite{Hij86} that, for $m\geq3$, if the Yamabe invariant
\[
Y(M,[\ig])=\inf_{\tilde\ig\in[\ig]}\frac{\int_M\text{Scal}_{\tilde\ig}d\vol_{\tilde\ig}}{\text{Vol}(M,\tilde\ig)^{\frac{m-2}m}} 
\]
is non-negative, then
\begin{\equ}\label{Hijazi inequality}
\lm_{min}^+(M,\ig,\sa)\geq \sqrt{\frac{m}{4(m-1)}Y(M,[\ig])}.
\end{\equ}
And the equality in \eqref{Hijazi inequality} holds for the round spheres.

In a similar way as for the Yamabe invariant, the value of the B\"ar-Hijazi-Lott invariant for an arbitrary closed spin manifold can not be larger than that for the round sphere (with the same dimension). This is known as the spinorial analogue of Aubin's inequality (see \cite{AGHM}):
\begin{\equ}\label{spinorial Aubin inequ}
	\lm_{min}^+(M,\ig,\sa)\leq \lm_{min}^+(S^m,\ig_{S^m},\sa_{S^m})=\frac m2 \om_m^{\frac1m}
\end{\equ}
where $(S^m,\ig_{S^m},\sa_{S^m})$ is the $m$-dimensional sphere equipped with its canonical metric $\ig_{S^m}$ and its unique spin structure $\sa_{S^m}$, and $\om_m$ stands for the volume of $(S^m,\ig_{S^m})$.  In the spirit of this, the next stage would consist in showing that \eqref{spinorial Aubin inequ} is a strict inequality if
$(M, \ig)$ is not conformally equivalent to $(S^m, \ig_{S^m})$. However, the strict inequality in \eqref{spinorial Aubin inequ} is only verified for some special cases (for instance, if $M$ is locally conformally flat, if $D_\ig$ is invertible and if the so-called {\it Mass endomorphism} is not identically zero \cite{AHM}, and all rectangular tori \cite{SX2020}), and a general result is still lacking (cf. \cite{ADHH, Ginoux, Hermann10}). The methods that can be used are sometimes similar to the ones of the Yamabe problem, but since we work with Dirac operator and spinors,  the reasoning is more involved
as the eigenvalues of the Dirac operator tend to both $-\infty$ and $+\infty$ and there is no maximum principle.

By looking at the variational problem  \eqref{BHL-invariant}, we can see the corresponding Euler-Lagrange equation is
\[
D_\ig\psi=\lm_{min}^+(M,\ig,\sa)|\psi|_\ig^{\frac2{m-1}}\psi \quad \text{with } \int_M|\psi|_\ig^{\frac{2m}{m-1}}d\vol_\ig=1
\]
or equivalently
\begin{\equ}\label{Euler-Lagrange BHL}
	D_\ig\psi=|\psi|_\ig^{\frac2{m-1}}\psi \quad \text{with } \int_M|\psi|_\ig^{\frac{2m}{m-1}}d\vol_\ig=\lm_{min}^+(M,\ig,\sa)^m
\end{\equ}
which is a first order semilinear PDE.
On closed spin manifolds, the existence of a solution to Eq. \eqref{Euler-Lagrange BHL} was obtained in \cite[Theorem 1.1]{Ammann2009} for $\lm_{min}^+(M,\ig,\sa)<\lm_{min}^+(S^m,\ig_{S^m},\sa_{S^m})$ by using the compactness of the subcritical Sobolev embeddings.

To go a bit further, we can slightly generalize Eq. \eqref{Euler-Lagrange BHL}  to consider the following one
\begin{\equ}\label{Dirac critical}
	D_\ig\psi=f(x)|\psi|_\ig^{\frac2{m-1}}\psi  \quad \text{on } M,
\end{\equ}
where $f:M\to\R$ is a given function. Geometrically speaking, Eq. \eqref{Dirac critical} is of particular interest in dimension $m=2$ because its solution provides a strong tool for showing the existence of prescribed mean curvature surfaces (here the function $f$ plays the role of the mean curvature). Special cases of such surfaces are constant mean curvature (CMC) surfaces (that is $f\equiv constant$) which have been studied before by completely different techniques, see for instance \cite{GB98}. The correspondence between a solution of Eq. \eqref{Dirac critical} on a Riemannian surface $M$ and a periodic conformal immersion (possibly with branching points) of the universal covering $\widetilde M$ into $\R^3$ with mean curvature $f$ is known as the spinorial Weierstra\ss\ representation. For details in this direction, we refer to \cite{Ammann, Ammann2009, Kenmotsu, Konopelchenko, KS, Friedrich98, Matsutani, Taimanov97, Taimanov98, Taimanov99} and references therein.

Besides the very special case $f\equiv constant$, which reduces Eq. \eqref{Dirac critical} to Eq. \eqref{Euler-Lagrange BHL},
there are only a few results about the existence of a solution to Eq. \eqref{Dirac critical}. For $f\neq constant$, Raulot \cite{Raulot} obtained an existence criterion for this problem which is similar to the strict inequality in \eqref{spinorial Aubin inequ}. Indeed, one of his results shows that if the Dirac operator $D_\ig$ is invertible, $f$ is positive and satisfies
\begin{\equ}\label{criterion}
	\lm_{min,f}\cdot\big(\max_{x\in M} f\big)^{\frac{m-1}{m}}<\lm_{min}^+(S^m,\ig_{S^m},\sa_{S^m}),
\end{\equ}
where
\[
\lm_{min,f}:=\inf\left\{
\frac{\Big(
	\int_M f^{-\frac{m-1}{m+1}}|D_{\ig}\phi|^{\frac{2m}{m+1}}d\vol_{\ig}
	\Big)^{\frac{m+1}{m}}}{\int_M(D_{\ig}\phi,\phi)d\vol_{\ig}} \ \ \Bigg| \ \ 
\int_M(D_{\ig}\psi,\psi)_\ig d\vol_{\ig}> 0
\right\},
\]
then there exists a solution to Eq. \eqref{Dirac critical}. However, the validity of condition \eqref{criterion} is not very clear. Particularly, the existence results in \cite{Raulot} does not apply for $(M,\ig)=(S^m,\ig_{S^m})$ since the strict inequality of \eqref{criterion} is not valid on the round spheres in any circumstance. For later developments, we mention that in \cite{Isobe13, Isobe15} the author obtained some sufficient conditions about $f$ under which there exists a non-trivial solution to Eq. \eqref{Dirac critical} on $(S^m,\ig_{S^m})$, $m\geq 2$. Precisely, up to a simple scaling, \cite{Isobe13, Isobe15} studied the case $f(x)=1+\eps H(x)$ for some $H\in C^2(S^m)$ and small $\eps>0$.

Much less is known about the existence of multiple solutions for Eq. \eqref{Dirac critical}. Examples can only be given on some product manifolds, see \cite{SX2020}, or on manifolds that are invariant under some group actions, see \cite{Maalaoui}.
Apart from these cases, we do not know any further multiplicity result. Hence, the purpose of this paper is to discuss new existence and multiplicity results for the problem \eqref{Dirac critical} and, when $m=2$, to give multiplicity results for the immersed spheres in $\R^3$ with prescribed mean curvature. Note that Eq. \eqref{Dirac critical} is invariant under the canonical action of $S^1=\{z\in\C:|z|=1\}$ on spinors. Moreover, for the cases $m\equiv 2,3,4$ (mod 8), the spinor bundle has a quaternionic structure which commutes with Clifford multiplication~\cite[Section 1.7]{Friedrich00} or \cite[Page 33, Table III]{LM}. Thus in these cases, Eq. \eqref{Dirac critical} is invariant under the canonical action of the unit quaternions $S^3=\{q\in \mathbb{H}:|q|=1\}$. Thus we are interested in multiple existence of $S^1$ or $S^3$-inequivalent solutions.

We are concerned with two concrete problems. First of all,  we study the case that the geometric potential in Eq. \eqref{Dirac critical} (i.e. the mean curvature function) is a perturbation from constant, namely, we are concerned with the equation
\begin{\equ}\label{Dirac problem 1}
	D_{\ig_{S^m}}\psi=(1+\eps H(x))|\psi|_{\ig_{S^m}}^{\frac2{m-1}}\psi
	\quad \text{on } S^m,
\end{\equ}
for $|\vr|>0$ small,
and we intend to improve the existence results in \cite{Isobe13, Isobe15}. Secondly, motivated by the (strict) inequality in \eqref{spinorial Aubin inequ}, we introduce a new family of metrics $\ig$ on $S^m$ (close to  $\ig_{S^m}$ but not conformally equivalent) such that the equation
\begin{\equ}\label{Dirac problem 2}
	D_\ig\psi=|\psi|_\ig^{\frac2{m-1}}\psi
	\quad \text{on } S^m,
\end{\equ}
has multiple solutions. An important feature of Eq. \eqref{Dirac problem 2} is that (as was indicated in \eqref{Euler-Lagrange BHL})  the integral $\int_{S^m}|\psi|_\ig^{\frac{2m}{m-1}}d\vol_\ig$ of a solution gives a description of the B\"ar-Hijazi-Lott invariant. And it would be of particular interest if one can derive a conformal spectral estimate for the Dirac operator $D_\ig$ so that \eqref{spinorial Aubin inequ} is a strict inequality.
 As we mentioned before, as far as geometric aspects are concerned, the solutions to both of the equations \eqref{Dirac problem 1} and \eqref{Dirac problem 2} correspond to conformal (branched) immersions of $S^2$ into $\R^3$. However, Eq. \eqref{Dirac problem 2} for $m=2$ does not reflect the geometric importance since, in view of Riemann mapping theorem, there is only one conformal structure on $S^2$.

\subsection{Existence results for the PDE problems}

Let us start with Eq. \eqref{Dirac problem 1}. For a function $H\in C^1(S^m)$, let us denote $Crit[H]$ the critical set of $H$. To state our results, we assume the following two standing conditions on $H$:
\begin{itemize}
	\item[$(\text{H-1})$] $H\in C^2(S^m)$ is a Morse function such that $\De_{\ig_{S^m}}H(p)\neq 0$ for $p\in Crit[H]$.
	\item[$(\text{H-2})$] $H$ satisfies that
	\[
	\sum_{p\in Crit[H],\ \De_{\ig_{S^m}}H(p)<0} (-1)^{\mfm(H,p)}\neq (-1)^m
	\]
	where $\mfm(H,p)$ is the Morse index of $H$ at $p\in Crit[H]$.
\end{itemize}
Here we mention that condition $(\text{H-2})$ is the well-known index counting condition which was first introduced in the scalar curvature problem in \cite{BC91, CGY93}.

In order to describe our general multiplicity result, it is necessary to put some additional requirements on $H$. Here we are interested in the case that the  function $H$ can be decomposed into several components such that each component  generates solutions to the Eq. \eqref{Dirac problem 1}.
For this purpose, for a point $p\in S^m$, let us denote $\pi_p:S^m\setminus\{p\}\to\R^m$ the standard stereographic projection. Then we introduce the following condition
\begin{itemize}
	\item[$(\text{H-3})$] for some $l\in\N$, $l\geq2$, $H$ has of the form
	\[
	H(x)=\left\{
	\aligned
	&\sum_{i=1}^l H_i\big( \pi_{p_0}^{-1}(\pi_{p_0}(x)-z_i) \big) & & x\in S^m\setminus\{p_0\} \\
	&\sum_{i=1}^l H_i(p_0)  & & x=p_0
	\endaligned
	\right.
	\]
	where $H_1,\cdots H_l\in C^2(S^m)$ have a common minimum point $p_0\in S^m$ and $z_1,\dots,z_l\in\R^m$ are fixed.
\end{itemize}
It is clear that $H_i\circ\pi_{p_0}^{-1}$ defines a $C^2$-function on $\R^m$ and $\lim_{|y|\to\infty}H_i\circ\pi_{p_0}^{-1}(y)=H_i(p_0)$ for each $i$. The function $H_i\big( \pi_{p_0}^{-1}(\pi_{p_0}(\cdot)-z_i) \big): S^m\setminus\{p_0\}\to\R$ can be viewed as a translation of $H_i$ on $S^m$ with $H_i(p_0)$ being fixed. Hence the condition $(\text{H-3})$ describes a function that is (approximately) concentrated on the points $\pi_{p_0}^{-1}(z_i)\in S^m\setminus\{p_0\}$, $i=1,\dots,l$, and is well-defined on $S^m$. Here, the only reason we keep $p_0$ being fixed is that we want to present the function $H$ in a simple manner. And, of course, $(\text{H-3})$ can be replaced by other forms. Now we prove the following existence result.

\begin{Thm}\label{main theorem 1}
	For $m\geq2$, let $\msh$ denote a family of functions in $C^2(S^m)$ as
	\[
	\msh:=\big\{ H\in C^2(S^m): \, H \text{ satisfies the conditions } (\mathrm{H}\text{-}1) \text{ and } (\mathrm{H}\text{-}2) \big\}.
	\]
	Then
	\begin{itemize}
		\item[$(1)$] for a generic function $H\in \msh$, there exists $\eps_0>0$ such that Eq. \eqref{Dirac problem 1} has at least $2^{[\frac{m}{2}]}$ distinct non-trivial $S^1$-orbits of solutions for $0<|\eps|<\eps_0$. Furthermore, when $m\equiv 2,3,4$ (mod 8), there are at least $2^{[\frac{m}{2}]-1}$ distinct non-trivial $S^3$-orbits of solutions for $0<|\eps|<\eps_0$;
	\item[$(2)$] if $H\in C^2(S^m)$ satisfies condition $(\mathrm{H}\text{-}3)$ with generic $H_1,\dots, H_l\in \msh$ and let $z_1,\dots,z_l\in\R^m$ be located far enough apart, then there exists $\eps_0>0$ such that Eq. \eqref{Dirac problem 1} has at least $2^{[\frac{m}{2}]}l$ distinct non-trivial $S^1$-orbits of solutions for $0<|\eps|<\eps_0$. Furthermore, when $m\equiv 2,3,4$ (mod 8), there are at least $2^{[\frac{m}{2}]-1}l$ distinct non-trivial $S^3$-orbits of solutions for $0<|\eps|<\eps_0$;
	\item[$(3)$] if $\psi_\eps$ denotes any of the solutions obtained in $(1)$ and $(2)$, then 
	\[
	\int_{S^m}(1+\eps H(x))|\psi_\eps|_{\ig_{S^m}}^{\frac{2m}{m-1}}d\vol_{\ig_{S^m}}=\Big(\frac m2\Big)^m\om_m + O(\eps)
	\]
	as $\eps\to0$, moreover, $|\psi_\eps|_{\ig_{S^m}}>0$ on $S^m$ provided that $|\eps|$ is small.
	\end{itemize}
\end{Thm}
\begin{Rem}
Here, by ``generic" we mean that the existence results holds for $H$ belonging to a residual subset of $\msh$.
\end{Rem}

\medskip

	Next, we turn to consider the existence results for Eq. \eqref{Dirac problem 2} equipped with a new family of metrics on the spheres. To characterize the new metric $\ig$ on $S^m$, we use again the stereographic projection $\pi_p:S^m\setminus\{p\}\to\R^m$  (for an arbitrarily fixed $p\in S^m$) to have the following  one-to-one correspondence between $\ig$ on $S^m\setminus\{p\}$ and a metric $\tilde\ig$ on $\R^m$:
	\begin{\equ}\label{the metric g}
		\tilde\ig=f^{-2}\cdot(\pi_p^{-1})^*\ig, \quad f(x)=\frac2{1+|x|^2}, \ x\in\R^m.
	\end{\equ}
	Clearly, if $\tilde\ig=\ig_{\R^m}$ is the canonical Euclidean metric, then the metric $\ig$ on $S^m\setminus\{p\}$ can be extended globally to the standard round metric. In what follows, we assume that $\tilde\ig$ takes the form $\tilde\ig=\ig_{\R^m}+\eps \ih$ where $\ih$ is a smooth symmetric bilinear form on $\R^m$. In particular, let us consider a specific case
	\begin{\equ}\label{metric form}
		\tilde\ig(x)=\diag\big(\tilde\ig_{11}(x),\dots,\tilde\ig_{mm}(x) \big)
		\quad \text{with} \quad 
		\tilde\ig_{ii}(x)=1+\eps\ih_{ii}(x), 
	\end{\equ}
    where $\ih_{ii}:\R^m\to\R$, $i=1,\dots,m$, are smooth functions.  Let us point out here that, for a general choice of $\ih$, the metric $\ig$ on $S^m\setminus\{p\}$ may be discontinuous at the point $p$. Hence, in order to extend $\ig$ globally on $S^m$, it is natural to require the entries $\ih_{ii}$, $i=1,\dots,m$, and their derivatives behave ``nicely" at infinity. To give an idea of our multiplicity results, let us consider $\ih$ to be compactly supported (that is to say $\cup_i\supp\ih_{ii}$ is a compact subset of $\R^m$). And then, we are interested in the situation where $\tilde\ig$ is not conformal to the Euclidean metric.

For the sake of clarity, we construct a simple and effective formulation for perturbations $\ih$ in \eqref{metric form}. Let us make this precise. 

\begin{Def}\label{def k-elementary}
\it Given a smooth $m\times m$ diagonal matrix function
		\[
		\ih(x)=\diag\big(\ih_{11}(x),\dots, \ih_{mm}(x) \big) \quad \text{for } x\in \R^m,\ m\geq2,
		\]
		and a point $\xi=(\xi_1,\dots,\xi_m)\in\R^m$. For $k\in\{1,\dots,m\}$ and $p\in[1,\infty)$, we say that $\ih$ is  $(k,p)$-elementary at $\xi$, if 	$\xi\not\in\supp \ih_{kk}$ and, for $x=(x_1,\dots,x_m)\in\R^m$ close to $\xi$ and $i\neq k$, 
		\[
		\ih_{ii}(x)=\ih_{ii}(\xi)+c_i(x_i-\xi_i)+c_k(x_k-\xi_k)+o(|x-\xi|^p)
		\]
		where $c_i\in\R$, $i=1,\dots,m$, are constants with particularly $c_k\neq0$. Moreover, if the $o(|x-\xi|^p)$ term vanishes identically in the above local expansion of $\ih_{ii}$'s, then we say $\ih$ is $(k,\infty)$-elementary at $\xi$. In this way, we call  $p\in[1,\infty)\cup\{\infty\}$ the remainder exponent of $\ih$ at $\xi$.
\end{Def}

Without breaking the reading, examples and a brief explanation of such $(k,p)$-elementary matrices are given in the Appendix \ref{A-kpmatrix}. It can be seen from Definition \ref{def k-elementary} that ``{\it elementary}" matrix is a local concept. In the sequel, if it is clear from the context to which dimension we refer, we will simply use the name ``elementary matrix" to designate a member $\ih$ (without mentioning its tag numbers $k$, $p$ or the location point $\xi$). 
\begin{Thm}\label{main theorem 2}
	For $m\geq4$, let $\ig$ be the metric defined via \eqref{the metric g} and \eqref{metric form} with  
	\begin{\equ}\label{h splits}
	\ih(x)=\ih^{(1)}(x)+\ih^{(2)}(x-x_0),
	\end{\equ}
	where $x_0\in\R^m$ is fixed, $\ih^{(1)}$ and $\ih^{(2)}$ have compact supports and are elementary matrices. Assume additionally that, for $i=1,2$,
		\[
		\ih^{(i)} \text{ has remainder exponent } \begin{cases}
			p^{(i)}=\infty & m=4, \\
			p^{(i)}>2 & m=5,\\
			p^{(i)}\geq2 & m\geq6.
		\end{cases}
		\]
	 If $|x_0|$ is large enough, then
	\begin{itemize}
		\item[$(1)$] Eq. \eqref{Dirac problem 2} has at least two non-trivial solutions $\psi_1$ and $\psi_2$ provided that $|\vr|>0$ is small;
		\item[$(2)$] for each $i=1,2$, there holds
		\[
		\int_{S^m}|\psi_i|_{\ig}^{\frac{2m}{m-1}}d\vol_\ig=\Big(\frac m2\Big)^m\om_m + C_{i,m}\,\eps^2+o(\vr^2),
		\]
		where $C_{i,m}<0$ is a negative constant depending only on $\ih^{(i)}$ and the dimension $m$, $i=1,2$.
	\end{itemize}
\end{Thm}
\begin{Rem}
\begin{itemize}
\item[(1)] Although we proved the case for $\ih$ splits into two elementary matrices, Theorem \ref{main theorem 2} extends to the situation where $\ih$ splits into an arbitrary finite number of elementary matrices. Here the results do not reflect the canonical action of $S^1$ (or $S^3$ when $m\equiv2,3,4$ (mod 8)) as mentioned in Theorem \ref{main theorem 1} since each one of the solutions is generated by the components  of $\ih$ in \eqref{h splits} individually.

\item[(2)] For $p\geq2$, let $[p]\in\Z$ denotes the largest integer that does not exceed $p$. Definition \ref{def k-elementary} indicates that, when $\ih$ is $(k,p)$-elementary at some point $\xi$, all the $j$-th partial derivative of each $\ih_{ii}$ vanish at this point for $j=2,\dots,[p]$. One possible extension of this condition is to consider the case where the local expansion of $\ih_{ii}$ has quadratic terms. The proofs can be carried out by using similar arguments.
\end{itemize}
\end{Rem}

\subsection{Geometric applications}

As was mentioned before, the solutions of Eq. \eqref{Dirac problem 1} and Eq. \eqref{Dirac problem 2} has two major applications. The first application is that, when $m=2$, we obtain existence results for immersed spheres in $\R^3$ with prescribed mean curvature. And the second application gives us an upper bound estimate for the B\"ar-Hijazi-Lott invariant.

\begin{Rem}\label{Rem: geometric applications}
\begin{itemize}
\item[(1)] For $m=2$, Theorem \ref{main theorem 1} implies multiple existence of conformal immersions $S^2\to\R^3$ with mean curvature equals to $1+\eps H$. Precisely, if $\psi_\eps$ denotes any one of the solutions obtained in Theorem \ref{main theorem 1}, since $|\psi_\eps|_{\ig_{S^2}}>0$ we can introduce a conformal metric $\ig_\eps=|\psi_\eps|_{\ig_{S^2}}^4\ig_{S^2}$ on $S^2$. Then, due to the conformal covariance of the Dirac operator, we see that there is a spinor field $\va_\eps$ on $(S^2,\ig_\eps)$ such that
\[
D_{\ig_\eps}\va_\eps=(1+\eps H(x))\va_\eps \quad 
\text{and} \quad
|\va_\eps|_{\ig_\eps}\equiv1.
\]
Hence, by the spinorial Weierstra\ss\ representation, there is an isometric immersion $\Pi: (S^2,\ig_\eps)\to (\R^3,\ig_{\R^3})$ with mean curvature equals to $1+\eps H$. Furthermore, since the pull-back of the Euclidean volume form under this immersion is $\Pi^*(d\vol_{\ig_{\R^3}})=|\psi_\eps|_{\ig_{S^2}}^4d\vol_{\ig_{S^2}}$, the associated Willmore energy $W(\Pi)$ for this immersion satisfies
\[
W(\Pi)=\int_{S^2}(1+\eps H(x))^2|\psi_\eps|_{\ig_{S^m}}^4d\vol_{\ig_{S^2}}
\leq 4\pi+O(\eps)<8\pi
\]
as $\eps\to0$. Due to Li-Yau's inequality \cite[Theorem 6]{LY82}, the immersion $\Pi$ covers points in $\R^3$ at most once. Hence $\Pi$ is an embedding.

\item[(2)] The existence result in Theorem \ref{main theorem 2} directly implies a conformal spectral estimate for the Dirac operator $D_\ig$ on $(S^m,\ig)$.  In fact, by the definition of B\"ar-Hijazi-Lott invariant \eqref{BHL-invariant}, we have
\[
	\lm_{min}^+(S^m,\ig,\sa_{S^m})\leq\frac{\Big(
		\int_{S^m} |D_{\ig}\psi_i|_\ig^{\frac{2m}{m+1}}d\vol_{\ig}
		\Big)^{\frac{m+1}{m}}}{\int_{S^m}(D_{\ig}\psi_i,\psi_i)_\ig d\vol_{\ig}}=\Big(\int_{S^m}|\psi_i|_{\ig}^{\frac{2m}{m-1}}d\vol_\ig\Big)^{\frac1m}<\frac m2 \om_m^{\frac1m}
\]
for each $\psi_i$, $i=1,2$, obtained in Theorem \ref{main theorem 2}. Hence, the strict inequality in \eqref{spinorial Aubin inequ} holds true for the metric $\ig$. It would be very interesting whether analogous results hold for dimension $3$.
\end{itemize}
\end{Rem}

Now, by summarizing Remark \ref{Rem: geometric applications}, we have the following two geometric applications of Theorem \ref{main theorem 1} and \ref{main theorem 2}.
	
\begin{Thm}\label{geometric main theorem 1}
Let $H$ satisfy the conditions of Theorem \ref{main theorem 1} (2), there are at least $l$-distinct conformal isometric embeddings of $(S^2,\ig_{S^2})$ into $(\R^3,\ig_{\R^3})$ whose images have mean curvature $1+\eps H$, for $|\eps|>0$ small.
\end{Thm}

\begin{Thm}\label{geometric main theorem 2}
	For $m\geq4$, let $\ig$ be the metric as in Theorem \ref{main theorem 2}. Then the strict inequality 
	\begin{\equ}\label{strict inequ}
	\lm_{min}^+(S^m,\ig,\sa_{S^m})< \lm_{min}^+(S^m,\ig_{S^m},\sa_{S^m})
	\end{\equ}
	holds provided that $|\eps|>0$ is small. 
\end{Thm}

\begin{Rem}
	\begin{itemize}
		\item[$(1)$] Theorem \ref{geometric main theorem 1} can be viewed as part of a program to understand to what extent one can use mean curvature data to determine the shape of a surface in $\R^3$. Usually, a generic immersion is uniquely determined up to a rigid motion by its first fundamental form and its mean curvature function, but there are some exceptions, for instance most constant mean curvature immersions. A classical result by Bonnet states that if there exists a diffeomorphism $\Psi:\Sigma_1\to\Sigma_2$ between two closed immersed surfaces $\Sigma_1$, $\Sigma_2$ of genus zero in $\R^3$ such that $\Psi$ preserves both the metric and the mean curvature function of the surfaces, then $\Sigma_1$ and $\Sigma_2$ are congruent in $\R^3$ (i.e., they differ by a rigid motion). In view of this,  one may address the following natural questions:  
		\begin{question}
		Given a function $f:S^2\to\R$, what conditions on $f$ guarantee the existence of an isometric immersion $\Pi:S^2\to\R^3$ realizing $f$ as its mean curvature?
		\end{question} 
	    \begin{question}\label{question 2}
	    Assuming that $f:S^2\to\R$ is given such that it is the mean curvature function of an isometric immersion $\Pi: S^2\to\R^3$. Is it possible that there exists another isometric immersion $\widehat\Pi:S^2\to\R^3$, not congruent with $\Pi$, but also realizing $f$ as its mean curvature? If so, what can be said about the pull-back metrics $\Pi^*\ig_{\R^3}$ and $\widehat\Pi^*\ig_{\R^3}$?
	    \end{question}
        Note that in Question \ref{question 2} we are not assuming that the immersed spheres $\Pi(S^2)$ and $\widehat\Pi(S^2)$ are isometric, which is a critical hypothesis of Bonnet's result. Theorem \ref{geometric main theorem 1} provides answers to both Questions 1 and 2: we find sufficient conditions on $f$ such that it can be realized as the mean curvature of distinct isometric embeddings, moreover, the pull-back metrics are conformally related.

		\item[$(2)$] Theorem \ref{geometric main theorem 2} provides an explicit family of examples such that the strict inequality in  \eqref{spinorial Aubin inequ} holds. We are interested in obtaining this strict inequality for following reasons:
		\begin{itemize}
			\item The strict inequality in \eqref{spinorial Aubin inequ} implies that the invariant $\lm_{min}^+(M,\ig,\sa)$ is attained by a generalized metric, i.e., a metric of the form $f^2\ig$ where $f$ may have some zeros.

			\item Using Hijazi's inequality \cite{Hij86}, one obtains a solution of the standard Yamabe problem which consists of finding a metric with constant scalar curvature in the conformal class of $\ig$ in the case of $m\geq3$.
		\end{itemize}
		
		With the aim of obtaining this inequality, Ammann et al. considered in \cite{AHM} the case of a locally conformally flat spin manifold $(M,\ig,\sa)$ with $\ker D_\ig=\{0\}$ and introduced for a fixed point $p\in M$, a further datum, the {\it mass endomorphism}. The main result of \cite{AHM} is that the strict inequality in \eqref{spinorial Aubin inequ} holds, if $\ig$ is conformally flat around $p$ and the mass endomorphism at $p$ has a positive eigenvalue. Meanwhile, it is also proven in \cite{AHM} that the mass endomorphism always vanishes in dimension $2$ and on flat tori $\mathbb{T}^m$. To describe the dependence of the mass endomorphism on the Riemannian metric, let $\cm_{p,\text{flat}}(M)$ be the set of all Riemannian metrics $\ig$ on $M$, which are flat on a neighborhood of $p\in M$ and satisfy $\ker D_\ig=\{0\}$. It is proven in \cite{ADHH}, see also \cite{Hermann10}, that for dimension $m\geq3$, the subset of all Riemannian metrics with non-vanishing mass endomorphism at $p$ is dense in $\cm_{p,\text{flat}}(M)$ with respect to the $C^\infty$-topology. However, for an arbitrarily given metric $\ig$, examining the non-vanishingness of its associated mass endomorphism is still a difficult issue.
		
	    It can be seen from Theorem \ref{geometric main theorem 2} that, even if the metric $\ig$ is  flat somewhere on $M$, the appearance of the non-locally-conformally-flat part is the key reason we obtain the strict inequality. It is well known that, for $m\geq4$, the conformal flatness of $(M,\ig)$ is characterized by the nullity of the Weyl tensor. However, what is the exact role of the Weyl tensor in the general inequality $\lm_{min}^+(M,\ig,\sa)<\lm_{min}^+(S^m,\ig_{S^m},\sa_{S^m})$  when $(M,\ig)$ is non-locally conformally flat remains unexplored.	To the best of our knowledge, Theorem \ref{geometric main theorem 2} is the first result in the non-locally-conformally-flat setting. We consider it hence as remarkable that Theorem \ref{geometric main theorem 2} might give us a good chance to uncover the very nature of the B\"ar-Hijazi-Lott invariant in general cases. 
	\end{itemize}
\end{Rem}

\subsection{Overview of the method}

The proof of Theorem \ref{main theorem 1} and \ref{main theorem 2} relies on a suitable use of abstract perturbation methods in critical point theory. This idea is very common to many nonlinear problems, such as the classical Yamabe problem and related topics (see for instance \cite{AGP, AM99, AM, CY91, Mal02}). It was carried out in detail in the papers \cite{Isobe13, Isobe15} that such an idea can be employed to spinorial Yamabe equation \eqref{Dirac problem 1}. Roughly speaking, the perturbation methods consist of two main steps: first to use Lyapunov-Schmidt reduction to reduce the original problem to a certain finite dimensional variational problem and then to show that ``stable" critical points of the reduced problem produce solutions of the original one.

As was shown in \cite{Isobe13, Isobe15}, the reduced problem of spinorial Yamabe type equations is very similar to the corresponding problem for elliptic equations at a first glance, see for example \cite{AGP, AM, Mal02}. However, unlike the elliptic cases, here the classical abstract framework can not be used in a straightforward manner. This is because the finite dimensional problem associated to the spinorial Yamabe problem is degenerate, that is, any critical point of the main term of the reduced functional is not isolated, and the collection of these critical points appear as critical manifolds of positive dimension. Thus, it is not clear whether critical points of the reduced functional create true solutions of the original problem. To overcome this difficulty, the author developed, in \cite{Isobe13}, a variant of Morse-Bott theory and, in \cite{Isobe15}, a Conley index theory for such variational problem. 

In this paper, our approach is different from the ones used in the aforementioned works. We have a two-step reduction procedure before finding a critical point. We start with the Lyapunov-Schmidt reduction to obtain a reduced finite dimensional variational problem, and the new perspective is the observation that the reduced functional is defined on a space endowed with a fiber bundle structure, where the fiber is non-compact and the base space is a close manifold. The key point in our argument is that the degeneracy of the reduced functional essentially occurs on the base manifold. Hence, as a second step, we get a chance to restrict the reduced functional on each fiber rather than on the total space and finally (via implicit function theorem, minimizing/maximizing methods) to show the critical points of the reduced problem give rise to solutions of the original one. An advantage of our approach is that it produces proofs that are rather simpler than the ones of \cite{Isobe13, Isobe15}. And this might be meaningful to provide new insight on nonperturbative problems.

\subsection{Organization of the paper}

This paper is organized as follows: In the next section, we quickly review the transformation of the equations \eqref{Dirac problem 1} and \eqref{Dirac problem 2} into equivalent equations on Euclidean spaces, and then we set up our abstract perturbation framework with degenerate conditions. Section \ref{sec perturbation in nonlinearity} is concerned with our first problem, i.e., Eq. \eqref{Dirac problem 1} where the perturbation is appearing in the geometric potential. And Section \ref{sec perturbation in metric} deals with the second problem, that is, Eq. \eqref{Dirac problem 2} where the perturbation is involved in the background metric. Both in Section \ref{sec perturbation in nonlinearity} and \ref{sec perturbation in metric}, we first collect properties of the variational problem and then apply the abstract results in Section \ref{sec Preliminaries} to obtain the existence results. Finally, in the Appendix, we collect some straight calculations which are needed in the arguments.

\section{The precise framework}\label{sec Preliminaries}

\subsection{Reduction to problems on Euclidean spaces}

We first mention that the Eq. \eqref{Dirac problem 1} on $S^m$ is equivalent to an equation on $\R^m$ by conformal equivalence. Precisely, let $H\in C^2(S^m)$ satisfies the hypotheses of Theorem \ref{main theorem 1}, we assume $H$ takes its minimum at $p_0\in S^m$. Denote by $\pi_{p_0}:S^m\setminus\{p_0\}\to\R^m$ the stereographic projection from $p_0$, we have $(\pi_{p_0}^{-1})^*\ig_{S^m}=f^2\ig_{\R^m}$ where $\ig_{S^m}$ is the round metric on $S^m$ with constant sectional curvature $1$ and $f (x) = \frac{2}{1+r^2}$ ($r = |x|$ and $x \in \R^m \cong T_{p_0}S^m$ is the stereographic coordinate with respect to $\pi_{p_0}$). 

On the other hand, the general equation $D_\ig\psi=H(x)|\psi|_\ig^{\frac2{m-1}}\psi$ on a manifold $(M,\ig)$ is invariant under conformal change of the metric. In fact, let $\bar\ig=e^{2u}\ig$ for some function $u$ on $M$, there is an isomorphism of vector bundles $F:\mbs(M,\ig)\to\mbs(M,\bar\ig)$
(here $\mbs(M,\ig)$ and $\mbs(M,\bar\ig)$ are spinor bundles on $M$ with respect to the metrics $\ig$ and $\bar\ig$, respectively) which is a fiberwise isometry such that
\[
D_{\bar\ig}\big( F(\va)\big)= F\big( e^{-\frac{m+1}2u}D_\ig(e^{\frac{m-1}2u}\va) \big)
\]
for $\va\in C^\infty((M,\ig),\mbs(M,\ig))$ (for more detailed definitions and facts about Clifford algebras, spin structures on manifolds and Dirac operators, please consult \cite{Friedrich00, LM}). Thus, when $\psi$ is a solution to the equation $D_\ig\psi=H(x)|\psi|_\ig^{\frac2{m-1}}\psi$  on $(M, \ig)$, then $\va:= F(e^{-\frac{m-1}2u}\psi) $ also satisfies the same equation on $(M, \bar\ig)$: $D_{\bar\ig}\va=H(x)|\va|_{\bar\ig}^{\frac2{m-1}}\va$ on $(M, \bar\ig)$.

Applying the above observation to Eq. \eqref{Dirac problem 1} with $M=S^m$, if $\psi\in C^1(S^m,\mbs(S^m))$ is a solution to  the Eq. \eqref{Dirac problem 1}, then $\va=f^{\frac{m-1}2}F(\psi\circ\pi_{p_0}^{-1})$ satisfies the equation
\begin{\equ}\label{reduced Dirac problem 1}
D_{\ig_{\R^m}}\va=(1+\eps a(x))|\va|_{\ig_{\R^m}}^{\frac2{m-1}}\va
\quad \text{on }\R^m
\end{\equ}
where $a(x)=H(\pi_{p_0}^{-1}(x))$.

Conversely, by the regularity theorem for weak solutions for Eq. \eqref{Dirac critical} (see \cite[Appendix]{Isobe11}), for any solution $\va\in L^{\frac{2m}{m-1}}(\R^m, \mbs(\R^m))$ to Eq. \eqref{reduced Dirac problem 1} with H\"older continuous function $a$, there corresponds to a $C^1$-solution $\psi=\va\circ\pi_{p_0}$ to Eq. \eqref{Dirac problem 1} on $S^m$ with $H(x) = a(\pi_{p_0}(x))$. Therefore, the study of Eq. \eqref{Dirac problem 1} is equivalent to the study of Eq. \eqref{reduced Dirac problem 1}.

\medskip

For our second problem, i.e., Eq. \eqref{Dirac problem 2}, we have a similar equivalent problem on the Euclidean space. Recall that, by \eqref{the metric g}, we have the conformal equivalence $(S^m\setminus\{p\},\ig)\cong (\R^m,\tilde\ig)$. Thus, by repeating the above arguments, we find that Eq. \eqref{Dirac problem 2} is equivalent to the equation
\begin{\equ}\label{reduced Dirac problem 2-2}
D_{\tilde\ig}\va=|\va|_{\tilde\ig}^{\frac2{m-1}}\va \quad \text{on } \R^m
\end{\equ}
where $\tilde\ig=\ig_{\R^m}+\eps\ih$ as was assumed in \eqref{the metric g} and \eqref{metric form}. 

For later purpose, we point out that both Eq. \eqref{reduced Dirac problem 1}
and \eqref{reduced Dirac problem 2-2} can be seen as perturbations from the unperturbed equation
\begin{\equ}\label{unperturbed equ}
	D_{\ig_{\R^m}}\psi=|\psi|_{\ig_{\R^m}}^{\frac2{m-1}}\psi
	\quad \text{on }\R^m.
\end{\equ}
Hence, in the sequel, our perturbation framework will be build upon the study of Eq. \eqref{unperturbed equ} and its associated energy functional
\begin{\equ}\label{unperturbed functional}
	\cj_0(\psi)=\frac12\int_{\R^m}(\psi, D_{\ig_{\R^m}}\psi)_{\ig_{\R^m}} d\vol_{\ig_{\R^m}}-\frac{m-1}{2m}\int_{\R^m}|\psi|_{\ig_{\R^m}}^{\frac{2m}{m-1}}d\vol_{\ig_{\R^m}}
\end{\equ}
where $(\cdot,\cdot)_{\ig_{\R^m}}$ and $|\cdot|_{\ig_{\R^m}}$ are the canonical hermitian product and its induced metric on the spinor bundle $\mbs(\R^m)$.

\subsection{Configuration spaces}

To treat the Eq. \eqref{Dirac critical} and \eqref{unperturbed equ} from a variational point of view, it is necessary to set up a functional framework. Suitable function spaces are $H^{\frac12}(M,\mbs(M))$ and $\msd^{\frac12}(\R^m,\mbs(\R^m))$ of spinor fields which are introduced in \cite{Isobe11, Isobe13}. For completeness, we give the definitions as follows. 

Recall that the Dirac operator $D_\ig$ on a compact spin manifold $(M,\ig)$ is self-adjoint on $L^2(M, \mbs(M))$ and has compact resolvents (see \cite{Friedrich00, LM}). Particularly, there exists a complete orthonormal basis $\psi_1,\psi_2, . . .$ of the Hilbert space $L^2(M, \mbs(M))$ consisting of the eigenspinors of $D_\ig$: $D_\ig\psi_k = \lm_k\psi_k$. Moreover, $|\lm_k| \to\infty$as $k \to\infty$.

For $s\geq0$, we define the operator $|D_\ig|^s:L^2(M, \mbs(M))\to L^2(M, \mbs(M))$ by
\[
|D_\ig|^s\psi=\sum_{k=1}^\infty |\lm_k|^s\al_k\psi_k,
\]
for $\psi=\sum_{k=1}^\infty\al_k\psi_k\in L^2(M, \mbs(M))$. Denoted by $H^s(M,\mbs(M))$ the domain of $|D_\ig|^s$, then it is clear that $\psi=\sum_{k=1}^\infty\al_k\psi_k\in H^s(M,\mbs(M))$ if and only if $\sum_{k=1}^\infty|\lm_k|^{2s}|\al_k|^2<\infty$. And hence $H^s(M,\mbs(M))$ coincides with the usual Sobolev space of order $s$, that is $W^{s,2}(M,\mbs(M))$ (cf. \cite{Adams, Ammann}). Furthermore, we can equip $H^s(M,\mbs(M))$ with the inner product
\[
\inp{\psi}{\va}_{s,2}:=\real(|D_\ig|^s\psi,|D_\ig|^s\va)_2+\real(\psi,\va)_2
\]
where $(\cdot,\cdot)_2$ is the $L^2$-inner product on spinors and the induced norm $\|\cdot\|_{s,2}$. In this paper, we are mainly concerned with the space $H^\frac12(M,\mbs(M))$ for $M=S^m$. Notice that the spectrum of $D_{\ig_{S^m}}$ on $S^m$ is bounded away from $0$ and one checks easily that $\|\psi\|_{1/2}=\big||D_{\ig_{S^m}}|^\frac12\psi\big|_2$
defines an equivalent norm on $H^\frac12(S^m,\mbs(S^m))$.

For ease of notation, throughout this paper, we shall denote $L^q:=L^q(M,\mbs(M))$ with the norm $|\psi|_q^q= \int_M |\psi|^qd\vol_{\ig} $ for $q\geq1$ and denote $2^*=\frac{2m}{m-1}$  the critical Sobolev exponent of the embedding $H^{1/2}(M,\mbs(M))\hookrightarrow L^q(M,\mbs(M))$ for $1\leq q\leq 2^*$.

On $\R^m$, a similar function space will also be useful in our argument. Let us denote by $\msd^{\frac12}(\R^m,\mbs(\R^m))$ the set of spinor fields $\psi$ on $\R^m$ such that $\big||D_{\ig_{\R^m}}|^{1/2}\psi\big|_2^2
<\infty$ and $|\psi|_{2^*}<\infty$ with norm $\|\psi\|_{1/2,2}:=\big||D_{\ig_{\R^m}}|^{1/2}\psi\big|_2+|\psi|_{2^*}$. Here, $|D_{\ig_{\R^m}}|^{1/2}$ is defined via the Fourier transformation: $\msf(|D_{\ig_{\R^m}}|^{1/2}\psi)(\xi)=|\xi|^{1/2}\msf(\psi)(\xi)$
and $\big||D_{\ig_{\R^m}}|^{1/2}\psi\big|_2:=\big| |\cdot|^{1/2}\msf(\psi) \big|_2$. Notice that $\msd^{\frac12}(\R^m,\mbs(\R^m))$ is isomorphic to $H^\frac12(S^m,\mbs(S^m))$ via the stereographic projection. The dual space of $\msd^{\frac12}(\R^m,\mbs(\R^m))$ will be denoted by $\msd^{-\frac12}(\R^m,\mbs(\R^m))$ with norm $\|\cdot\|_{-1/2,2}$.

\subsection{The perturbation method with  degenerate conditions}

It is clear that, on $\msd^{\frac12}(\R^m,\mbs(\R^m))$, the functional $\cj_0$ in \eqref{unperturbed functional} is well-defined and is of $C^2$ class. We call $\cj_0$ the unperturbed functional. Its critical points are solutions to the unperturbed spinorial Yamabe equation \eqref{unperturbed equ}. Recall that Eq. \eqref{reduced Dirac problem 1} and \eqref{reduced Dirac problem 2-2} are perturbations from \eqref{unperturbed equ} and it is very natural to expect that solutions of  Eq. \eqref{reduced Dirac problem 1} and \eqref{reduced Dirac problem 2-2} can be obtained as perturbations of critical points of $\cj_0$. In general, a well known Lyapunov-Schmidt reduction technique provides a powerful tool, see for instance \cite[Chapter 10]{MW} and \cite[II, 6]{Chang} where the variational problem satisfies the Palais-Smale condition and \cite{AB, AB98, AM} for the case that the Palais-Smale condition fails. For completeness, following \cite{AM}, we sketch the idea as follows.

Let $(\ch,\inp{\cdot}{\cdot})$ be a Hilbert space with the associated norm $\|\cdot\|:=\inp{\cdot}{\cdot}^{1/2}$. Suppose that $L_0\in C^2(\ch,\R)$ and $\Ga\in C^2(\ch,\R)$ are given. For $\eps>0$ small, we consider the perturbed functional
\begin{\equ}\label{model problem}
L_\eps(z)=L_0(z)+\eps \Ga(z)+o(\eps).
\end{\equ}
Assume that $L_0$ has a non-degenerate critical manifold $\cm\subset\ch$, that is,
\begin{itemize}
\item[$(A1)$] $\cm$ is a $d$-dimensional $C^2$-submanifold of $\ch$ such that $\nabla L_0(z)=0$ for all $z \in \cm$,

\item[$(A2)$] $\cm$ is non-degenerate in the sense that for all $z \in\cm$, we have $T_z\cm = \ker \nabla^2 L_0 (z)$,

\item[$(A3)$] $\nabla^2 L_0 (z) : \ch \to\ch$ is a Fredholm operator with index zero for all $z \in\cm$.
\end{itemize}

Once the conditions $(A1)$-$(A3)$ are satisfied, the problem of finding critical points of $L_\eps$ on $\ch$ can be reduced to the same problem on $\cm$ for a reduced functional $L_\eps^{red}$ (defined below). And critical points of $L_\eps$ will be obtained as small perturbations of critical points of $L_0$.

Set $\cw_z:=T_z\cm^\bot$, where the orthogonal complement is taking with respect to $\inp{\cdot}{\cdot}$ in $\ch$. We look for critical points of $L_\eps$ in the form $u=z+w$ where $z\in\cm$ and $w\in \cw_z$. Let $P_z:\ch\to\cw_z$ be the orthogonal projection onto $\cw_z$, the Euler-Lagrange equation $\nabla L_\eps(z+w)=0$ is equivalent to
\begin{\equ}\label{L-S reduction}
\left\{\aligned
&P_z\nabla L_\eps(z+w)=0  & & (\textit{auxiliary equation})\\
&(I-P_z)\nabla L_\eps(z+w)=0 & &  (\textit{bifurcation equation}).
\endaligned \right.
\end{\equ}
Then, under the conditions $(A2)$ and $(A3)$, the auxiliary equation in \eqref{L-S reduction} can solved firstly for $w$ by applying the implicit function theorem: for arbitrary $z\in\cm$ there is a unique solution $w=w_\vr(z)\in\cw_z$ for small values of $\eps$. Furthermore, on any compact subset $\cm_c\subset\cm$, one can have an uniform estimate:
\begin{\equ}\label{AM esti}
\cm_c\ni z\mapsto w_\eps(z)\in \cw_z \text{ is } C^1 \text{ and }
\|w_\eps(z)\|,\ \|w_\eps'(z)\|= O(\eps) \text{ uniformly for }
z\in\cm_c.
\end{\equ}
For the proofs of \eqref{AM esti} and more detailed explanation about this construction, see \cite[Chapter 2]{AM}.

The next step is to consider the bifurcation equation in \eqref{L-S reduction}. For this end, we introduce the reduced functional $L_\eps^{red}:\cm\to\R$ by
\[
L_\eps^{red}(z)=L_\eps(z+w_\eps(z))
\]
Then we have the following theorem
\begin{Thm}[Theorem 2.12 in \cite{AM}]\label{AM Theorem}
Suppose $(A1)$-$(A3)$ are satisfied and assume for a compact subset $\cm_c\subset\cm$ and $\eps>0$ small, $L_\eps^{red}$ has a critical point $z_\eps\in\cm_c$. Then $u_\eps=z_\eps+w_\eps(z_\eps)$ is a critical point of $L_\eps$ on $\ch$.
\end{Thm}

Thanks to the uniform estimate \eqref{AM esti}, the reduced functional $L_\eps^{red}$ is well approximated in the sense that
\begin{\equ}\label{L approximate}
L_\eps^{red}(z)=L_0(z)+\eps\Ga(z)+o(\eps), \quad
\nabla L_\eps^{red}(z)=\eps\nabla \Ga(z)+o(\eps).
\end{\equ}
and $L_0(z)$ is constant on any connected component of $\cm$. Thus if $z\in\cm$ is a non-degenerate critical point of $\Ga$ in some certain sense (for example, the local degree of $\nabla\Ga$ at $z$ is non-zero), then $z$ generates a critical point of
$L_\eps$ on $\ch$ (see \cite{AM, Chang, MW} for details).

For our problem, we intend to apply the above abstract framework to the functionals defined by $L_0=\cj_0$ and some certain perturbation terms on $\ch=\msd^{\frac12}(\R^m,\mbs(\R^m))$ (see \eqref{perturbed functional 1} and Lemma \ref{expansion J functional} later). As was already shown in \cite[Section 5, 6]{Isobe13} that $\cj_0$ satisfies $(A1)$-$(A3)$ for a critical manifold $\cm$ defined as
\begin{\equ}\label{critical manifold}
\cm:=\big\{\psi_{\lm,\xi,\ga}:\, \lm>0,\ \xi\in\R^m,\ \ga\in\mbs_m,\ |\ga|=1\big\},
\end{\equ}
where
\begin{\equ}\label{critical manifold explicit}
\psi_{\lm,\xi,\ga}(x)=\frac{m^{\frac{m-1}2}\lm^{\frac{m-1}2}}
{\big(\lm^2+|x-\xi|^2 \big)^{\frac m2}}(\lm-(x-\xi))\cdot_{\ig_{\R^m}}\ga
\end{\equ}
for $\lm>0$, $\xi\in\R^m$, $\ga\in\mbs_m$ with $|\ga|=1$ ($\mbs_m$ is the spinor module, see\cite{Friedrich00, LM}) and $\cdot_{\ig_{\R^m}}$ denotes the Clifford multiplication with respect to the Euclidean metric. Note that $\cm$ is diffeomorphic to $(0,\infty)\times\R^m\times S^{2^{[\frac m2]+1}-1}(\mbs_m)$ via the canonical map $(\lm,\xi,\ga)\mapsto \psi_{\lm,\xi,\ga}$, where $S^{2^{[\frac m2]+1}-1}(\mbs_m)$ stands for the $(2^{[\frac m2]+1}-1)$-dimensional unit sphere in $\mbs_m$. And hence $\cm$ is a non-compact manifold and the dimension of $\cm$ is $m+2^{[\frac m2]+1}$.

By virtue of Theorem \ref{AM Theorem}, we intend to reduce the problem to a finite dimensional variational problem on $\cm$. But in our case, as we will see later, the reduced functional $L_\eps^{red}$ is very degenerate on $\cm$: the perturbation term $\Ga$ in \eqref{L approximate} is degenerate on $\cm$ (in the sense that the map $\nabla \Ga(z)|_{T_z\cm}:T_z\cm\to\R$ has a non-trivial kernel for any $z\in\cm$). Hence, standard methods as in \cite{AM99, AM, Chang, MW} do not apply.

In order to make the perturbation framework applicable, besides the assumptions $(A1)$-$(A3)$, we assume the following additional conditions for the critical manifold $\cm$:
\begin{itemize}
	\item [$(A4)$] $\cm$ admits a (globally) trivializable fiber bundle structure over a compact base space $\cn$ with projection $\vartheta:\cm\to\cn$ and fiber $\cg$. Precisely, there is a fiber preserving diffeomorphism $\iota:\cg\times\cn\to\cm$ such that the following diagram commutes
	\begin{displaymath}
	\xymatrix{
		\cg\times\cn \ar[r]^{\ \ \iota}  \ar[d]_{Proj} & \cm \ar[d]^{\vartheta} \\
		\cn\ar[r]^{\ \id} & \cn
	}
	\end{displaymath}

    \item[$(A5)$]  $T_\ga\,\cn\subset \ker\nabla(\Ga\circ\iota)(g,\ga)$ for any $(g,\ga)\in \cg\times\cn$, where we have identified $T_\ga\,\cn$ as a subspace of the total tangent space $T_{(g,\ga)}(\cg\times\cn)$.

\end{itemize}

\begin{Rem}
	In our application $\cn=S^{2^{[\frac m2]+1}-1}(\mbs_m)$, $\cg=(0,+\infty)\times\R^m$ and $\iota(g,\ga):=\psi_{\lm,\xi,\ga}$ for $g=(\lm,\xi)\in\cg$ and $\ga\in\cn$, hence we have a very natural bundle structure on $\cm$. Particularly, we note that there is a continuous action $\cg\times\cm\to\cm$ such that $\cg$ preserves the fibers of $\cm$ (i.e. if $(\mu,y)\in \cg$ and $\psi_{\lm,\xi,\ga}\in\cm_\ga$ then $\psi_{\lm,\xi,\ga}*(\mu,y)=\psi_{\lm\mu,\,\xi+y,\ga}\in\cm_\ga$). Hence the critical manifold in \eqref{critical manifold} is essentially a principal $\cg$-bundle. And since it admits a global section, we easily see that $\cm$ is trivializable. This is the reason we introduce $(A4)$. Since $\cm$ is parameterized via $\iota$, condition $(A4)$ makes the variational problem even clearer: it is equivalent to consider the functional $L_\eps^{red}\circ\iota:\cg\times\cn\to\R$.  Comparing with the standard theory in \cite{AM, Chang, MW}, the distinct new feature $(A5)$ describes a certain degenerate situation and, particularly, it implies that $\Ga\circ\iota(g,\ga)$ depends only on the variables in the fiber space $\cg$. Thus we set $\tilde\Ga(g)=\Ga\circ\iota(g,\ga)$. For later use, we distinguish $(A5)$ into the following two cases:
	\[
	\left\{
	\aligned
	&\ker\nabla(\Ga\circ\iota)(g,\ga)\neq T_{(g,\ga)}(\cg\times\cn) & &\text{for some } (g,\ga)\in\cg\times\cn, \\[0.5em]
	&\ker\nabla(\Ga\circ\iota)(g,\ga)\equiv T_{(g,\ga)}(\cg\times\cn) & &\text{for all } (g,\ga)\in\cg\times\cn,
	\endaligned	\right.
	\]
	and we shall establish some abstract results which are useful in the spinorial Yamabe problem.
\end{Rem}

\subsubsection*{Case 1: $\ker\nabla(\Ga\circ\iota)(g,\ga)\neq T_{(g,\ga)}(\cg\times\cn)$ for some $(g,\ga)\in\cg\times\cn$}

Clearly, in this case, $\Ga\circ\iota(g,\ga)\neq constant $ on $\cg\times\cn$ (equivalently, $\Ga(z)\neq constant$ on $\cm$).

\begin{Thm}\label{abstract result1}
	Let $L_0,\Ga\in C^2(\ch,\R)$ as in \eqref{model problem} and suppose that $(A1)$-$(A5)$ are satisfied. If there is a $\bar g\in\cg$ such that $\bar g$ is a strict local maximum or minimum of $\tilde\Ga$. Then, for $|\eps|$ small, the functional $L_\eps$ has a critical point on $\cm$.
	
\end{Thm}
\begin{proof}
	It is sufficient to consider $\bar g\in\cg$ is a strict local minimum and $\eps>0$, the other cases can be done similarly. Let $\de>0$ and $U\subset\cg$ be a bounded open neighborhood of $\bar g$ such that
	\[
	\tilde\Ga(g)\geq\tilde\Ga(\bar g)+\de, \quad \forall  g\in\pa U.
	\]
	By the first expansion in \eqref{L approximate}, we can see
	\[
	L_\eps^{red}\circ\iota(g,\ga)-L_\eps^{red}\circ\iota(\bar g,\ga)=\eps\big( \tilde\Ga(g)-\tilde\Ga(\bar g) \big) +o(\eps)>0 \quad \forall  g\in\pa U, \, \forall \ga\in\cn
	\]
	for $\eps>0$ small. Since $\ov U\times \cn$ is compact, by minimizing the functional $L_\eps^{red}\circ\iota$ in $\ov U\times\cn$ and by the above inequality, the minimum is attained in the interior and we obtain a critical point of $L_\eps^{red}\circ\iota$ and equivalently a critical point of
	 $L_\eps^{red}$ in $\iota(U\times\cn)\subset\cm$. Therefore, by Theorem \ref{AM Theorem}, it generates a critical point of $L_\eps$.
\end{proof}

In the hypotheses of Theorem \ref{abstract result1}, the  strict local maximum or minimum of $\tilde\Ga$ could possibly be degenerate. If we require certain non-degeneracy on the critical points of $\tilde\Ga$ we can have even more solutions as the following result indicates.

\begin{Thm}\label{abstract result2}
	Let $L_0,\Ga\in C^2(\ch,\R)$ as in \eqref{model problem} and suppose that $(A1)$-$(A5)$ are satisfied. If $\tilde\Ga$ is a Morse function on $\cg$ and there is an open bounded subset $\Om\subset\cg$ such that the topological degree $\tdeg(\nabla\tilde\Ga,\Om,0)\neq0$. Then, for $|\eps|$ small, the functional $L_\eps^{red}$ has at least $\cat(\cn)$ critical points on $\cm$.
\end{Thm}
Here $\cat(\cn)$ denotes the Lusternik-Schnierelman category of $\cn$, namely the smallest integer $k$ such that $\cn\subset \cup_{i=1}^k A_k$, where the sets $A_k$ are closed and contractible in $\cn$.
\begin{proof}
	Fix $\ga\in\cn$ arbitrarily and let us consider the functional $L_\eps^{red}|_{\vartheta^{-1}(\ga)}$. By the assumption $\tdeg(\nabla\tilde\Ga,\Om,0)\neq0$, there exists $\bar g\in\Om$ such that $\nabla\tilde\Ga(\bar g)=0$ and the local index of $\nabla\tilde\Ga$ at $\bar g$ is non-zero. By the second expansion in \eqref{L approximate} and the continuity property of the topological degree, we have
	\[
	\tdeg(\nabla L_\eps^{red}|_{\vartheta^{-1}(\ga)},\iota(\Om\times\{\ga\}),0)=\tdeg(\nabla\Ga|_{\vartheta^{-1}(\ga)},\iota(\Om\times\{\ga\}),0)=\tdeg(\nabla\tilde\Ga,\Om,0)\neq0
	\]
	for small values of $\eps$. Hence, for $\ga$ fixed, $L_\eps^{red}\circ\iota(\cdot,\ga)$ has a critical point $g(\ga)\in \Omega$. Moreover, we can choose $g(\ga)$ such that $g(\ga)=\bar g+o(1)$ as $\eps\to0$.
	
	If $\tilde\Ga$ is additionally a Morse function on $\cg$, then by the implicit function theorem, $g(\ga)$ is smooth as a function of $\ga\in\cn$. Note that we can choose the global branch of critical points $g(\ga)$ defined for all $\ga\in\cn$ by the condition $g(\ga)=\bar g + o(1)$. Hence, the search for critical points of $L_\eps^{red}$ on $\cm$ is reduced to the one for the following  functional defined on $\cn$
	\[
	\ga\mapsto L_\eps^{red}\circ\iota(g(\ga),\ga).
	\]
	By the well-known Lusternik-Schnierelman theory \cite{Struwe}, the above function has at least $\cat(\cn)$ critical points and we obtain the assertion.
\end{proof}

\begin{Rem}\label{compactness of the critcal points}
In the proof of Theorem \ref{abstract result2}, since $\tilde\Ga$ is a Morse function, any critical point of $\tilde\Ga$ is isolated. Hence, for each critical point $\bar g$ of $\tilde\Ga$, there exist at least $\cat(\cn)$ critical points $(g(\ga),\ga)\in\cg\times\cn$ of $L_\eps^{red}$ each of which satisfies $g(\ga)=\bar g +o(1)$ as $\eps\to0$.
\end{Rem}

\subsubsection*{Case 2: $\ker\nabla(\Ga\circ\iota)(g,\ga)\equiv T_{(g,\ga)}(\cg\times\cn)$ for all $(g,\ga)\in\cg\times\cn$}

In this case, we have $\Ga\circ\iota(g,\ga)\equiv constant$ on $\cg\times\cn$. As was pointed out in \cite{AM}, we need to evaluate further terms in the expansion of $L_\eps^{red}$. For this purpose, let us develop the expansion \eqref{model problem} in powers of $\eps$ as
\begin{\equ}\label{model problem2}
L_\eps(z)=L_0(z)+\eps\Ga(z)+\eps^2\Phi(z)+o(\eps^2).
\end{\equ}

Note that $\Ga\circ\iota(g,\ga)\equiv constant$ on $\cg\times\cn$ is equivalent to $\Ga(z)\equiv constant$ on $\cm$. It follows that $\nabla \Ga(z)\in \cw_z:=T_z\cm^\bot$. Recall that $w_\eps(z)$ is solution to the auxiliary equation $P_z\nabla L_\eps(z+w)=0$, hence we have
$\nabla L_\eps(z+w_\eps(z))\in T_z\cm$. For a fixed $z\in\cm$, using Taylor expansion, one sees
\[
\aligned
\nabla L_\eps(z+w_\eps(z))&=\nabla L_0(z+w_\eps(z))+\eps\nabla\Ga(z+w_\eps(z))+o(\eps)\\[0.2em]
&=\nabla^2L_0(z)[w_\eps(z)] +\eps\nabla\Ga(z)+\eps\nabla^2\Ga(z)[w_\eps(z)] + o(\|w_\eps(z)\|) +o(\eps).
\endaligned
\]
Then, form \eqref{AM esti} and the fact $\nabla L_\eps(z+w_\eps(z))\in T_z\cm$, it follows that
\[
\nabla^2L_0(z)[w_\eps(z)]+\eps \nabla\Ga(z)+o(\eps)\in T_z\cm.
\]
And hence we deduce
\[
w_\eps(z)=-\eps K_z(\nabla\Ga(z))+o(\eps),
\]
where $K_z$ stands for the inverse of $\nabla^2L_0(z)$ restricted to $\cw_z$. Now, we can expand $L_\eps^{red}(z):=L_\eps(z+w_\eps(z))$ as
	\[
	\aligned
	L_\eps^{red}(z)&=L_0(z)+\frac12\nabla^2L_0(z)[w_\eps(z),w_\eps(z)]+\eps\Ga(z)+\eps\nabla\Ga(z)[w_\eps(z)]+\eps^2\Phi(z)+o(\eps^2)\\
	&=L_0(z)+\eps \Ga(z)+\eps^2\Big( \Phi(z)-\frac12\inp{K_z(\nabla\Ga(z))}{\nabla\Ga(z)} \Big) +o(\eps^2).
	\endaligned
	\]
Here, we emphasize that both $L_0(z)$ and $\Ga(z)$ are constants on $\cm$. Hence by slightly changing the notation, for each fixed $\eps$, the above expansion can be understood as an analogue of the situation considered in \eqref{L approximate}. Therefore, similar to Theorem \ref{abstract result1}, we have the following result.

\begin{Thm}\label{abstract result3}
		Let $L_0,\Ga,\Phi\in C^2(\ch,\R)$ as in \eqref{model problem2} and suppose that $(A1)$-$(A5)$ are satisfied. If $\Ga\circ\iota\equiv constant$ on $\cg\times\cn$ and there is an open bounded subset $U\subset\cg$ such that
		\[
		\inf_{\ga\in\cn}\Big(\min_{\pa U}\hat\Phi\big|_{\vartheta^{-1}(\ga)} -\min_{\ov U}\hat\Phi\big|_{\vartheta^{-1}(\ga)}\Big)>0 \quad \text{or} \quad 
		\sup_{\ga\in\cn}\Big(\max_{\pa U}\hat\Phi\big|_{\vartheta^{-1}(\ga)} -\max_{\ov U}\hat\Phi\big|_{\vartheta^{-1}(\ga)}\Big)<0,
		\]
		where $\hat\Phi\big|_{\vartheta^{-1}(\ga)}=\hat\Phi\circ\iota(\cdot,\ga)$ and
		\[
		\hat\Phi(z):=\Phi(z)-\frac12\inp{K_z(\nabla\Ga(z))}{\nabla\Ga(z)} \quad \text{for } z\in\cm.
		\]
		 Then, for $|\eps|$ small, the functional $L_\eps$ has a critical point on $\cm$.
\end{Thm}

Last but not least, analogous to Theorem \ref{abstract result2}, we have

\begin{Thm}\label{abstract result4}
		Let $L_0,\Ga,\Phi\in C^2(\ch,\R)$ as in \eqref{model problem2} and suppose that $(A1)$-$(A5)$ are satisfied.  Assume further that $\cn$ is simply connected. If $\Ga\circ\iota\equiv constant$ on $\cg\times\cn$ and, for each $\ga\in\cn$, if $\hat\Phi|_{\vartheta^{-1}(\ga)}=\hat\Phi\circ\iota(\cdot,\ga)$ is a Morse function on $\cg$ and there is an open bounded subset $\Om(\ga)\subset\cg$ such that the topological degree $\tdeg(\nabla\hat\Phi|_{\vartheta^{-1}(\ga)},\iota(\Om(\ga)\times\{\ga\}),0)\neq0$. Then, for $|\eps|$ small, the functional $L_\eps^{red}$ has at least $\cat(\cn)$ critical points on $\cm$.
\end{Thm}
\begin{proof} 
	By the assumptions, for each $\gamma\in\cn$, there exists a critical point $g(\gamma)\in\Omega(\gamma)\subset\cg$.  For any $\gamma_0\in\cn$, by the Morse condition, we can choose a small neighborhood of $\gamma_0$ such that $g(\gamma)$ depends smoothly on $\ga$ in this neighborhood. These local branches of critical points are globally continued due to the simply connectivity of $\cn$ and we obtain a smooth $g(\gamma)$ defined globally on $\cn$. Then the proof proceed as in Theorem 2.4.
\end{proof}

\section{Perturbations of the geometric potential}\label{sec perturbation in nonlinearity}

In this section we deal with the problem \eqref{reduced Dirac problem 1}, namely,
\[
D_{\ig_{\R^m}}\psi=(1+\eps a(x))|\psi|_{\ig_{\R^m}}^{2^*-2}\psi
\quad \text{on } \R^m
\]
with $2^*=\frac{2m}{m-1}$. The associated functional can be written in the form
\begin{\equ}\label{perturbed functional 1}
\cj(\psi)=\cj_0(\psi)+\eps \Ga(\psi)
\end{\equ}
where
\[
\cj_0(\psi)=\frac12\int_{\R^m}(\psi,D_{\ig_{\R^m}}\psi)_{\ig_{\R^m}}d\vol_{\ig_{\R^m}}-\frac1{2^*}\int_{\R^m}|\psi|_{\ig_{\R^m}}^{2^*}d\vol_{\ig_{\R^m}}
\]
and 
\[
\Ga(\psi)=-\frac1{2^*}\int_{\R^m} a(x)|\psi|_{\ig_{\R^m}}^{2^*}d\vol_{\ig_{\R^m}}.
\]
We intend to use Theorem \ref{abstract result2} to obtain Theorem \ref{main theorem 1}. In order to do this, we need to have a detailed behavior of the functional $\Ga$ on the critical manifold $\cm$ mentioned in \eqref{critical manifold}-\eqref{critical manifold explicit}.

\subsection{The topological degree}

Recall that, for a function $H\in C^2(S^m)$, we have assumed $H$ takes its minimum at $p_0\in S^m$ and denote $\pi_{p_0}: S^m\setminus\{p_0\}\to\R^m$ the stereographic projection, and then the function $a$ is defined by $a(x):=H(\pi_{p_0}^{-1}(x))$ for $x\in\R^m$.

In what follows, we always assume that $a=H\circ\pi_{p_0}^{-1}$ for some function $H\in C^2(S^m)$ and, moreover, the conditions $(\text{H-1})$ and $(\text{H-2})$ are also satisfied.

\begin{Lem}\label{properties of a}
The function $a$ has the following properties:
\begin{itemize}
	\item[$(1)$] $a\in L^\infty(\R^m)\cap C^2(\R^m)$ and there is a constant $C_0>0$ such that
	\[
	|\pa_i a(x)|\leq C_0 (1+|x|^2)^{-1} \quad \text{and} \quad
	|\pa_i\pa_j a(x)|\leq C_0 (1+|x|^2)^{-\frac32}.
	\]
	where $\pa_i=\frac{\pa}{\pa x_i}$ and $\pa_i\pa_j=\frac{\pa^2}{\pa x_i\pa x_j}$.
	
	\item[$(2)$] $a$ is a Morse function. The number of critical points of $a$ on $\R^m$ is finite and $\De_{\ig_{\R^m}}a(x)\neq0$ for $x\in Crit[a]$. Moreover
	\[
	\sum_{x\in Crit[a],\ \De_{\ig_{\R^m}}a(x)<0} (-1)^{\mfm(a,x)}\neq (-1)^m
	\]
	where $Crit[a]$ stands for the critical set of $a$ on $\R^m$ and $\mfm(a,x)$ is the Morse index of $a$ at $x$.
	
	\item[$(3)$] there is $R_0>0$ such that $\nabla a(x)\cdot x<0$ for $|x|\geq R_0$ and 
	\[
	\int_{B_R}\nabla a(x)\cdot x \, d\vol_{\ig_{\R^m}}<0
	\]
	for any $R>R_0$.
\end{itemize}
\end{Lem}
\begin{proof}
For $(1)$ and $(2)$ we refer to \cite[Lemma 4.1]{Isobe13} for a detailed proof. Notice that a similar statements of $(3)$ were already proved in \cite[Lemma 3.2]{Isobe15}. Although, the dimension in \cite[Lemma 3.2]{Isobe15} was fixed as $m=3$, the same argument applies and hence we
omit the proof.
\end{proof}

Note that, for $\psi_{\lm,\xi,\ga}\in\cm$ (see \eqref{critical manifold explicit}),
\[
\Ga(\psi_{\lm,\xi,\ga})=-\frac{m-1}{2m}\int_{\R^m} a(x)|\psi_{\lm,\xi,\ga}|_{\ig_{\R^m}}^{\frac{2m}{m-1}}d\vol_{\ig_{\R^m}}
\]
and 
\[
|\psi_{\lm,\xi,\ga}|_{\ig_{\R^m}}=\frac{m^{\frac{m-1}{2}}\lm^{\frac{m-1}{2}}}{\big( \lm^2+|x-\xi|^2 \big)^{\frac{m-1}{2}}}.
\]
Then it is clear that $\Ga(\psi_{\lm,\xi,\ga})$ is independent of the factor $\ga\in S^{2^{[\frac m2]+1}-1}(\mbs_m)$. Hence, in order to study $\Ga(\psi_{\lm,\xi,\ga})$, it sufficient to consider (up to a constant)
\[
\Psi(\lm,\xi):=\int_{\R^m} a(x)|\psi_{\lm,\xi,\ga}|_{\ig_{\R^m}}^{\frac{2m}{m-1}}d\vol_{\ig_{\R^m}}=m^m\int_{\R^m}\frac{a(\lm x+\xi)}{(1+|x|^2)^m}d\vol_{\ig_{\R^m}}
\]
where the last equality comes from a change of variable. Now, by Lemma \ref{properties of a} $(1)$, one checks easily that, for any $\xi\in\R^m$,
\begin{\equ}\label{Psi derivative1}
\lim_{\lm\to0^+}\Psi(\lm,\xi)=C_1 a(\xi), \quad \text{where } C_1=m^m\int_{\R^m}\frac{1}{(1+|x|^2)^m}d\vol_{\ig_{\R^m}},
\end{\equ}
\[
\pa_\lm\Psi(\lm,\xi)=m^m\int_{\R^m}\frac{\nabla a(\lm x+\xi)\cdot x}{(1+|x|^2)^m} d\vol_{\ig_{\R^m}}
\]
and
\begin{\equ}\label{Psi derivative2}
\lim_{\lm\to0^+}\pa_\lm\Psi(\lm,\xi)=m^m\int_{\R^m}\frac{\nabla a(\xi)\cdot x}{(1+|x|^2)^m} d\vol_{\ig_{\R^m}}=0
\end{\equ}

With the above observation, by employing the ideas of Ambrosetti et al in \cite{AGP}, we intend to extend $\Psi$ on $(-\infty,\infty)\times\R^m$ as
\[
\widetilde\Psi(0,\xi)= C_1 a(\xi), \quad \widetilde\Psi(\lm,\xi)=\Psi(-\lm,\xi) \text{ for } \lm<0.
\]
By \eqref{Psi derivative1} and \eqref{Psi derivative2}, $\widetilde\Psi$ is of $C^1$ on $(-\infty,\infty)\times\R^m$ and 
\[
\pa_\lm\widetilde\Psi(0,\xi)=0 \quad \text{for } \xi\in\R^m.
\]
Moreover, we have
\[
\pa_i\widetilde\Psi(0,\xi)=\lim_{\lm\to0} m^m\int_{\R^m}\frac{\pa_i a(\lm x+\xi)}{(1+|x|^2)^m}d\vol_{\ig_{\R^m}}= C_1\pa_i a(\xi) \quad \text{for } \xi\in\R^m.
\]
Hence, we find the following one-to-one correspondence:
\begin{\equ}\label{critical point correspondence}
\xi\in Crit[a] \Longleftrightarrow (0,\xi)\in Crit[\widetilde\Psi]
\end{\equ}

Next, we will characterize the local index of $\widetilde\Psi$ at critical points of the form $(0,\xi)$ and evaluate the topological degree of $\widetilde\Psi$ in a bounded subset of $(0,\infty)\times \R^m$.
In what follows (for the proofs of Lemma \ref{degree 1}, Lemma \ref{degree 2} and Corollary \ref{degree 3}), since the calculation is essentially the same
as in \cite[Section 3]{Isobe15}, we shall sketch most of the details for general case $m\geq3$. Here, for the sake of completeness, we mainly specialize to the proofs for the case $m=2$, where some additional arguments are required.

\begin{Lem}\label{degree 1}
	There exist at most finitely many critical points of $\widetilde\Phi$ in the form of $(0,\xi)\in (-\infty,\infty)\times\R^m$. Moreover, these critical points are isolated and local index of $\widetilde\Psi$ at a critical point $(0,\xi)$ is
	\[
	\ind\big(\nabla \widetilde\Phi, (0,\xi)\big)=\left\{
	\aligned
	& (-1)^{\mfm(a,\xi)} & & \text{if } \De_{\ig_{\R^m}}a(\xi)>0, \\
	& (-1)^{\mfm(a,\xi)+1} & & \text{if } \De_{\ig_{\R^m}}a(\xi)<0 .
	\endaligned \right.
	\]
\end{Lem}
\begin{proof}
For $m\geq 3$, it follows directly from the computation in \cite[Section 7]{Isobe13} that
\[
\lim_{\lm\to0}\pa_\lm^2\widetilde\Psi(\lm,\xi)=C_2\De_{\ig_{\R^m}}a(\xi) \quad \text{where } C_2=m^{m-1}\int_{\R^m}\frac{|x|^2}{(1+|x|^2)^m}d\vol_{\ig_{\R^m}}
\]
\[
\lim_{\lm\to0}\pa_i\pa_j\widetilde\Psi(\lm,\xi)=C_1\pa_1\pa_ja(\xi)
\]
and 
\[
\lim_{\lm\to0}\pa_i\pa_\lm\widetilde\Psi(\lm,\xi)=0
\]
Hence the Hessian of $\widetilde\Phi$ at $(0,\xi)$ is 
\[
\nabla^2\widetilde\Phi(0,\xi)=\begin{pmatrix}
	C_2\De_{\ig_{\R^m}} a(\xi) & O\\[0.5em]
	O& C_1 \nabla^2 a(\xi)
\end{pmatrix}.
\]
If $(0,\xi_0)\in Crit[\widetilde\Phi]$, then one sees that 
\[
\widetilde\Phi(\lm,\xi) \sim \widetilde\Phi(0,\xi_0)+C_2\De_{\ig_{\R^m}} a(\xi_0) \lm^2+C_1 \nabla^2 a(\xi_0)[\xi-\xi_0,\xi-\xi_0]
\]
provided $(\lm,\xi)$ is close to $(0,\xi_0)$. Hence $(0,\xi_0)$ is isolated and the assertion follows from the multiplicative property of the index.

\medskip

For $m=2$, $\pa^2_\lm\widetilde\Phi(0,\xi)=\lim_{\lm\to0}\pa^2_\lm\widetilde\Phi(\lm,\xi)$ does not exist. Thus, we need to know detailed behavior of $\pa_\lm\widetilde\Phi(\lm,\xi)$ as $\lm\to0$.

To this end, for $\lm>0$, let us write
\[
\aligned
\pa_\lm\Phi(\lm,\xi)&=4\sum_{i=1}^2\int_{\R^2}\frac{\pa_ia(\lm x+\xi)x_i}{(1+|x|^2)^2}d\vol_{\ig_{\R^2}} \\
&=4\sum_{i=1}^2\int_{|x|\leq \frac1\lm}\frac{\pa_ia(\lm x+\xi)x_i}{(1+|x|^2)^2}d\vol_{\ig_{\R^2}} +4\sum_{i=1}^2\int_{|x|\geq\frac1\lm}\frac{\pa_ia(\lm x+\xi)x_i}{(1+|x|^2)^2}d\vol_{\ig_{\R^2}} \\
&=4I_1(\lm,\xi)+4I_2(\lm,\xi).
\endaligned
\]
Note that $a\in C^2(\R^2)$, by Taylor's theorem, we have
\[
\pa_ia(\lm x+\xi)=\pa_i a(\xi)+\lm\nabla\pa_ia(\xi)\cdot x+ o(|\lm x|)
\]
where $o(|\lm x|)\to0$ as $|\lm x|\to0$. Hence, we deduce
\[
\aligned
I_1(\lm,\xi)&=\sum_{i=1}^2\int_{|x|\leq \frac1\lm}\frac{\pa_ia(\xi)x_i}{(1+|x|^2)^2}d\vol_{\ig_{\R^2}}
+\sum_{i=j}^2\sum_{i=1}^2\int_{|x|\leq \frac1\lm}\frac{\lm\pa_j\pa_i a(\xi)x_jx_i}{(1+|x|^2)^2}d\vol_{\ig_{\R^2}} \\
&\qquad + o\Big( \int_{|x|\leq \frac1\lm}\frac{\lm|x|^2}{(1+|x|^2)^2}d\vol_{\ig_{\R^2}} \Big) \\
&=\frac\lm2\De_{\ig_{\R^2}} a(\xi)\int_{|x|\leq\frac1\lm}\frac{|x|^2}{(1+|x|^2)^2}d\vol_{\ig_{\R^2}}
+o(\lm\ln\lm) \\[0.2em]
&=-\pi\lm\ln\lm\big(\De_{\ig_{\R^2}} a(\xi)+o(1) \big)+o(\lm)
\endaligned
\]
and, by Lemma \ref{properties of a} $(1)$,
\[
I_2(\lm,\xi)=O\Big( \int_{|x|\geq\frac1\lm}\frac{|x|}{(1+|x|^2)^2}d\vol_{\ig_{\R^2}} \Big)
=O(\lm)
\]
as $\lm\to0$. Thus, we have
\[
\pa_\lm\Phi(\lm,\xi)=-4\pi\lm\ln\lm\big(\De_{\ig_{\R^2}} a(\xi)+o(1) \big)
\]
as $\lm\to0$. If $(0,\xi_0)\in Crit[\widetilde\Psi]$, since $\De_{\ig_{\R^2}} a(\xi_0)\neq 0$, it is clear that the function 
\[
\lm\mapsto -4\pi\De_{\ig_{\R^2}} a(\xi_0)\lm\ln\lm
\]
is homotopic to the function $\lm\mapsto \De_{\ig_{\R^2}} a(\xi_0)\lm$ through homotopy which does not take value $0$ when $\lm$ is close to $0$. Hence
\[
\ind\big(\pa_\lm\widetilde\Psi(\cdot,\xi_0), 0\big)=\left\{ 
\aligned
 1 & & \text{if } \De_{\ig_{\R^2}} a(\xi_0)>0, \\
-1 & & \text{if } \De_{\ig_{\R^2}} a(\xi_0)<0.
\endaligned
\right.
\]
And therefore, the assertion follows immediately due to the multiplicative property of the index.
\end{proof}

\begin{Lem}\label{degree 2}
	For $m\geq2$, there exists $R>0$ such that $Crit[\widetilde\Psi]\subset B_R^{m+1}:=\big\{ (\lm,\xi)\in(-\infty,\infty)\times\R^m: \lm^2+|\xi|^2\leq R^2 \big\}$ and 
	\[
	\tdeg(\nabla\widetilde\Psi, B_R^{m+1},0)=(-1)^{m+1}.
	\]
\end{Lem}
\begin{proof}
Although the arguments in \cite{Isobe15} apply to the general case $m\geq2$ without any difficulty, for completeness, we sketch the proof as follows.

We first show that there exists $R>0$ such that
\begin{\equ}\label{XXX1}
\nabla\widetilde\Psi(\lm,\xi)\cdot(\lm,\xi)<0	
\end{\equ}
for $(\lm,\xi)\in(-\infty,\infty)\times\R^m$ with $\lm^2+|\xi|^2\geq R^2$.

To see this, let us compute
\begin{eqnarray}\label{XXX2}
	\nabla\widetilde\Psi(\lm,\xi)\cdot(\lm,\xi)
	&=&m^m\int_{\R^m}\frac{\nabla a(\lm x+\xi)\cdot(\lm x+\xi)}{(1+|x|^2)^m}d\vol_{\ig_{\R^m}} \nonumber \\[0.5em]
	&=&m^m\lm^{-m}\int_{\R^m}\frac{\lm^{2m}\nabla a(x)\cdot x}{(\lm^2+|x-\xi|^2)^m}d\vol_{\ig_{\R^m}}.
\end{eqnarray}
Observe that, for arbitrary $r>0$, we can decompose the above integral as
\[
\aligned
\int_{\R^m}\frac{\lm^{2m}\nabla a(x)\cdot x}{(\lm^2+|x-\xi|^2)^m}d\vol_{\ig_{\R^m}}
&=\int_{|x|\leq r}\frac{\lm^{2m}\nabla a(x)\cdot x}{(\lm^2+|x-\xi|^2)^m}d\vol_{\ig_{\R^m}} \\[0.5em]
&\qquad +\int_{|x|\geq r}\frac{\lm^{2m}\nabla a(x)\cdot x}{(\lm^2+|x-\xi|^2)^m}d\vol_{\ig_{\R^m}}\\[0.5em]
&=I_1(r,\lm,\xi)+I_2(r,\lm,\xi)
\endaligned
\]
One checks easily that, by the first half of Lemma \ref{properties of a} $(3)$, $I_2(r,\lm,\xi)<0$ if $r>R_0$. On the other hand, we can decompose 
$I_1(r,\lm,\xi)$ as
\[
\aligned
I_1(r,\lm,\xi)&=\int_{|x|\leq r}\frac{\lm^{2m}(\nabla a(x)\cdot x )_+}{(\lm^2+|x-\xi|^2)^m}d\vol_{\ig_{\R^m}} -\int_{|x|\leq r}\frac{\lm^{2m}(\nabla a(x)\cdot x )_-}{(\lm^2+|x-\xi|^2)^m}d\vol_{\ig_{\R^m}}  \\[0.5em]
&\leq \max_{|x|\leq r}\frac{\lm^{2m}}{(\lm^2+|x-\xi|^2)^m} \int_{|x|\leq r} (\nabla a(x)\cdot x )_+d\vol_{\ig_{\R^m}} \\[0.5em]
&\qquad - \min_{|x|\leq r}\frac{\lm^{2m}}{(\lm^2+|x-\xi|^2)^m} \int_{|x|\leq r} (\nabla a(x)\cdot x )_-d\vol_{\ig_{\R^m}} 
\endaligned
\]
where $(\nabla a(x)\cdot x )_+$ and $(\nabla a(x)\cdot x )_-$ are the positive and  the negative parts of $\nabla a(x)\cdot x$, respectively.
To proceed, by the second half of Lemma \ref{properties of a} $(3)$, we mention that
\begin{\equ}\label{XXX3}
\de\int_{|x|\leq r} (\nabla a(x)\cdot x )_+d\vol_{\ig_{\R^m}} +\int_{|x|\leq r} \nabla a(x)\cdot x d\vol_{\ig_{\R^m}} <0
\end{\equ}
provided that $r\geq R_0$ and $\de$ fixed small enough. Set
\[
M(r,\lm,\xi)=\max_{|x|\leq r}\frac{\lm^{2m}}{(\lm^2+|x-\xi|^2)^m}=\left\{ 
\aligned
&\frac{\lm^{2m}}{(\lm^2+(|\xi|-r)^2)^m} & & \text{if } |\xi|>r\\
&1 && \text{if } |\xi|\leq r
\endaligned
\right.
\]
and 
\[
m(r,\lm,\xi)=\min_{|x|\leq r}\frac{\lm^{2m}}{(\lm^2+|x-\xi|^2)^m}=\frac{\lm^{2m}}{(\lm^2+(r+|\xi|)^2)^m},
\]
 we find $\lim_{\lm^2+|\xi|^2\to\infty}\frac{M(r,\lm,\xi)}{m(r,\lm,\xi)}=1$.  Hence, there exists $R_1>0$ such that
 \[
 1\leq \frac{M(r,\lm,\xi)}{m(r,\lm,\xi)} \leq 1+\de
 \]
 for $\lm^2+|\xi|^2\geq R_1$. Now, by \eqref{XXX3}, we can get the following estimate
 \[
 \aligned
 I_1(r,\lm,\xi)&\leq M(r,\lm,\xi)\int_{|x|\leq r} (\nabla a(x)\cdot x )_+d\vol_{\ig_{\R^m}} - m(r,\lm,\xi)\int_{|x|\leq r} (\nabla a(x)\cdot x )_-d\vol_{\ig_{\R^m}} \\
 &\leq m(r,\lm,\xi)\Big( \de\int_{|x|\leq r} (\nabla a(x)\cdot x )_+d\vol_{\ig_{\R^m}} +\int_{|x|\leq r} \nabla a(x)\cdot x d\vol_{\ig_{\R^m}} \Big) <0,
 \endaligned
 \]
 which indicates \eqref{XXX1}. 
 
 Since \eqref{XXX1} suggests $Crit[\widetilde\Psi]\subset B_R^{m+1}$ for some $R>0$ large, it remains to evaluate the topological degree of $\widetilde\Psi$. Note that \eqref{XXX1} forces $-t(\lm,\xi)+(1-t)\nabla\widetilde\Psi(\lm,\xi)\neq 0$ for any $(\lm,\xi)\in \pa B_R^{m+1}$ and $t\in[0,1]$. It follows that $\nabla\widetilde\Psi(\cdot,\cdot)$ is homotopic to $-id$ on $B_R^{m+1}$ through homotopy which does not take value $0$ on $\pa B_R^{m+1}$ and, particularly,
 \[
 \tdeg(\nabla\widetilde\Psi, B_R^{m+1},0)=\tdeg(-id, B_R^{m+1},0)=(-1)^{m+1}.
 \]
 This completes the proof.
\end{proof}

To get the topological degree of $\Psi$ on $(0,\infty)\times \R^m$, let us decompose $Crit[\widetilde\Psi]$ as
\[
Crit[\widetilde\Psi]=Crit_+[\widetilde\Psi]\cup Crit_0[\widetilde\Psi]\cup Crit_-[\widetilde\Psi]
\]
where 
\[
Crit_\pm[\widetilde\Psi]=\big\{(\lm,\xi)\in Crit[\widetilde\Psi]: \pm\lm>0 \big\} \quad \text{and} \quad
Crit_0[\widetilde\Psi]=\big\{(\lm,\xi)\in Crit[\widetilde\Psi]: \lm=0 \big\}.
\]
By the symmetry of $\widetilde\Psi$, we have
\[
Crit_-[\widetilde\Psi]=\big\{(\lm,\xi):\, (-\lm,\xi)\in Crit_+[\widetilde\Psi]\big\} .
\]
Thanks to Lemma \ref{degree 1} and \ref{degree 2}, we can find a bounded domain $\Om_+\subset(0,\infty)\times \R^m$ such that
\[
\Om_+\subset B_R^{m+1} \quad \text{and} \quad
Crit_+[\widetilde\Psi]\subset\Om_+.
\]
As a consequence of Lemma \ref{degree 2}, we have

\begin{Cor}\label{degree 3}
	$\tdeg(\nabla\Psi,\Om_+,0)\neq 0$.
\end{Cor}
\begin{proof}
Since $\Psi=\widetilde\Psi|_{(0,\infty)\times \R^m}$, it is sufficient to evaluate $\tdeg(\nabla\widetilde\Psi,\Om_+,0)$.

Assume to the contrary that $\tdeg(\nabla\widetilde\Psi,\Om_+,0)=0$. Let us set 
\[
\Om_-=\big\{(\lm,\xi):\, (-\lm,\xi)\in\Om_+\big\}.
\]
Then, it is clear that $Crit_-[\widetilde\Psi]\subset\Om_-$ and, by the assumption, $\tdeg(\nabla\widetilde\Psi,\Om_-,0)=0$. Due to 
the additive property of the degree, we can see from Lemma \ref{degree 1} and \ref{degree 2}
\begin{eqnarray}\label{d1}
	(-1)^{m+1}&=&\tdeg(\nabla\widetilde\Psi,B_R^{m+1},0)  \nonumber\\
	&=&\tdeg(\nabla\widetilde\Psi,\Om_+,0)+\tdeg(\nabla\widetilde\Psi,\Om_-,0)+\tdeg(\nabla\widetilde\Psi,B_R^{m+1}\setminus(\Om_+\cup\Om_-),0)   \nonumber \\
	&=&0+0+ \sum_{(0,\xi)\in Crit_0[\widetilde\Psi]}\ind\big(\nabla\widetilde\Psi, (0,\xi)\big)  \nonumber\\
	&=&\sum_{\xi\in Crit[a],\, \De_{\ig_{\R^m}} a(\xi)>0}\ind(\nabla a, \xi) 
	-\sum_{\xi\in Crit[a],\, \De_{\ig_{\R^m}} a(\xi)<0}\ind(\nabla a, \xi).
\end{eqnarray}

On the other hand, by Lemma \ref{properties of a} (3), we can proceed analogously to the proof of Lemma \ref{degree 2} to get
\[
\tdeg(\nabla a, B_R, 0)=(-1)^m
\]
where $B_R=\{x\in\R^m:\, |x|\leq R\}$. Since
\[
\tdeg(\nabla a, B_R, 0)=\sum_{\xi\in Crit[a],\, \De_{\ig_{\R^m}} a(\xi)>0}\ind(\nabla a, \xi) 
+\sum_{\xi\in Crit[a],\, \De_{\ig_{\R^m}} a(\xi)<0}\ind(\nabla a, \xi),
\]
it follows from \eqref{d1} that
\[
\sum_{\xi\in Crit[a],\, \De_{\ig_{\R^m}} a(\xi)>0}\ind(\nabla a, \xi) =0 \quad \text{and} \quad
\sum_{\xi\in Crit[a],\, \De_{\ig_{\R^m}} a(\xi)<0}\ind(\nabla a, \xi)=(-1)^m.
\]
Therefore, we have
\[
\sum_{x\in Crit[a],\ \De_{\ig_{\R^m}}a(x)<0} (-1)^{\mfm(a,x)}=\sum_{\xi\in Crit[a],\, \De_{\ig_{\R^m}} a(\xi)<0}\ind(\nabla a, \xi)=(-1)^m
\]
which contradicts to Lemma \ref{properties of a} $(2)$.
\end{proof}

\subsection{Proof of the main result}

For $l\in\N$, let $a_1,\dots,a_l\in L^\infty(\R^m)\cap C^2(\R^m)$ be functions satisfying the properties mentioned in Lemma \ref{properties of a}. For $z_1,\dots z_l\in\R^m$, we consider
\[
a(x)=\sum_{i=1}^l a_i(x-z_i).
\]
Then for $(\lm,\xi)\in(0,\infty)\times\R^m$ the function $\Psi$ takes the from
\[
\Psi(\lm,\xi)=\sum_{i=1}^l\Psi_i(\lm,\xi)
\]
where 
\[
\Psi_i(\lm,\xi)=m^m\int_{\R^m}\frac{a_i(\lm x+\xi-z_i)}{(1+|x|^2)^m}d\vol_{\ig_{\R^m}}
\]
Note that $\Psi_i$ is independent of the variable $\ga\in  S^{2^{[\frac m2]+1}-1}(\mbs_m)$.

The proofs in the previous section indicates that there exist bounded domains $\Om_i\subset (0,\infty)\times\R^m$, $i=1,\dots,l$, such that $Crit_+[\Psi_i]\subset \Om_i(z_i)$ and
\[
\tdeg(\nabla\Psi_i, \Om_i(z_i),0)\neq0
\]
where 
$
Crit_+[\Psi_i]=\big\{ (\lm,\xi):\, \nabla\Psi_i(\lm,\xi)=0 \text{ and } \lm>0\big\}
$
and $\Om_i(z_i)=\big\{(\lm,x):\, (\lm, x-z_i)\in\Om_i \big\}$.
Then, in order to apply Theorem \ref{abstract result2}, we are left with tasks to show
\begin{itemize}
	\item[$(T1)$]  $\nabla\Psi_i(\lm,\xi)\to0$ as $\lm+|\xi-z_i|\to\infty$;
	
	\item[$(T2)$] for generic choice of $a_i$ ($i=1,\dots,l$), $\Psi_i$ is a Morse function of $(\lm,\xi)$.
\end{itemize}
Notice that, if $z_1,\dots,z_l$ are located far apart, $\Om_i(z_i)\cap \Om_j(z_j)=\emptyset$ provided $i\neq j$. Thus, if the above two statements are proven,   $\Psi$ has at least one critical point in each $\Om_i(z_i)$ provided that $\min_{1\leq i\neq j\leq l}|z_i-z_j|$ is large enough. Then, recall that
\[
\Ga(\psi_{\lm,\xi,\ga})=-\frac{m-1}{2m}\Psi(\lm,\xi) \quad \text{for }
\psi_{\lm,\xi,\ga}\in\cm
\]
and
\[
\cj(\psi)=\cj_0(\psi)+\eps \Ga(\psi) \quad \text{for } \psi\in \msd^{1/2}(\R^m,\mbs(\R^m)).
\]
By applying Theorem \ref{abstract result2} and Remark \ref{compactness of the critcal points}, the functional $\cj$ has at least $l\cdot\cat\big(S^{2^{[\frac m2]+1}-1}(\mbs_m)\big) $ critical points. However, these critical points may appear as an $S^1$-orbit (or $S^3$-orbit when $m=2,3,4$ (mod 8)) of a single solution. To find non-$S^1$-equivalent (or $S^3$-equivalent) solutions, we observe that, due to the $S^1$-equivariance (or $S^3$-equivariance) property of Eq. \eqref{L-S reduction}, our construction of $w_{\eps}=w_{\eps}(z)$ can be done equivariantly with respect to these actions. Namely, $w_{\eps}(z)$ satisfies $w_{\eps}(\tau\star z)=\tau\star w_{\eps}(z)$ for $z\in\cm$ and $\tau\in S^1$ (or $\tau\in S^3$), where ``$\star$" stands for the corresponding group action. Then the reduced functional $L^{red}_{\eps}$ is $S^1$ (or $S^3$)-invariant, $L^{red}_{\eps}\circ\iota(g,\tau\star\ga)=L^{red}_{\eps}\circ\iota(g,\ga)$ for all $(g,\ga)\in\cg\times\cn$ and $\tau\in S^1$ (or $\tau\in S^3$ when $m=2,3,4$ (mod 8)). Therefore, $L^{red}_{\eps}$ can be reduced to be defined on $\cg\times\cn/S^1$ (or $\cg\times\cn/S^3$ when $m=2,3,4$ (mod 8)). We define $\hat{L}^{red}_{\eps}(g,[\ga])=L^{red}_{\eps}(g,\ga)$, where $[\ga]\in\cn/S^1$ (or $\cn/S^3$) is the equivalence class represented by $\ga\in\cn$.
Since $\cn=S^{2^{[\frac m2]+1}-1}(\mbs_m)$ is homeomorphic to the unit sphere in $\C^{2^{[\frac{m}{2}]}}$ (and in $\mathbb{H}^{2^{[\frac{m}{2}]-1}}$ when $m=2,3,4$ (mod 8)), $\cn/S^1\cong\C P^{2^{[\frac{m}{2}]}-1}$ ($\cn/S^3\cong\mathbb{H}P^{2^{[\frac{m}{2}]-1}-1}$ when $m=2,3,4$ (mod 8)). To estimate the Lusternik-Schnirelmann category of these quotients, recall that the category of a topological space $X$ is estimated from below by the cuplength, $\cat(X)\ge\cpl(X)+1$. Here the cuplength of $X$ is defined as the greatest integer $k\ge 0$ such that there exist $\om_1,\ldots,\om_k\in\mathrm{H}^{\ast}(X;\Z)$ so that $\deg\om_j>0$ and $\om_1\cup\cdots\cup\om_k\ne 0$. The cohomology rings of $\C P^{2^{[\frac{m}{2}]}-1}$ and $\mathbb{H}P^{2^{[\frac{m}{2}]-1}-1}$ are respectively isomorphic to truncated polynomial rings $\mathrm{H}^{\ast}(\C P^{2^{[\frac{m}{2}]}-1};\Z)=\Z[\al]/(\al^{2^{[\frac{m}{2}]}})$ and $\mathrm{H}^{\ast}(\mathbb{H}P^{2^{[\frac{m}{2}]-1}-1};Z)=\Z[\bt]/(\bt^{2^{[\frac{m}{2}]-1}})$, where $\al\in H^2(\C P^{2^{[\frac{m}{2}]}-1};\Z)$ and $\bt\in H^4(\mathbb{H}P^{2^{[\frac{m}{2}]-1}-1};Z)$. Thus the cuplength of the quotient $\cn/S^1$ and $\cn/S^3$ are $\cpl(\cn/S^1)=2^{[\frac{m}{2}]}-1$ and $\cpl(\cn/S^3)=2^{[\frac{m}{2}]-1}-1$ when $m=2,3,4$ (mod 8). Therefore, we have $\cat(\cn/S^1)\ge 2^{[\frac{m}{2}]}$ and $\cat(\cn/S^3)\ge 2^{[\frac{m}{2}]-1}$ when $m=2,3,4$ (mod 8). From this, we conclude that the functional $\cj$ has at least $2^{[\frac{m}{2}]}l$ critical points whose $S^1$-orbits are different in the general case $m\ge 2$ and when $m=2,3,4$ (mod 8) there are at least $2^{[\frac{m}{2}]-1}l$ critical points whose $S^3$-orbits are different.

To study the zero sets of these solutions, we first observe from Remark \ref{compactness of the critcal points} that: let $\psi_\eps$ be one of the solutions obtained above, then $\psi_\eps$ converge to some $\psi_{\lm,\xi,\ga}\in\cm$, as $\eps\to 0$. Note that
\[
\lim_{|x|\to\infty}f^{-\frac{m-1}2}(x)|\psi_{\lm,\xi,\ga}(x)|_{\ig_{\R^m}}=\Big(\frac m2\Big)^{\frac{m-1}2}\lm^{\frac{m-1}2}>0, \quad \text{where } f(x)=\frac{2}{1+|x|^2}
\]
it follows that $\psi_{\lm,\xi,\ga}$ corresponds to a solution $\va_0$ to the equation
\[
D_{\ig_{S^m}}\va=|\va|_{\ig_{S^m}}^{2^*-2}\va \quad \text{on } S^m
\]
via stereographic projection (see in Section \ref{sec Preliminaries}) with $|\va_0|_{\ig_{S^m}}>0$ on $S^m$. Hence, by the convergence of $\{\psi_\eps\}$, if we pull-back $\psi_\eps$ on $S^m$ we have 
\[
\big|(f^{{-\frac{m-1}2}}\psi_\eps)\circ\pi_p \big|_{\ig_{S^m}}=|\va_0|_{\ig_{S^m}}+o_\eps(1)>0
\]
as $\eps\to0$. And hence, the corresponding solutions to the Eq. \eqref{Dirac problem 1} have no zero on $S^m$.

\medskip

In what follows, the proofs of the above statements $(T1)$ and $(T2)$ will be carried out in Lemma \ref{lemma statement 1} and Proposition \ref{proposition generic}.

\begin{Lem}\label{lemma statement 1}
For each $i$,  $\nabla\Psi_i(\lm,\xi)\to0$ as $\lm+|\xi-z_i|\to\infty$.
\end{Lem}
\begin{proof}
	This easilly follows from Lemma \ref{properties of a} (1) and the dominated convergence theorem.
	
\end{proof}

\begin{Prop}\label{proposition generic}
	For $H\in C^2(S^m)$ and $(\lm,\xi)\in (0,\infty)\times \R^m$, we define
	\[
	\Psi(\lm,\xi;H)=m^m\int_{\R^m}\frac{a(\lm x+\xi)}{(1+|x|^2)^m}d\vol_{\ig_{\R^m}}
	\]
	where $a=H\circ \pi_{p}^{-1}$ ($\pi_{p}:S^m\setminus\{p\}\to\R^m$ is the stereographic projection with respect to a fixed $p\in S^m$). Then
	there exists a residual subset $\ca\subset C^2(S^m)$ such that,
	for all $H\in\ca$, $\Psi$ is a Morse function of $(0,\infty)\times \R^m$.
\end{Prop}

Recall that a subset $\cs$ in a Banach space $\cb$ is residual if $\cs$ contains a countable intersection of open dense subset of $\cb$. By the Baire’s category theorem, residual set is dense in $\cb$.

To prove the above proposition, we need some preparations. First, we define
\[
\cf(H,\lm,\xi)=\nabla \Psi(\lm,\xi;H)\in\R^{m+1}.
\]
Let us consider the following universal critical set
\[
\cc=\big\{ (H,\lm,\xi)\in C^2(S^m)\times(0,\infty)\times\R^m:\, \cf(H,\lm,\xi)=0 \big\},
\]
where $C^2(S^m)\times(0,\infty)\times\R^m$ is a separable Banach manifold modeled by $C^2(S^m)\times\R\times\R^m$ with a norm $\|(K,\mu,\zeta)\|=\|K\|_{C^2(S^m)}+|\mu|+|\zeta|$ for $(K,\mu,\zeta)\in T_{(H,\lm,\xi)}(C^2(S^m)\times(0,\infty)\times\R^m)=C^2(S^m)\times\R\times\R^m$. Note that $(H,\lm,\xi)\in\cc$ if and only if $(\lm,\xi)\in(o,\infty)\times\R^m$ is a critical point of $\Psi(\cdot,\cdot;H)$.

\begin{Lem}\label{lemma surjective}
	$0\in\R^{m+1}$ is a regular value of $\cf$, that is, the derivative
	\[
	\nabla\cf(H,\lm,\xi): C^2(S^m)\times \R\times\R^m\to\R^{m+1}
	\]
	is surjective at any $(H,\lm,\xi)\in \cc$. 
\end{Lem}
\begin{proof}
Let $(H,\lm,\xi)\in \cc$, we have
	\[
	\cf(H,\lm,\xi)=\Big( m^m\int_{\R^m}\frac{\nabla a(\lm x+\xi)\cdot x}{(1+|x|^2)^m}d\vol_{\ig_{\R^m}}, \,
	m^m\int_{\R^m}\frac{\nabla a(\lm x+\xi)}{(1+|x|^2)^m}d\vol_{\ig_{\R^m}}  \Big) \in\R\times\R^m.
	\]
	We shall prove that the restriction of $\nabla\cf$ to the subspace $C^2(S^m)\times\{0\}\times\{0\}$ is onto. Note that
	\[
	\aligned
	\nabla\cf(H,\lm,\xi)[K,0,0]&=\Big( m^m\int_{\R^m}\frac{\nabla b(\lm x+\xi)\cdot x}{(1+|x|^2)^m}d\vol_{\ig_{\R^m}}, \,
	m^m\int_{\R^m}\frac{\nabla b(\lm x+\xi)}{(1+|x|^2)^m}d\vol_{\ig_{\R^m}}  \Big) \\[0.2em]
	&=\cf(K,\lm,\xi)
	\endaligned
	\]
	where $b=K\circ\pi_p^{-1}$ for $K\in C^2(S^m)$.
	
	Assume to the contrary that the image of $C^2(S^m)\times\{0\}\times\{0\}$  by $\nabla\cf(H,\lm,\xi)$ is a proper subspace of $\R^{m+1}=\R\times\R^m$. Then there exists a non-zero vector $(\lm_0,\xi_0)\in\R\times\R^m$ with $\xi_0=(\xi_1,\dots,\xi_m)\in\R^m$ which annihilates the image. That is,
	\begin{\equ}\label{XX3}
	0=	\nabla\cf(H,\lm,\xi)[K,0,0]\cdot (\lm_0,\xi_0)=m^m\int_{\R^m}\frac{\nabla b(\lm x+\xi)\cdot (\lm_0x+\xi_0)}{(1+|x|^2)^m}d\vol_{\ig_{\R^m}}
	\end{\equ}
for arbitrary $K\in C^2(S^m)$ and $b=K\circ\pi_p^{-1}$.

For $1\leq i \leq m+1$, let $\bt_i\in C^\infty(S^m)$ be a smooth function obtained by restricting the $i$-th coordinate function of $\R^{m+1}$ on $S^m$, i.e., $\bt_i(x)=x_i|_{S^m}$ where $x=(x_1,\dots, x_{m+1})$. We define $K_i\in C^\infty(S^m)$ by $K_i(x)=\bt_i\big(\pi_p^{-1}(\frac{\pi_p(x)-\xi}{\lm}) \big)$. In other words, $K_i$ is the $i$-th component of the conformal diffeomorphism of $S^m$ which is $\R^m\ni x\mapsto \frac{x-\xi}{\lm}$ in term of the identification $S^m\cong \R^m\cup\{\infty\}$ by the stereographic projection $\pi_p$. Then we have $b_i(x)=K_i\circ\pi_p^{-1}(x)=\bt_i\big(\pi_p^{-1}(\frac{x-\xi}{\lm}) \big)$ for $x\in\R^m$.

Now, we substitute $K_i$ and $b_i$ into \eqref{XX3}. Since $\nabla b_i(x)=\frac1\lm\nabla(\bt_i\circ\pi_p^{-1})(\frac{x-\xi}{\lm})$ and $\nabla b_i(\lm x+\xi)=\frac1\lm\nabla(\bt_i\circ\pi_p^{-1})(x)$, we have
\begin{\equ}\label{XX4}
\int_{\R^m}\frac{\nabla(\bt_i\circ\pi_p^{-1})(x)\cdot(\lm_0 x+\xi_0)}{(1+|x|^2)^m}d\vol_{\ig_{\R^m}}=0
\end{\equ}
for $i=1,\dots,m+1$.

Note that $\bt_i\circ\pi_p^{-1}(x)=\frac{2x_i}{1+|x|^2}$ for $i=1,\dots,m$ and $\bt_{m+1}\circ \pi_p^{-1}(x)=\frac{|x|^2-1}{1+|x|^2}$ (here, without loss of generality, we assume $p$ is the north pole). We first take  $i=m+1$. Since $\nabla(\bt_{m+1}\circ \pi_p^{-1})(x)=\frac{4x}{(1+|x|^2)^2}$, \eqref{XX4} gives
\[
0=\int_{\R^m}\frac{4x\cdot(\lm_0 x+\xi_0)}{(1+|x|^2)^{m+2}}d\vol_{\ig_{\R^m}}=\lm_0\int_{\R^m}\frac{4|x|^2}{(1+|x|^2)^{m+2}}d\vol_{\ig_{\R^m}},
\]
and hence $\lm_0=0$.

For $1\leq i\leq m$, we have $\pa_j(\bt_i\circ\pi_p^{-1})(x)=\frac{2\de_{ij}}{1+|x|^2}-\frac{4x_ix_j}{(1+|x|^2)^2}$. Thus, it follows from \eqref{XX4} and $\lm_0=0$ that
\[
\aligned
0&=\int_{\R^m}\frac1{(1+|x|^2)^m}\Big[ \Big( \frac2{1+|x|^2}-\frac{4x_i^2}{(1+|x|^2)^2}\Big)\xi_i -\sum_{j\neq i}\frac{4x_ix_j\xi_j}{(1+|x|^2)^2}\Big]d\vol_{\ig_{\R^m}} \\[0.5em]
&=\xi_i\int_{\R^m}\frac{2(1+|x|^2-2x_i^2)}{(1+|x|^2)^{m+2}}d\vol_{\ig_{\R^m}}.
\endaligned
\]
Since
\[
\int_{\R^m}\frac{x_i^2}{(1+|x|^2)^{m+2}}d\vol_{\ig_{\R^m}}
=\frac1m\int_{\R^m}\frac{|x|^2}{(1+|x|^2)^{m+2}}d\vol_{\ig_{\R^m}}
<\frac12\int_{\R^m}\frac{1}{(1+|x|^2)^{m+1}}d\vol_{\ig_{\R^m}},
\]
we have $\xi_i=0$ for $i=1,\dots,m$. Therefore we have $(\lm_0,\xi_0)=0$ which is a contradiction. Thus the restriction of $\nabla\cf$ to the subspace $C^2(S^m)\times\{0\}\times\{0\}$ is onto.
\end{proof}

Since the kernel of $\nabla \cf(H,\lm,\xi)$ for $(H,\lm,\xi)\in\cc$ has a complemented subspace (since its codimension is finite), by the inverse function theorem, $\cc\subset C^2(S^m)\times(0,\infty)\times\R^m$ is a $C^1$-submanifold.

Let $\cp:\cc\to C^2(S^m)$ be the projection, i.e., $\cp(H,\lm,\xi)=H$. We next prove:
\begin{Lem}\label{lemma Fredholm}
	$\cp:\cc\to C^2(S^m)$ is a Fredholm map with index $0$.
\end{Lem}
\begin{proof}
	We first note that the tangent space of $\cc$ at $(H,\lm,\xi)\in\cc$ is 
	\[
	T_{(H,\lm,\xi)}\cc=\big\{ (K,\mu,\zeta)\in C^2(S^m)\times(0,\infty)\times\R^m:\, \nabla\cf(H,\lm,\xi)[K,\mu,\zeta]=0 \big\}.
	\]
Since $\nabla \cp(H,\lm,\xi)[K,\mu,\zeta]=K$ for $(K,\mu,\zeta)\in T_{(H,\lm,\xi)}\cc$, we see that $(K,\mu,\zeta)\in \ker\nabla P(H,\lm,\xi)$ if and only if $K=0$, and thus 
\[
\ker\nabla P(H,\lm,\xi)\cong \ker\nabla\cf(H,\lm,\xi)[0,\cdot,\cdot]\big|_{\{0\}\times \R\times\R^m}.
\]

On the other hand, let us define a map
\[
\aligned
\cz:\, C^2(S^m)  &\longrightarrow \R^{m+1}/\im\nabla\cf(H,\lm,\xi)[0,\cdot,\cdot]\big|_{\{0\}\times
 \R\times\R^m}  \\
K&\longmapsto [\cf(K,\lm,\xi)]
\endaligned
\]
where $[v]$ stands for the equivalence class of a vector $v\in\R^{m+1}$ in the quotient space. Then, from the fact
\[
\nabla \cf(H,\lm,\xi)[K,\mu,\zeta]=\cf(K,\lm,\xi)+\nabla\cf(H,\lm,\xi)[0,\mu,\zeta] 
\]
for all $ (K,\mu,\zeta)\in C^2(S^m)\times\R\times\R^m$, 
it follows that $\ker \cz=\im\nabla \cp(H,\lm,\xi)$. To proceed, let us claim that $\cz$ is surjective. In fact, since $\nabla\cf(H,\lm,\xi)$ is surjective by Lemma \ref{lemma surjective}, we see that for any vector $v\in\R^{m+1}$ there exists $(K,\mu,\zeta)\in C^2(S^m)\times\R\times\R^m$ such that
\[
v=\nabla\cf(H,\lm,\xi)[K,\mu,\zeta]=\cf(K,\lm,\xi)+\nabla\cf(H,\lm,\xi)[0,\mu,\zeta].
\]
This implies $[v]=[\cf(K,\lm,\xi)]$ in the quotient space $\R^{m+1}/\im\nabla\cf(H,\lm,\xi)[0,\cdot,\cdot]\big|_{\{0\}\times
	\R\times\R^m}$.
Therefore, the map $\cz$ induces an isomorphism
\[
\coker\nabla\cp(H,\lm,\xi)=\frac{C^2(S^m)}{\im\nabla\cp(H,\lm,\xi)}=
\frac{C^2(S^m)}{\ker\cz}\cong\frac{\R^{m+1}}{\im\nabla\cf(H,\lm,\xi)[0,\cdot,\cdot]\big|_{\{0\}\times\R\times\R^m}}.
\]
Hence the map $\cp$ is Fredholm with $\ind\nabla\cp(H,\lm,\xi)=\ind\nabla\cf(H,\lm,\xi)\big|_{\{0\}\times\R\times\R^m}=0$.
\end{proof}

\begin{proof}[Proof of Proposition \ref{proposition generic}]
By Lemma \ref{lemma Fredholm}, $\cp:\cc\to C^2(S^m)$ is a $C^1$-Fredholm map between separable $C^1$-Banach manifold with index $0$. By the Sard-Smale theorem, the set of regular values of $\cp$ is residual in $C^2(S^m)$. For any regular value $H\in C^2(S^m)$, the proof of Lemma \ref{lemma Fredholm} shows that 
\[
0=\dim\coker\nabla\cp(H,\lm,\xi)
=\dim\frac{\R^{m+1}}{\im\nabla\cf(H,\lm,\xi)[0,\cdot,\cdot]\big|_{\{0\}\times\R\times\R^m}}.
\]
By noting that
\[
\nabla\cf(H,\lm,\xi)[0,\mu,\zeta]=\nabla^2\Psi(\lm,\xi;H)[(\mu,\zeta)],
\quad \text{for } (\mu,\zeta)\in\R\times\R^m=\R^{m+1}
\]
we find the Hessian $\nabla^2\Psi(\lm,\xi;H)$ is non-degenerate. This means that $\Psi(\cdot,\cdot;H)$ is a Morse function on $(0,\infty)\times\R^m$. This completes the proof.
\end{proof}

\begin{Rem}\label{remark generic}
Let $H\in C^2(S^m)$ be an arbitrary function satisfying conditions $(\text{H-1})$ and $(\text{H-2})$, then there exists $\de>0$ such that
$\tilde H\in C^2$ with $\|H-\tilde H\|_{C^2(S^m)}<\de$ satisfies the same conditions. Proposition \ref{proposition generic} implies that for any given $H\in C^2(S^m)$ satisfying conditions $(\text{H-1})$ and $(\text{H-2})$, there exists $\tilde H\in C^(S^m)$, close to $H$ in $C^2$-norm, which satisfies the same conditions and the function
\[
(\lm,\xi)\longmapsto m^m\int_{\R^m}\frac{\tilde a(\lm x+\xi)}{(1+|x|^2)^m}d\vol_{\ig_{\R^m}}
\]
is a Morse function, where $\tilde a(x):=\tilde H\circ\pi_p^{-1}(x)$. Thus our claim ``$\Psi$ is Morse for generic choice of $a$" is verified.
\end{Rem}

\section{Perturbations of the background metric}\label{sec perturbation in metric}

In this section, let us consider the equation
\begin{\equ}\label{Dirac dimension m}
	D_\ig \psi= |\psi|_{\ig}^{2^*-2}\psi \quad \text{on } S^m
\end{\equ}
involving a perturbed metric, where $2^*=\frac{2m}{m-1}$.  Precisely, by using the stereographic projection $\pi_p:S^m\setminus\{p\}\to\R^m$ (for a fixed $p\in S^m$), the metric $\ig$ is defined via \eqref{the metric g} and \eqref{metric form}.
And in this setting, \eqref{Dirac dimension m} is equivalent to the equation
\begin{\equ}\label{Dirac euclidean m}
	D_{\tilde\ig} \psi=|\psi|_{\tilde\ig}^{2^*-2}\psi \quad \text{on } \R^m.
\end{\equ}

\subsection{Bourguignon-Gauduchon identification}

In this part we will collect some basic results which will enable us to write the energy functional $\cj$ for \eqref{Dirac euclidean m} in a form suitable for using the perturbation method mentioned before. In order to carry this out, we need to identify Dirac operators and spinor fields on $\R^m$ with respect to the canonical Euclidean metric $\ig_{\R^m}$ and the new metric $\tilde\ig$. The construction by Bourguignon and Gauduchon \cite{BG} provides us such a necessary identification.

To begin with, let us denote $\odot_+^2(\R^m)$ the space of  positive definite symmetric bilinear forms, i.e. metrics on $\R^m$. Then, for small values of $\eps>0$, we consider the identity map
\[
\aligned
\id: (\R^m,\ig_{\R^m})&\to (\R^m,\tilde\ig)  \\
x\ &\mapsto\  x
\endaligned
\]
and the map
\[
\aligned
\tilde G :\R^m \ &\to \ \odot_+^2(\R^m) \\
x\ &\mapsto \ \tilde G_x:=(\tilde{\ig}_{ij}(x))_{ij}
\endaligned
\]
which associates to a point $x\in\R^m$ the matrix of the coefficients of the metric $\tilde{\ig}$ at this point, expressed in the basis $\pa_i=\frac{\pa}{\pa x^i}$, $i=1,\dots,m$. Notice that $\tilde G_x\in\odot_+^2(\R^m)$, hence $G_x$ is invertible and there is a unique matrix $B_x\in\odot_+^2(\R^m)$ such that $B_x^2=\tilde G_x^{-1}$. Let $b_{ij}(x)$, $i,j=1,\dots,m$, be the entries of $B_x$, we have
\[
\aligned
B_x: (T_x\R^m\cong\R^m,\ig_{\R^m}) &\to (T_x\R^m,\tilde\ig_x) \\
v=\sum_k v_k\pa_k\ &\mapsto \ B_x(v):=\sum_j\big( \sum_k b_{jk}(x)v_k\big)\pa_j
\endaligned
\]
defines an isometry for each $x\in\R^m$. As the matrix $B_x$ depends smoothly on $x$, we obtain an isomorphism of $SO(m)$-principal bundles:
\begin{displaymath}
	\xymatrix{
		P_{SO}(\R^m,\ig_{\R^m}) \ar[r]^{ \ \eta}  \ar[d] & P_{SO}(\R^m,\tilde\ig) \ar[d] \\
		\R^m\ar[r]^{\ \ \id} & \R^m
	}
\end{displaymath}
where $\eta\{v_1,\dots,v_m\}=\{B(v_1),\dots,B(v_m)\}$ for an oriented frame $\{v_1,\dots,v_m\}$ on $(\R^m,\ig_{\R^m})$. Note that the map $\eta$ commutes with the right action of $SO(m)$, it can be lifted to spin structures:
\begin{displaymath}
	\xymatrix{
		P_{Spin}(\R^m,\ig_{\R^m}) \ar[r]^{ \tilde\eta} \ar[d] & P_{Spin}(\R^m,\tilde\ig)  \ar[d]\\
		\R^m \ar[r]^{\id}  &  \R^m
	}
\end{displaymath}
And hence we obtain an isomorphism between the spinor bundles $\mbs(\R^m,\ig_{\R^m})$ and $\mbs(\R^m,\tilde{\ig})$:
\begin{\equ}\label{spinor identify}
	\aligned
	\mbs(\R^m,\ig_{\R^m}) := P_{Spin}(\R^m,\ig_{\R^m})\times_\rho \mbs_m &\longrightarrow \mbs(\R^m,\tilde{\ig}) := P_{Spin}(\R^m,\ig)\times_\rho \mbs_m  \\
	\psi=[s,\va]&\longmapsto \tilde\psi=[\tilde\eta(s),\va]
	\endaligned
\end{\equ}
where $\rho$ is the complex spinor representation and $[s,\va]$ stands for the equivalence class of $(s,\va)$ under the action of $Spin(m)$. This identifies the spinor fields.

For the Dirac operators, as was shown by \cite[Proposition 3.2]{AGHM}, the identification can be expressed in the following formula
\begin{\equ}\label{Dirac identify}
		D_{\tilde\ig}\tilde\psi=\widetilde{D_{\ig_{\R^m}}\psi}+W\cdot_{\tilde\ig}\tilde\psi+X\cdot_{\tilde\ig}\tilde\psi+\sum_{i,j}(b_{ij}-\de_{ij})\tilde\pa_i\cdot_{\tilde\ig}\widetilde{\nabla_{\pa_j}\psi}
\end{\equ}
where $\cdot_{\tilde\ig}$ denotes the Clifford multiplication with respect to the metric $\tilde\ig$,
\[
W = \frac14\sum_{\substack{i,j,k \\ i\neq j\neq k\neq i}}\sum_{\al,\bt} b_{i\al}(\pa_{\al}b_{j\bt})b_{\bt k}^{-1}\,\tilde\pa_i\cdot_{\tilde\ig} \tilde\pa_j\cdot_{\tilde\ig} \tilde\pa_k,
\]
with $b_{ij}^{-1}$ being the entries of the inverse matrix of $B$, $\tilde\pa_i=B(\pa_i)$ and
\[
X = \frac12\sum_{i,k} \tilde\Ga_{ik}^i \tilde\pa_k,
\]
with $\tilde\Ga_{ij}^k=\tilde\ig(\tilde\nabla_{\tilde\pa_i}\tilde\pa_j,\tilde\pa_k) $ being the Christoffel symbols of the second kind. 

\begin{Rem}
	On spin manifolds, since the tangent bundle is embedded in the bundle of Clifford algebra, vector fields have two different actions on spinors, i.e. the Clifford multiplications and the covariant derivatives. Here, to distinguish the two actions on a spinor $\psi$, we denote $\pa_i\cdot_{\ig_{\R^m}}\psi$ the Clifford multiplication of $\pa_i$  and $\nabla_{\pa_i}\psi$ the covariant derivative with respect to the metric $\ig_{\R^m}$ (respectively, $\tilde\pa_i\cdot_{\tilde\ig}\tilde\psi$ the Clifford multiplication of $\tilde\pa_i$ and $\tilde\nabla_{\tilde\pa_i}\tilde\psi$ the covariant derivative with respect to the metric $\tilde\ig$). For functions, we shall simply denote $\pa_i u$ for its partial derivative.
	
\end{Rem}

\subsection{Basic calculations}

The following expansions come from elementary calculations.

\begin{Lem}\label{basic expansions}
Let $\tilde\ig$ be given by \eqref{metric form}, we have
\begin{\equ}\label{expansion det G}
	\sqrt{\det\tilde G}=1+\frac\eps2\tr\ih+\eps^2\Big(
	\frac18(\tr\ih)^2-\frac14\tr(\ih^2)\Big)+o(\eps^2),
\end{\equ}
\begin{\equ}\label{expansion B}
	B=I-\frac\eps2\ih+\frac{3\,\eps^2}8\ih^2+o(\eps^2)
\end{\equ}
and
\begin{\equ}\label{expansion B inverse}
	B^{-1}=I+\frac\eps2\ih-\frac{\eps^2}8\ih^2+o(\eps^2).
\end{\equ}
\end{Lem}

With the notation of Bourguignon-Gauduchon identification, for a spinor $\tilde\psi$ in $\mbs(\R^m,\tilde\ig)$, the energy functional associated to \eqref{Dirac euclidean m} is defined as
\begin{\equ}\label{functional dimension m}
	\cj(\tilde\psi)=\frac12\int_{\R^m}(\tilde\psi,D_{\tilde\ig}\tilde\psi)_{\tilde\ig}\,d\vol_{\tilde\ig}-\frac1{2^*}\int_{\R^m}|\tilde\psi|_{\tilde\ig}^{2^*}\,d\vol_{\tilde\ig}.
\end{\equ}
It turns out that
$\cj$ is an $\eps$-involved functional because $\tilde\ig$ is given by \eqref{metric form}. And now, via the Bourguignon-Gauduchon identification, we have the following expansion:

\begin{Lem}\label{expansion J functional}
Let $\tilde\ig$ be given by \eqref{metric form}, then
	\[
	\cj(\tilde\psi)=\cj_0(\psi)+\eps \Ga(\psi)+\eps^2\Phi(\psi)+o(\eps^2),
	\]
	where 
	\[
	\cj_0(\psi)=\frac12\int_{\R^m}(\psi,D_{\ig_{\R^m}}\psi)_{\ig_{\R^m}}\,d\vol_{\ig_{\R^m}}-\frac1{2^*}\int_{\R^m}|\psi|_{\ig_{\R^m}}^{2^*}\,d\vol_{\ig_{\R^m}},
	\]
	\[
	\Ga(\psi)=\int_{\R^m}\frac{\tr\ih}2\Big[ \frac12\big(\psi,D_{\ig_{\R^m}}\psi \big)_{\ig_{\R^m}}-\frac1{2^*}|\psi|_{\ig_{\R^m}}^{2^*}\Big]
	-\frac14\sum_i\ih_{ii}\real\big(\pa_i\cdot_{\ig_{\R^m}}\nabla_{\pa_i}\psi,\psi \big)_{\ig_{\R^m}} d\vol_{\ig_{\R^m}}
	\]
	and
	\[
	\aligned
	\Phi(\psi)&=\int_{\R^m}\Big(
	\frac18(\tr\ih)^2-\frac14\tr(\ih^2)\Big)\Big[ \frac12\big(\psi,D_{\ig_{\R^m}}\psi \big)_{\ig_{\R^m}}-\frac1{2^*}|\psi|_{\ig_{\R^m}}^{2^*}\Big] \\[0.5em]
	&\qquad
	+\frac1{16}\sum_i\Big( 3\ih_{ii}^2-2(\tr\ih)\ih_{ii} \Big)\real(\pa_i\cdot_{\ig_{\R^m}}\nabla_{\pa_i}\psi,\psi)_{\ig_{\R^m}} d\vol_{\ig_{\R^m}}
	\endaligned
	\]
	for $\psi\in\msd^{\frac12}(\R^m,\mbs(\R^m))$. 
\end{Lem}   
\begin{proof}
	We first expand the quadratic part of $\cj$ in terms of $\eps$. Using  \eqref{Dirac identify} and the fact $X\in T\R^m$, we get
	\[
	\real(\tilde\psi,D_{\tilde\ig}\tilde\psi)_{\tilde\ig} =\real
	(\tilde\psi,\widetilde{D_{\ig_{\R^m}}\psi})_{\tilde\ig}
	+\real(W\cdot_{\tilde\ig}\tilde\psi,\tilde\psi)_{\tilde\ig}
	+\sum_{i,j}(b_{ij}-\de_{ij})\real(\tilde{\pa_i}\cdot_{\tilde\ig}\widetilde{\nabla_{\pa_j}\psi},\tilde\psi)_{\tilde\ig}.
	\]
	Since the map $\psi\mapsto\tilde\psi$ defined in \eqref{spinor identify} is fiberwisely isometric, we obtain
	\[
	\real(\tilde\psi,\widetilde{D_{\ig_{\R^m}}\psi})_{\tilde\ig}=\real\big(\psi,D_{\ig_{\R^m}}\psi \big)_{\ig_{\R^m}}, \quad
	\real(\tilde{\pa_i}\cdot_{\tilde\ig}\widetilde{\nabla_{\pa_j}\psi},\tilde\psi)_{\tilde\ig}=\real(\pa_i\cdot_{\ig_{\R^m}}\nabla_{\pa_j}\psi,\psi)_{\ig_{\R^m}}
	\]
	and
	\[
	\real(W\cdot_{\tilde\ig}\tilde\psi,\tilde\psi)_{\tilde\ig}=\frac14
	\sum_{\substack{i,j,k \\ i\neq j\neq k\neq i}}\Big(\sum_{\al,\bt} b_{i\al}(\pa_{\al}b_{j\bt})b_{\bt k}^{-1}\Big)\real(\pa_i\cdot_{\ig_{\R^m}} \pa_j\cdot_{\ig_{\R^m}} \pa_k\cdot_{\ig_{\R^m}}\psi,\psi)_{\ig_{\R^m}}.
	\]
	Plainly, $\real(W\cdot_{\tilde\ig}\tilde\psi,\tilde\psi)_{\tilde\ig}\equiv0$ for  $m=2$. For $m\geq3$, by \eqref{expansion B} and \eqref{expansion B inverse}, we see that
	\[
	\left\{
	\aligned
	&b_{i\al}=\de_{i\al}-\frac\eps2\ih_{i\al}+\frac{3\,\eps^2}8\sum_l\ih_{il}\ih_{l\al}+o(\eps^2) , \\
	&\pa_\al b_{j\bt}=-\frac\eps2\pa_\al\ih_{j\bt}+\frac{3\,\eps^2}8\sum_l\Big(
	\pa_\al\ih_{jl}\ih_{l\bt}+\ih_{jl}\pa_\al\ih_{l\bt}\Big) + o(\eps^2), \\
	&b_{\bt k}^{-1}=\de_{\bt k}+\frac\eps2\ih_{\bt k}-\frac{\eps^2}8
	\sum_l\ih_{\bt l}\ih_{lk} + o(\eps^2).
	\endaligned
	\right.
	\]
	Hence
	\[
	\aligned
	b_{i\al}(\pa_{\al}b_{j\bt})b_{\bt k}^{-1}&=-\frac\eps2\de_{\bt k}\de_{i\al}\pa_\al\ih_{j\bt}+\frac{\eps^2}4\pa_\al\ih_{j\bt}\big( \de_{\bt k}\ih_{i\al}-\de_{i\al}\ih_{\bt k} \big) \\
	&\qquad +\frac{3\de_{\bt k}\de_{i\al}\eps^2}8\sum_l\Big(
	\pa_\al\ih_{jl}\ih_{l\bt}+\ih_{jl}\pa_\al\ih_{l\bt}\Big) + o(\eps^2).
	\endaligned
	\]
	Note that we have assumed $\ih_{ij}=0$ for $i\neq j$, we soon get
	\[
	b_{i\al}(\pa_{\al}b_{j\bt})b_{\bt k}^{-1}=o(\eps^2).
	\]
	And thus
	\begin{\equ}\label{expansion quadratic term}
		\aligned
		\real(\tilde\psi,D_{\tilde\ig}\tilde\psi)_{\tilde\ig} &=\real\big(\psi,D_{\ig_{\R^m}}\psi \big)_{\ig_{\R^m}}-\frac\eps2\sum_i\ih_{ii}\,\real(\pa_i\cdot_{\ig_{\R^m}}\nabla_{\pa_i}\psi,\psi)_{\ig_{\R^m}}  \\
		&\qquad +\frac{3\,\eps^2}8\sum_i\ih_{ii}^2\,\real(\pa_i\cdot_{\ig_{\R^m}}\nabla_{\pa_i}\psi,\psi)_{\ig_{\R^m}} + o(\eps^2)
		\endaligned
	\end{\equ}
	Since $|\tilde\psi|_{\tilde\ig}=|\psi|_{\ig_{\R^m}}$, it follows from \eqref{expansion det G} and \eqref{expansion quadratic term} that
	\[
	\aligned
	\cj(\tilde\psi)&=\cj_0(\psi) +\eps\int_{\R^m}\frac{\tr{\ih}}2\Big[ \frac12\big(\psi,D_{\ig_{\R^m}}\psi \big)_{\ig_{\R^m}}-\frac1{2^*}|\psi|_{\ig_{\R^m}}^{2^*}\Big] \\
	&\qquad \qquad \qquad
	-\frac14\sum_i\ih_{ii}\real\big(\pa_i\cdot_{\ig_{\R^m}}\nabla_{\pa_i}\psi,\psi \big) d\vol_{\ig_{\R^m}} \\
	&\qquad
	+\eps^2\int_{\R^m}\Big(
	\frac18(\tr\ih)^2-\frac14\tr(\ih^2)\Big)\Big[ \frac12\big(\psi,D_{\ig_{\R^m}}\psi \big)_{\ig_{\R^m}}-\frac1{2^*}|\psi|_{\ig_{\R^m}}^{2^*}\Big] \\[0.5em]
	&\qquad\qquad \qquad
	-\frac{\tr\ih}8\sum_i\ih_{ii}\real(\pa_i\cdot_{\ig_{\R^m}}\nabla_{\pa_i}\psi,\psi)_{\ig_{\R^m}} \\
	&\qquad\qquad \qquad
	+\frac{3}{16}\sum_i\ih_{ii}^2\,\real(\pa_i\cdot_{\ig_{\R^m}}\nabla_{\pa_i}\psi,\psi)_{\ig_{\R^m}} d\vol_{\ig_{\R^m}} +o(\eps^2)
	\endaligned
	\]
	which completes the proof.
\end{proof}

Recall the critical manifold $\cm\subset\msd^{\frac12}(\R^m,\mbs(\R^m))$ defined in \eqref{critical manifold}, an immediate consequence of the above lemma is
\begin{Cor}\label{detailed expansion J functional}
For $\psi_{\lm,\xi,\ga}\in\cm$ with $\lm>0$, $\xi\in\R^m$ and $\ga\in S^{2^{[\frac m2]+1}-1}(\mbs_m)$, we have
	\[
	\Ga(\psi_{\lm,\xi,\ga})=0
	\]
	and
	\[
	\Phi(\psi_{\lm,\xi,\ga})=\frac{m^{m-1}\lm^m}{16}\int_{\R^m}\frac{ \tr(\ih^2)-(\tr\ih)^2}{\big( \lm^2+|x-\xi|^2 \big)^m} d\vol_{\ig_{\R^m}}.
	\]  
\end{Cor}
\begin{proof}
	For  $\psi_{\lm,\xi,\ga}\in\cm$, one easily checks that it solves the equation
	\[
	D_{\ig_{\R^m}}\psi_{\lm,\xi,\ga}=|\psi_{\lm,\xi,\ga}|^{2^*-2}_{\ig_{\R^m}}\psi_{\lm,\xi,\ga}
	\]
	on $\R^m$. Furthermore, by the explicit expression of $\psi_{\lm,\xi,\ga}$, we have
	\[
	|\psi_{\lm,\xi,\ga}|_{\ig_{\R^m}}=\frac{m^{\frac{m-1}{2}}\lm^{\frac{m-1}{2}}}{\big( \lm^2+|x-\xi|^2 \big)^{\frac{m-1}{2}}}
	\]
	and
	\begin{\equ}\label{d psi}
		\aligned
		\nabla_{\pa_i}\psi_{\lm,\xi,\ga}(x) &=-\frac{m^{\frac{m+1}{2}}\lm^{\frac{m-1}{2}}}{\big( \lm^2+|x-\xi|^2 \big)^{\frac{m}{2}+1}}(x_i-\xi_i)(\lm-(x-\xi))\cdot_{\ig_{\R^m}}\ga	\\
		&\qquad
		+\frac{m^{\frac{m-1}{2}}\lm^{\frac{m-1}{2}}}{\big( \lm^2+|x-\xi|^2 \big)^{\frac{m}{2}}}(-\pa_i)\cdot_{\ig_{\R^m}}\ga.
		\endaligned
	\end{\equ}
	Hence, by the rules of Clifford multiplication, we get $\pa_i\cdot_{\ig_{\R^m}}\pa_i=-1$ and
	\begin{\equ}\label{Cd psi}
		\aligned
		\pa_i\cdot_{\ig_{\R^m}}\nabla_{\pa_i}\psi_{\lm,\xi,\ga}(x)
		&=-\frac{m^{\frac{m+1}{2}}\lm^{\frac{m-1}{2}}}{\big( \lm^2+|x-\xi|^2 \big)^{\frac{m}{2}+1}}(x_i-\xi_i)\,\pa_i\cdot_{\ig_{\R^m}}(\lm-(x-\xi))\cdot_{\ig_{\R^m}}\ga	\\
		&\qquad
		+\frac{m^{\frac{m-1}{2}}\lm^{\frac{m-1}{2}}}{\big( \lm^2+|x-\xi|^2 \big)^{\frac{m}{2}}}\ga	
		\endaligned
	\end{\equ}
	and
	\[
	\real(\pa_i\cdot_{\ig_{\R^m}}\nabla_{\pa_i}\psi_{\lm,\xi,\ga},\psi_{\lm,\xi,\ga})_{\ig_{\R^m}}
	=\frac{m^{m-1}\lm^m}{\big( \lm^2+|x-\xi|^2 \big)^m} \quad
	\text{for }i=1,\dots,m.
	\]
	Therefore,
	\[
	\Ga(\psi_{\lm,\xi,\ga})=\int_{\R^m}\bigg(\frac{\tr\ih}2\frac1{2m}\frac{m^m\lm^m}{\big( \lm^2+|x-\xi|^2 \big)^m}-\frac{\tr\ih}4\frac{m^{m-1}\lm^m}{\big( \lm^2+|x-\xi|^2 \big)^m} \bigg) d\vol_{\ig_{\R^m}}\equiv 0
	\]
	and
	\[
	\Phi(\psi_{\lm,\xi,\ga})=\frac{m^{m-1}\lm^m}{16}\int_{\R^m}\frac{ \tr(\ih^2)-(\tr\ih)^2}{\big( \lm^2+|x-\xi|^2 \big)^m} d\vol_{\ig_{\R^m}}.
	\]
\end{proof}

By Lemma \ref{expansion J functional} and Corollary \ref{detailed expansion J functional}, we find Theorem \ref{abstract result3} is capable of showing the existence issues here. Hence, in what follows, we will investigate the behavior of
\[
\hat\Phi(z):=\Phi(z)-\frac12\inp{K_z(\nabla\Ga(z))}{\nabla\Ga(z)}
\]
on a fiber $\vartheta^{-1}(\ga)$ for a fixed $\ga\in\cn=S^{2^{[\frac m2]+1}-1}(\mbs_m)$, where $K_z$ stands for the inverse of $\nabla^2 \cj_0(z)$ restricted to $\cw_z:=T_z\cm^\bot\subset \msd^{\frac12}(\R^m,\mbs(\R^m))$.
Note that $\vartheta^{-1}(\ga)$ is parameterized by $\cg=(0,\infty)\times\R^m$, it is very natural to study the values of $\hat\Phi(\psi_{\lm,\xi,\ga})$ when $\lm\to0$ and $\lm+|\xi|\to\infty$.

\begin{Lem}\label{lemma Phi1}
Assume that we are in the hypotheses of Lemma \ref{expansion J functional}. There holds
	\[
	\lim_{\lm\to0}	\Phi(\psi_{\lm,\xi,\ga})=C_0\big( \tr(\ih^2)-(\tr\ih)^2 \big)(\xi) \quad \text{for any }\xi\in\R^m,
	\]
	where
	\[
	C_0=\frac{m^{m-1}}{16}\int_{\R^m}\frac{1}{\big( 1+|x|^2 \big)^m}d\vol_{\ig_{\R^m}}.
	\]  
\end{Lem}
\begin{proof}
	By change os variables, one easily deduce that
	\[
	\Phi(\psi_{\lm,\xi,\ga})=\frac{m^{m-1}}{16}\int_{\R^m}\frac{\big( \tr(\ih^2)-(\tr\ih)^2 \big)(\lm x+\xi)}{\big( 1+|x|^2 \big)^m}d\vol_{\ig_{\R^m}}.
	\]
	And hence we have
	\[
	\lim_{\lm\to0}	\Phi(\psi_{\lm,\xi,\ga})=\frac{m^{m-1}}{16}\int_{\R^m}\frac{1}{\big( 1+|x|^2 \big)^m}d\vol_{\ig_{\R^m}}\big( \tr(\ih^2)-(\tr\ih)^2 \big)(\xi).
	\]
\end{proof}

The next two results, i.e. Lemma \ref{lemma Phi2} and Proposition \ref{behavior near 0}, are the main ingredients needed in our argument. For the sake of clarity, since the proofs of these two results contain the main difficulties of the paper, they are postponed to Appendix \ref{A-proofs}. 
\begin{Lem}\label{lemma Phi2}
Assume that we are in the hypotheses of Lemma \ref{expansion J functional}. There holds
	\[
	\lim_{\lm\to0}\inp{K_{\psi_{\lm,\xi,\ga}}(\nabla\Ga(\psi_{\lm,\xi,\ga}))}{\nabla\Ga(\psi_{\lm,\xi,\ga})}=C_1\big( \tr(\ih^2)-(\tr\ih)^2 \big)(\xi) \quad \text{for any }\xi\in\R^m,
	\]
	where
	\[
	C_1=\frac{m^{m-1}}{4}\int_{\R^m}\frac{|x|^2}{\big( 1+|x|^2 \big)^{m+1}}d\vol_{\ig_{\R^m}}.
	\]  
\end{Lem}

It follows from Lemma \ref{lemma Phi1} and \ref{lemma Phi2} that
\[
\aligned
\lim_{\lm\to0}\hat\Phi(\psi_{\lm,\xi,\ga})=\Big(C_0-\frac12 C_1\Big)\big(\tr(\ih^2)-(\tr\ih)^2 \big)(\xi)
=0.
\endaligned
\]
since $C_0=\frac12 C_1$. With this observation in mind, to see whether $\hat\Phi$ has local maximum or minimum on the fiber $\vartheta^{-1}(\ga)$, we have to go further to study the asymptotic behavior of $\hat\Phi(\psi_{\lm,\xi,\ga})$ as $\lm\to0$. By using the concept of elementary matrix in Definition \ref{def k-elementary}, the following proposition characterize the exact expansion of $\hat\Phi(\psi_{\lm,\xi,\ga})$ with respect to $\lm$ small.

\begin{Prop}\label{behavior near 0}
For $m\geq4$, assume that we are in the hypotheses of Lemma \ref{expansion J functional}. Let $k\in\{1,\dots,m\}$, $p\in[2,\infty]$ and $\ih=\diag(\ih_{11},\dots,\ih_{mm})$ be  $(k,p)$-elementary at a point $\xi\in\R^m$ with $\pa_k\ih_{ii}(\xi)\equiv c_k\neq0$, for $i\neq k$. If
\[
\begin{cases}
	p=\infty & m=4,\\
	p>2 & m=5,\\
	p\geq2 & m\geq6,
\end{cases}
\]
then 
	\[
	\hat\Phi(\psi_{\lm,\xi,\ga})=-\frac{3m^{m-2}(m-1)(m-2)c_k^2}{128} \lm^2\int_{\R^m}\frac{|x|^2}{(1+|x|^2)^m}d\vol_{\ig_{\R^m}}+o(\lm^2) \quad \text{as }\lm\to0.
	\]
	In particular, $\hat\Phi(\psi_{\lm,\xi,\ga})<0$ for small values of $\lm$.
\end{Prop}

The next proposition shows the decay of $\hat\Phi(\psi_{\lm,\xi,\ga})$ as $\lm+|\xi|\to\infty$.

\begin{Prop}\label{hat Phi decay}
	Assume that we are in the hypotheses of Lemma \ref{expansion J functional}. Then
	\[
	\hat\Phi(\psi_{\lm,\xi,\ga})\to0 \quad \text{as } \lm+|\xi|\to\infty.
	\]
\end{Prop}
\begin{proof}
	Let $r:\R^m\setminus\{0\}\to\R^m\setminus\{0\}$ be defined as $r(x)=-\frac{x}{|x|^2}$. Then, for each $(\lm,\xi,\ga)\in(0,\infty)\times\R^m\times S^{2^{[\frac m2]+1}-1}(\mbs_m)$, we can define $\psi_{\lm,\xi,\ga}^\sharp(x)=\psi_{\lm^\sharp,\xi^\sharp,\ga}(x)$, where $\lm^\sharp=\frac{\lm}{\lm^2+|\xi|^2}$ and $\xi^\sharp=\frac{-\xi}{\lm^2+|\xi|^2}$. Then $\psi_{\lm,\xi,\ga}^\sharp\in\cm$ and it follows from a direct computation that  that
	\[
	\big|\psi_{\lm,\xi,\ga}^\sharp(x)\big|_{\ig_{\R^m}}=\big|\psi_{\lm^\sharp,\xi^\sharp,\ga}(x)\big|_{\ig_{\R^m}}=\frac1{|x|^{m-1}}|\psi_{\lm,\xi,\ga}(r(x))|_{\ig_{\R^m}}.
	\]
	Roughly speaking, the map $\psi_{\lm,\xi,\ga}\mapsto \psi_{\lm,\xi,\ga}^\sharp$ can be interpreted as an one-to-one correspondence in $\cm$ between the elements at infinity and the elements close to $\psi_{t,0,\ga}$ with small $t>0$.
	
	For an arbitrary function $P:\R^m\to\R$, let us define $P^\sharp(x)=P(r(x))$. By observing that $r$ is conformal with conformal factor $\frac1{|x|^4}$, i.e. $r^*\ig_{\R^m}=\frac1{|x|^4}\ig_{\R^m}$, we find
	\begin{\equ}\label{I1}
		\aligned
		\int_{\R^m}P(x)|\psi_{\lm,\xi,\ga}(x)|_{\ig_{\R^m}}^{2^*}d\vol_{\ig_{\R^m}}&=\int_{\R^m}P(r(x))|\psi_{\lm,\xi,\ga}(r(x))|_{\ig_{\R^m}}^{2^*}\frac1{|x|^{2m}}d\vol_{\ig_{\R^m}}\\
		&=\int_{\R^m} P^\sharp(x)\big|\psi_{\lm,\xi,\ga}^\sharp(x)\big|_{\ig_{\R^m}}^{2^*}d\vol_{\ig_{\R^m}}.
		\endaligned
	\end{\equ}
	Moreover, by notice that
	\[
	\real\big( \pa_i\cdot_{\ig_{\R^m}}\nabla_{\pa_i}\psi_{\lm,\xi,\ga}^\sharp(x), \psi_{\lm,\xi,\ga}^\sharp(x) \big)_{\ig_{\R^m}}=\frac1m\big|\psi_{\lm,\xi,\ga}^\sharp(x)\big|_{\ig_{\R^m}}^{2^*}
	\]
	we have
	\begin{\equ}\label{I2}
		\aligned
		&\real\int_{\R^m}P(x)(\pa_i\cdot_{\ig_{\R^m}}\nabla_{\pa_i}\psi_{\lm,\xi,\ga},\psi_{\lm,\xi,\ga})_{\ig_{\R^m}}d\vol_{\ig_{\R^m}}\\
		&\qquad =\real\int_{\R^m}P^\sharp(x)\big(\pa_i\cdot_{\ig_{\R^m}}\nabla_{\pa_i}\psi_{\lm,\xi,\ga}^\sharp,\psi_{\lm,\xi,\ga}^\sharp\big)_{\ig_{\R^m}}d\vol_{\ig_{\R^m}}
		\endaligned
	\end{\equ}
	for $i=1,\dots,m$.
	
	Now let $\tilde\ig^\sharp(x):=\tilde\ig(r(x))$ and consider the corresponding new perturbed functional $\cj^\sharp$ obtained by substituting in \eqref{functional dimension m} $\tilde\ig$ with $\tilde\ig^\sharp$. Analogous to Lemma \ref{expansion J functional}, we have an expansion
	\[
	\cj^\sharp(\tilde\psi)=\cj_0(\psi)+\eps \Ga^\sharp(\psi)+\eps^2\Phi^\sharp(\psi)+o(\eps^2).
	\]
	And by \eqref{I1} and \eqref{I2}, we infer that
	\[
	\Ga^\sharp(\psi_{\lm,\xi,\ga}^\sharp)=\Ga(\psi_{\lm,\xi,\ga})
	\quad \text{and} \quad
	\Phi^\sharp(\psi_{\lm,\xi,\ga}^\sharp)=\Phi(\psi_{\lm,\xi,\ga}).
	\]
	Furthermore, we also have
	\[
	\hat\Phi^\sharp(\psi_{\lm,\xi,\ga}^\sharp)=\hat\Phi(\psi_{\lm,\xi,\ga}).
	\]
	Hence, by applying Lemma \ref{lemma Phi1} and \ref{lemma Phi2}, we have
	\[
	\lim_{\lm+|\xi|\to\infty}\hat\Phi(\psi_{\lm,\xi,\ga})=\lim_{\lm+|\xi|\to\infty}\hat\Phi^\sharp(\psi_{\lm,\xi,\ga}^\sharp)
	=\lim_{\substack{\lm^\sharp\to0 \\ \xi^\sharp\to0}}\hat\Phi^\sharp(\psi_{\lm^\sharp,\xi^\sharp,\ga})=0,
	\]
	which completes the proof.
\end{proof}

\begin{Rem}
	From Proposition \ref{behavior near 0}, \ref{hat Phi decay} and Theorem \ref{abstract result3} we can immediately conclude that: there exists a bounded open set $U\subset (0,\infty)\times\R^m$, independent of $\ga\in S^{2^{[\frac m2]+1}-1}(\mbs_m)$, such that $\hat\Phi(\psi_{\lm,\xi,\ga})$ achieves  a strict minimum in dimension $m\geq4$ in $U$, and hence the perturbed spinorial Yamabe equation \eqref{Dirac euclidean m} has a solution.
\end{Rem}

\subsection{Proof the main result}
The proof will be carried out via several steps, we begin with some basic estimates.

\begin{Lem}\label{Kz bdd}
The operator $K_z=\nabla^2\cj_0(z)^{-1}$ is uniformly bounded for $z\in\cm$.
\end{Lem}
\begin{proof}
For a fixed $\ga\in S^{2^{[\frac m2]+1}-1}(\mbs_m)$, we denote $z_0=\psi_{1,0,\ga}$. For arbitrary $g=(\lm,\xi)\in (0,\infty)\times\R^m$, we define a right action on elements in $\msd^{1/2}(\R^m,\mbs(\R^m))$ as
\[
(\psi\cdot g)(x)=\lm^{-\frac{m-1}2}\psi\Big(\frac{x-\xi}\lm\Big) \quad
\text{for } \psi\in \msd^{1/2}(\R^m,\mbs(\R^m)).
\]
Then it is easy to check that $z_0\cdot g=\psi_{\lm,\xi,\ga}$ and the inverse action is defined by $g^{-1}=(\lm^{-1},-\lm^{-1}\xi)$. Moreover, by standard Fourier transformation, we can see that the action of $g$ is an isometry, i.e. $\|\psi\cdot g\|=\|\psi\|$ for any $\psi\in\msd^{1/2}(\R^m,\mbs(\R^m))$ and $g\in(0,\infty)\times\R^m$.

Observe that $\nabla^2\cj_0(z_0)$ is self-adjoint and invertible on $\cw_{z_0}=T_{z_0}\cm^\bot$, there exists a constant $c_\ga>0$ such that for any $w\in \cw_{z_0}$ there is $v(w)\in \cw_{z_0}$ satisfying $\|v(w)\|=1$ and
\[
\nabla^2\cj_0(z_0)[w,v(w)]=\big\|\nabla^2\cj_0(z_0)[w]\big\|\geq c_\ga\|w\|.
\]
Here the subscript in the constant $c_\ga$ comes from the fact that $z_0$ is (implicitly) depending on $\ga$. In what follows, for arbitrary $w\in\cw_{z_0}$, we simply denote $w_g=w\cdot g$ for $g=(\lm,\xi)\in(0,\infty)\times\R^m$ and $z_g :=z_0\cdot g$. We mention that, by change of variables, 
\[
\real\int_{\R^m}(D_{\ig_{\R^m}}w_g,v(w)_g)_{\ig_{\R^m}}d\vol_{\ig_{\R^m}}=\real\int_{\R^m}(D_{\ig_{\R^m}}w,v(w))_{\ig_{\R^m}}d\vol_{\ig_{\R^m}},
\]
\[
\real\int_{\R^m}|z_g|^{\frac2{m-1}}(w_g,v(w)_g)_{\ig_{\R^m}}d\vol_{\ig_{\R^m}}=\real\int_{\R^m}|z_0|^{\frac2{m-1}}(w,v(w))_{\ig_{\R^m}}d\vol_{\ig_{\R^m}}
\]
and
\[
\aligned
&\int_{\R^m}|z_g|^{\frac{4-2m}{m-1}}\real(z_g,w_g)_{\ig_{\R^m}}\real(z_g, v(w)_g)_{\ig_{\R^m}}d\vol_{\ig_{\R^m}} \\
&\qquad =\real\int_{\R^m}|z_0|^{\frac{4-2m}{m-1}}\real(z_0,w)_{\ig_{\R^m}}\real(z_0, v(w))_{\ig_{\R^m}}d\vol_{\ig_{\R^m}}
\endaligned
\]
Hence one deduces 
\[
\aligned
\nabla^2\cj_0(z_g)[w,v(w_{g^{-1}})_g]=\nabla^2\cj_0(z_0)[w_{g^{-1}},v(w_{g^{-1}})] 
\geq c_\ga\|w_{g^{-1}}\|=c_\ga\|w\|
\endaligned
\]
Therefore, we see that the operator norm of $\|K_{z_\rho}\|$ is uniformly bounded by $\frac1{c_\ga}$ for all $g\in(0,\infty)\times\R^m$ since $K_{z_g}=\nabla^2\cj_0(z_g)^{-1}$. Note that $\ga\in S^{2^{[\frac m2]+1}-1}(\mbs_m)$, where $S^{2^{[\frac m2]+1}-1}(\mbs_m)$ is a compact manifold, we have $c_\ga$ is bounded away from $0$ uniformly in $\ga$. And thanks to the continuity of $K_{z_0}$ with respect to $\ga$, we conclude that $K_z$ is uniformly bounded for all $z\in\cm$ and this completes the proof.
\end{proof}

\begin{Lem}\label{nabla Ga decay}
Let $\tilde\ig$ be given by \eqref{metric form}. Assume that the matrix function $\ih$ is compactly supported. Then
\[
\|\nabla\Ga(\psi_{\lm,\xi,\ga})\|\to0 \quad \text{as } \lm\to\infty 
\text{ unifromly in } (\xi,\ga).
\]
\end{Lem}
\begin{proof}
Let $R>0$ be large enough such that $\supp\ih\subset B_R$, where $B_R$ is the open ball in $\R^m$ of radius $R$ centered at the origin. For any $v\in \msd^{1/2}(\R^m,\mbs(\R^m))$, by \eqref{nabla Ga}, we have
\[
\aligned
|\inp{\nabla\Ga(\psi_{\lm,\xi,\ga})}{v}|&\leq C \Big( \int_{B_R}|\nabla\psi_{\lm,\xi,\ga} |\cdot|v| +|\psi_{\lm,\xi,\ga}|\cdot|v|d\vol_{\ig_{\R^m}} \Big) \\
&\leq C\Big( \int_{B_R}\frac{\lm^{\frac{m-1}2}}{\big(\lm^2+|x-\xi|^2\big)^{\frac m2}}\cdot|v| +\frac{\lm^{\frac{m-1}2}}{\big(\lm^2+|x-\xi|^2\big)^{\frac{m-1}2}}\cdot|v|d\vol_{\ig_{\R^m}} \Big) \\
&\leq C\big(\lm^{-\frac{m+1}2}+\lm^{-\frac{m-1}2}\big)\|v\|.
\endaligned
\]
Thus, $\|\nabla\Ga(\psi_{\lm,\xi,\ga})\|\to0$ as  $\lm\to\infty$ 
unifromly in $(\xi,\ga)$.
\end{proof}

Next, let us consider a further situation of the perturbed metric in \eqref{metric form}. Precisely, we consider the matrix $\ih$ in the form of \eqref{h splits}. In this case, the matrix $\ih$ can be decomposed as $\ih(x)=\ih^{(1)}(x)+\ih^{(2)}(x-x_0)$, and it is clear that $\ih^{(1)}(\cdot)$ and $\ih^{(2)}(\cdot-x_0)$ have disjoint supports provided that $|x_0|$ is sufficiently large. In this setting, by substituting in \eqref{functional dimension m}  with the metrics $\tilde\ig^{(1)}=\ig_{\R^m}+\eps\ih^{(1)}$ and
$\tilde\ig^{(2)}(\cdot-x_0)=\ig_{\R^m}+\eps\ih^{(2)}(\cdot-x_0)$, we obtain $\Ga^{(1)}$, $\Ga^{(2),\, x_0}$, $\Phi^{(1)}$, $\Phi^{(2),\, x_0}$, etc., the corresponding functionals. It is easy to check that
\[
\Ga^{(2),\, x_0}(\psi_{\lm,\xi,\ga})=\Ga^{(2)}(\psi_{\lm,\xi-x_0,\ga}) 
\quad \text{and} \quad
\Phi^{(2),\, x_0}(\psi_{\lm,\xi,\ga})=\Phi^{(2)}(\psi_{\lm,\xi-x_0,\ga}).
\]
Consequently, we have
\[
\hat\Phi^{(2),\, x_0}(\psi_{\lm,\xi,\ga})=\hat\Phi^{(2)}(\psi_{\lm,\xi-x_0,\ga}).
\]
If $|x_0|$ is large enough such that $\ih^{(1)}$ and $\ih^{(2)}(\cdot-x_0)$ have disjoint supports, we find
\[
\Ga(\psi_{\lm,\xi,\ga})=\Ga^{(1)}(\psi_{\lm,\xi,\ga})+\Ga^{(2),\, x_0}(\psi_{\lm,\xi,\ga})
=\Ga^{(1)}(\psi_{\lm,\xi,\ga})+\Ga^{(2)}(\psi_{\lm,\xi-x_0,\ga}),
\]
\[
\Phi(\psi_{\lm,\xi,\ga})
=\Phi^{(1)}(\psi_{\lm,\xi,\ga})+\Phi^{(2)}(\psi_{\lm,\xi-x_0,\ga}),
\]
and 
\[
\nabla\Ga(\psi_{\lm,\xi,\ga})=\nabla\Ga^{(1)}(\psi_{\lm,\xi,\ga})+\nabla\Ga^{(2),\, x_0}(\psi_{\lm,\xi,\ga}).
\]
In order to get a similar decomposition for $\hat\Phi$, we need the following lemma.

\begin{Lem}\label{multi nabla Ga}
Given $\La>0$ arbitrarily, then
\[
\|\nabla\Ga^{(1)}(\psi_{\lm,\xi,\ga})\|\cdot\|\nabla\Ga^{(2),\,x_0}(\psi_{\lm,\xi,\ga})\|\to0 \quad \text{as } |x_0|\to\infty 
\]
uniformly in $(\lm,\xi,\ga)\in(0,\La]\times\R^m\times S^{2^{[\frac m2]+1}-1}(\mbs_m)$.
\end{Lem}
\begin{proof}
Similar to Lemma \ref{Kz bdd}, we first fix $R>0$ large such that $\supp\ih^{(1)}\cup\supp\ih^{(2)}\subset B_R$. Then, for arbitrary $v\in \msd^{1/2}(\R^m,\mbs(\R^m))$,
\[
|\inp{\nabla\Ga^{(1)}(\psi_{\lm,\xi,\ga})}{v}|\leq C^{(1)} \Big( \int_{B_R}\frac{\lm^{\frac{m-1}2}}{\big(\lm^2+|x-\xi|^2\big)^{\frac m2}}\cdot|v| +\frac{\lm^{\frac{m-1}2}}{\big(\lm^2+|x-\xi|^2\big)^{\frac{m-1}2}}\cdot|v|d\vol_{\ig_{\R^m}}  \Big)
\]
where $C^{(1)}>0$ depends on $\|\ih^{(1)}\|_{L^\infty}$ and $\|\nabla\ih^{(1)}\|_{L^\infty}$. Hence, if $|\xi|>R$, we see that
\[
\|\nabla\Ga^{(1)}(\psi_{\lm,\xi,\ga})\|\leq C^{(1)}\Big(  \frac{\La^{\frac{m-1}2}}{\dist(\xi,B_R)^m}+\frac{\La^{\frac{m-1}2}}{\dist(\xi,B_R)^{m-1}}\Big)
\]
for all $\lm\in(0,\La]$ and $\ga\in S^{2^{[\frac m2]+1}-1}(\mbs_m)$. And similarly, for some $C^{(2)}>0$, if $|\xi-x_0|>R$ we have
\[
	\|\nabla\Ga^{(2),\,x_0}(\psi_{\lm,\xi,\ga})\|=\|\nabla\Ga^{(2)}(\psi_{\lm,\xi-\xi_0,\ga})\|\leq C^{(2)}\Big(  \frac{\La^{\frac{m-1}2}}{\dist(\xi-x_0,B_R)^m}+\frac{\La^{\frac{m-1}2}}{\dist(\xi-x_0,B_R)^{m-1}}\Big)  
\]
for all $\lm\in(0,\La]$ and $\ga\in S^{2^{[\frac m2]+1}-1}(\mbs_m)$.

Observe that
\[
\dist(\xi,B_R)+\min\{|\xi|,\,R\}=|\xi|
\quad \text{and} \quad
\dist(\xi-x_0,B_R)+\min\{|\xi-x_0|,\,R\}=|\xi-x_0|.
\]
It follows directly from the triangle inequality that, if $|x_0|$ is large (say $|x_0|>4R$), it is always $\dist(\xi,B_R)\geq\frac14|x_0|$ or $\dist(\xi-x_0,B_R)\geq\frac14|x_0|$, hence 
\[
\|\nabla\Ga^{(1)}(\psi_{\lm,\xi,\ga})\|\cdot\|\nabla\Ga^{(2),\,x_0}(\psi_{\lm,\xi,\ga})\|\leq \frac{C\La^{\frac{m-1}2}}{|x_0|^{m-1}}\to0 
\]
as $|x_0|\to\infty$ uniformly in $(\lm,\xi,\ga)\in(0,\La]\times\R^m\times S^{2^{[\frac m2]+1}-1}(\mbs_m)$.
\end{proof}

An immediate consequence of Lemma \ref{Kz bdd}-\ref{multi nabla Ga}, we have the following decomposition of $\hat\Phi(\psi_{\lm,\xi,\ga})$.
\begin{Cor}\label{hat Phi decompsition}
There holds
\[ 
\inp{K_{\psi_{\lm,\xi,\ga}}(\nabla\Ga^{(1)}(\psi_{\lm,\xi,\ga}))}{\nabla\Ga^{(2),\, x_0}(\psi_{\lm,\xi,\ga})}\to0   
\]
as $|x_0|\to\infty$, uniformly in $(\lm,\xi,\ga)\in(0,\infty)\times\R^m\times S^{2^{[\frac m2]+1}-1}(\mbs_m)$. In particular,
\[
\hat\Phi(\psi_{\lm,\xi,\ga})
=\hat\Phi^{(1)}(\psi_{\lm,\xi,\ga})+\hat\Phi^{(2)}(\psi_{\lm,\xi-x_0,\ga})+o(1),
\]
where $o(1)\to0$ as $|x_0|\to\infty$, uniformly in $(\lm,\xi,\ga)\in(0,\infty)\times\R^m\times S^{2^{[\frac m2]+1}-1}(\mbs_m)$.
\end{Cor}

Now we are ready to prove our main result.

\begin{proof}[Proof of Theorem \ref{main theorem 2}] 
For $m\geq4$, from Proposition \ref{behavior near 0} and \ref{hat Phi decay}, we can choose
\[
U^{(1)}_R:=\{(\lm,\xi)\in (0,\infty)\times\R^2:\, \lm+|\xi|<R\}
\] 
and
\[
U^{(2)}_R:=\{(\lm,\xi)\in (0,\infty)\times\R^2:\, \lm+|\xi-x_0|<R\}
\]
with $R>0$ suitably large such that, for all $\ga\in S^{2^{[\frac m2]+1}-1}(\mbs_m)$,
\[
\min_{(\lm,\xi)\in \pa U^{(1)}_R}\hat\Phi^{(1)}(\psi_{\lm,\xi,\ga})-\min_{(\lm,\xi)\in  U^{(1)}_R}\hat\Phi^{(1)}(\psi_{\lm,\xi,\ga})\geq\frac1R
\]
and 
\[
\min_{(\lm,\xi)\in \pa U^{(2)}_R}\hat\Phi^{(2)}(\psi_{\lm,\xi-x_0,\ga})-\min_{(\lm,\xi)\in  U^{(2)}_R}\hat\Phi^{(2)}(\psi_{\lm,\xi-x_0,\ga})\geq\frac1R
\]
Hence, by Corollary \ref{hat Phi decompsition}, we find that for $|x_0|\gg R$ there exists $\de>0$ such that $\{(\lm,\xi):\hat\Phi(\psi_{\lm,\xi,\ga})<-\de\}$ has two disconnected components in $U^{(1)}_R$ and $U^{(2)}_R$ respectively. And therefore, by applying Theorem \ref{abstract result3}, the two distinct local minima of $\hat\Phi(\psi_{\lm,\xi,\ga})$ give rise to two distinct solutions of Eq. \eqref{Dirac euclidean m} (and hence two distinct solutions of Eq. \eqref{Dirac dimension m} via the pull-back of stereographic projection) for small values of $\eps$. 
\end{proof}

\appendix
\section{Appendix}

\subsection{An explanation of the $(k,p)$-elementary matrix}\label{A-kpmatrix}

Here we point out that linear combinations of the $(k,p)$-elementary matrices, $k=1,\dots,m$, create a very large family of diagonal matrices. To have a better look at this, let's restrict ourselves in the $3$-dimensional space (it will cause no confusion if we are in any other dimensions). First of all, let $\ih^{(k)}$, $k=1,2,3$, be  $(k,p)$-elementary at a common point $\xi=(\xi_1,\xi_2,\xi_3)\in\R^3$ with correspondingly $c_k\neq0$ being the non-zero coefficients, i.e., we have the local expressions
\[
\ih^{(1)}(x)=\begin{pmatrix}
	0 & 0 & 0 \\
	0 & \ih^{(1)}_{22}(x) & 0 \\
	0 & 0 & \ih_{33}^{(1)}(x)
\end{pmatrix}, \quad 
\ih^{(2)}(x)=\begin{pmatrix}
	\ih_{11}^{(2)}(x) & 0 & 0 \\
	0 & 0 & 0 \\
	0 & 0 & \ih_{33}^{(2)}(x)
\end{pmatrix} 
\]
and
\[
\ih^{(3)}(x)=\begin{pmatrix}
	\ih_{11}^{(3)}(x) & 0 & 0 \\
	0 & \ih_{22}^{(3)}(x) & 0 \\
	0 & 0 & 0
\end{pmatrix} 
\]
with
\[
\left.
\aligned
&\ih^{(1)}_{22}(x)=a_{22}^{(1)}+c_1(x_1-\xi_1)+c_2^{(1)}(x_2-\xi_2)+o(|x-\xi|^p)  \\[0.3em]
&\ih^{(1)}_{33}(x)=a_{33}^{(1)}+c_1(x_1-\xi_1)+c_3^{(1)}(x_3-\xi_3)+o(|x-\xi|^p)
\endaligned
\right\} \text{ for } \ih^{(1)},
\]
\[
\left.
\aligned
&\ih^{(2)}_{11}(x)=a_{11}^{(2)}+c_1^{(2)}(x_1-\xi_1)+c_2(x_2-\xi_2)+o(|x-\xi|^p)  \\[0.3em]
&\ih^{(2)}_{33}(x)=a_{33}^{(2)}+c_2(x_2-\xi_2)+c_3^{(2)}(x_3-\xi_3)+o(|x-\xi|^p)
\endaligned
\right\} \text{ for } \ih^{(2)},
\]
\[
\left.
\aligned
&\ih^{(3)}_{11}(x)=a_{11}^{(3)}+c_1^{(3)}(x_1-\xi_1)+c_3(x_3-\xi_3)+o(|x-\xi|^p)  \\[0.3em]
&\ih^{(3)}_{22}(x)=a_{22}^{(3)}+c_2^{(3)}(x_2-\xi_2)+c_3(x_3-\xi_3)+o(|x-\xi|^p)
\endaligned
\right\} \text{ for } \ih^{(3)},
\]
where $a_{ii}^{(k)}$ and $c_i^{(k)}$, $i,k=1,2,3$, are some given real constants. Then, a linear combination of $\ih^{(k)}$, say for instance $\al\ih^{(1)}+\bt\ih^{(2)}+\ga\ih^{(3)}$ for $\al,\bt,\ga\in\R$, gives rise to a matrix of the form (in local expansion around the point $\xi$)
\begin{\equ}\label{generated matrix}
	\begin{pmatrix}
		a_{11}+\boldsymbol{b}_1\cdot(x-\xi) & 0 &0 \\
		0 & a_{22}+\boldsymbol{b}_2\cdot(x-\xi)&0 \\
		0& 0& a_{33}+\boldsymbol{b}_3\cdot(x-\xi)
	\end{pmatrix}+o(|x-\xi|^p)
\end{\equ}
with $a_{11}=\bt a_{11}^{(2)}+\ga a_{11}^{(3)}$, $a_{22}=\al a_{22}^{(1)}+\ga a_{22}^{(3)}$, $a_{33}=\al a_{33}^{(1)}+\bt a_{33}^{(2)}$ characterizing the zero-order terms and the vector-valued coefficients
\[
\boldsymbol{b}_1=(\bt c_1^{(2)}+\ga c_1^{(3)},\bt c_2,\ga c_3), \quad \boldsymbol{b}_2=(\al c_1,\al c_2^{(1)}+\ga c_2^{(3)},\ga c_3), \quad 
\boldsymbol{b}_3=(\al c_1,\bt c_2, \al c_3^{(1)}+\bt c_3^{(2)})
\]
characterizing the first-order terms. Such specific local expansion is satisfied by a large class of diagonal matrix functions. Moreover, notice that \eqref{generated matrix} is defined only around  one spatial point, we can consider such kind of expressions around an arbitrary finite number of spatial points. And even more, there is no restriction on the global behavior of $\ih^{(k)}$, $k=1,2,3$. This is why we use the word ``{\it elementary}" in Definition \ref{def k-elementary}. 

\subsection{Some basic estimates}

In the sequel we use the notation $f\lesssim g$ for two functions $f$ and $g$, when there exists a constant $C>0$ such that $f\leq C g$. Then, by the explicit expression of $\psi_{1,0,\ga}$ in \eqref{critical manifold explicit}, we have
\[
|\psi_{1,0,\ga}(x)|\lesssim \frac1{(1+|x|^2)^{\frac{m-1}2}} \quad \text{and} \quad 
|\nabla\psi_{1,0,\ga}(x)|\lesssim \frac1{(1+|x|^2)^{\frac{m}2}} .
\]

Given $R>\de>0$ and let $\lm\searrow0$ be a parameter, for arbitrary $\va\in \msd^{\frac12}(\R^m,\mbs(\R^m))$ with $\|\va\|=1$, we estimate: 
\begin{eqnarray}\label{E1}
	\Big| \real\int_{|x|\geq\frac\de\lm}(\pa_i\cdot_{\ig_{\R^m}}\nabla_{\pa_i}\psi_{1,0,\ga},\va)_{\ig_{\R^m}}d\vol_{\ig_{\R^m}} \Big| &\leq &\Big( \int_{|x|>\frac\de\lm}|\nabla_{\pa_i}\psi_{1,0,\ga}|^{\frac{2m}{m+1}}d\vol_{\ig_{\R^m}} \Big)^{\frac{m+1}{2m}}|\va|_{\frac{2m}{m-1}}  \nonumber \\[0.5em] 
	&\lesssim&\Big(\int_{\frac\de\lm}^\infty \frac{r^{m-1}}{(1+r^2)^{\frac{m^2}{m+1}}}dr \Big)^{\frac{m+1}{2m}}  \nonumber \\[0.5em]
	&\lesssim& \lm^{\frac{m-1}2}  \qquad \text{for } m\geq2,
\end{eqnarray} 
\begin{eqnarray}\label{E2}
	\Big| \real\int_{|x|\leq\frac\de\lm}x_j(\pa_i\cdot_{\ig_{\R^m}}\nabla_{\pa_i}\psi_{1,0,\ga},\va)_{\ig_{\R^m}}d\vol_{\ig_{\R^m}} \Big| 
	&\lesssim&\Big( \int_0^{\frac\de\lm}\frac{r^{\frac{2m}{m+1}+m-1}}{(1+r^2)^{\frac{m^2}{m+1}}}dr \Big)^{\frac{m+1}{2m}} \nonumber \\[0.5em]
	&=&\Big(\int_0^1+ \int_1^{\frac\de\lm}\frac{r^{\frac{2m}{m+1}+m-1}}{(1+r^2)^{\frac{m^2}{m+1}}}dr \Big)^{\frac{m+1}{2m}}  \nonumber \\[0.5em]
	&\lesssim&\begin{cases}
		\lm^{-\frac12} &  m=2, \\
		|\ln\lm|^{\frac23}  & m=3, \\
		1 &  m\geq4,
	\end{cases} 
\end{eqnarray}
\begin{eqnarray}\label{E3}
	\Big| \real\int_{\frac\de\lm\leq|x|\leq\frac R\lm}x_j(\pa_i\cdot_{\ig_{\R^m}}\nabla_{\pa_i}\psi_{1,0,\ga},\va)d\vol_{\ig_{\R^m}} \Big| &\lesssim& \Big( \int_{\frac\de\lm}^{\frac R\lm} \frac{r^{\frac{2m}{m+1}+m-1}}{(1+r^2)^{\frac{m^2}{m+1}}} dr \Big)^{\frac{m+1}{2m}} \nonumber \\[0.5em]
	&\lesssim&\begin{cases}
		\lm^{-\frac12} & m=2,\\
		1 & m=3, \\
		\lm^{\frac{m-3}2} &  m\geq4,
	\end{cases}
\end{eqnarray}
\begin{eqnarray}\label{E4}
	\Big| \real\int_{|x|\leq\frac\de\lm}x_jx_k(\pa_i\cdot_{\ig_{\R^m}}\nabla_{\pa_i}\psi_{1,0,\ga},\va)d\vol_{\ig_{\R^m}} \Big|
	&\lesssim &\Big( \int_0^{\frac\de\lm} \frac{r^{\frac{4m}{m+1}+m-1}}{(1+r^2)^{\frac{m^2}{m+1}}}dr \Big)^{\frac{m+1}{2m}}  \nonumber\\[0.5em]
	&= &\Big( \int_0^1+\int_1^{\frac\de\lm} \frac{r^{\frac{4m}{m+1}+m-1}}{(1+r^2)^{\frac{m^2}{m+1}}}dr \Big)^{\frac{m+1}{2m}}  \nonumber\\[0.5em]
	&\lesssim&\begin{cases}
		\lm^{\frac{m-5}2}  & 2\leq m\leq 4, \\
		|\ln\lm|^{\frac35} & m=5, \\
		1  & m\geq 6.
	\end{cases}
\end{eqnarray}
\begin{eqnarray}\label{E5}
	\Big|  \real\int_{|x|\leq\frac\de\lm}(\pa_k\cdot_{\ig_{\R^m}}\psi_{1,0,\ga},\va)d\vol_{\ig_{\R^m}}\Big|  
	&\lesssim&\Big(  \int_0^{\frac\de\lm}\frac{r^{m-1}}{(1+r^2)^{\frac{m(m-1)}{m+1}}}dr\Big)^{\frac{m+1}{2m}} \nonumber\\[0.5em]   
	&=& \Big( \int_0^1+ \int_1^{\frac\de\lm}\frac{r^{m-1}}{(1+r^2)^{\frac{m(m-1)}{m+1}}}dr\Big)^{\frac{m+1}{2m}} \nonumber\\[0.5em]   
	&\lesssim& \begin{cases}
		\lm^{-\frac12}  & m=2, \\
		|\ln\lm|^{\frac23} & m=3,\\
		1  & m\geq4,
	\end{cases}
\end{eqnarray}
\begin{eqnarray}\label{E6}
	\Big|  \real\int_{\frac\de\lm<|x|\leq \frac R\lm}(\pa_k\cdot_{\ig_{\R^m}}\psi_{1,0,\ga},\va)d\vol_{\ig_{\R^m}}\Big| 
	&\lesssim &\Big( \int_{\frac\de\lm}^{\frac R\lm}\frac{r^{m-1}}{(1+r^2)^{\frac{m(m-1)}{m+1}}}dr \Big)^{\frac{m+1}{2m}}\nonumber\\[0.5em] 
	&\lesssim& \begin{cases}
		\lm^{-\frac12}  &m=2, \\
		1  & m=3,\\
		\lm^{\frac{m-3}2} & m\geq4,
	\end{cases}
\end{eqnarray}
\begin{eqnarray}\label{E7}
	\Big|  \real\int_{|x|\leq\frac\de\lm}x_j(\pa_k\cdot_{\ig_{\R^m}}\psi_{1,0,\ga},\va)d\vol_{\ig_{\R^m}}\Big| 
	&\lesssim& \Big( \int_0^{\frac\de\lm} \frac{r^{\frac{2m}{m+1}+m-1}}{(1+r^2)^{\frac{m(m-1)}{m+1}}} dr \Big)^{\frac{m+1}{2m}} \nonumber \\[0.5em] 
	&=&\Big( \int_0^1+\int_1^{\frac\de\lm} \frac{r^{\frac{2m}{m+1}+m-1}}{(1+r^2)^{\frac{m(m-1)}{m+1}}} dr \Big)^{\frac{m+1}{2m}} \nonumber\\[0.5em] 
	&\lesssim& \begin{cases}
		\lm^{\frac{m-5}2} & 2\leq m\leq 4,\\
		|\ln\lm|^{\frac35}  & m=5,\\
		1  & m\geq6.
	\end{cases}
\end{eqnarray}

In particular, for $m=5$, we have $\frac{2m}{m+1}=\frac53$ and we estimate: for $p>2$ fixed,
\begin{\equ}\label{E8}
	\Big( \int_{|x|\leq\frac\de\lm} |x|^{\frac{5p}{3}}|\nabla_{\pa_i}\psi_{1,0,\ga}|^{\frac{5}{3}}d\vol_{\ig_{\R^m}} \Big)^{\frac{3}{5}} \lesssim
	\Big( \int_0^{\frac\de\lm}\frac{r^{\frac{5p}{3}+4}}{(1+r^2)^{\frac{25}{6}}}dr\Big)^{\frac{3}{5}}  \lesssim  \lm^{2-p}
\end{\equ}
and
\begin{\equ}\label{E9}
	\Big( \int_{|x|\leq\frac\de\lm} |x|^{\frac{5(p-1)}{3}}|\psi_{1,0,\ga}|^{\frac{5}{3}}d\vol_{\ig_{\R^m}} \Big)^{\frac{3}{5}} \lesssim 
	\Big( \int_0^{\frac\de\lm}\frac{r^{\frac{5(p-1)}{3}+4}}{(1+r^2)^{\frac{10}{3}}}dr\Big)^{\frac{3}{5}}  \lesssim  \lm^{2-p}.
\end{\equ}

\subsection{Proofs of Lemma \ref{lemma Phi2} and Proposition \ref{behavior near 0}}\label{A-proofs}
\begin{proof}[Proof of Lemma \ref{lemma Phi2}]
	Let $z\in\cm$, note that $K_z$ stands for the inverse of $\nabla^2 \cj_0(z)$ on $\cw_z=T_z\cm^\bot$, by setting $\bar w_z=-K_z(\nabla\Ga(z))$, we see that
	\begin{\equ}\label{TzGaz}
		\inp{K_z(\nabla\Ga(z))}{\nabla\Ga(z)}=-\inp{\nabla\Ga(z)}{\bar w_z}=\nabla^2\cj_0(z)[\bar w_z,\bar w_z].
	\end{\equ}
	
	Thanks to the expression of $\Ga$ in Lemma \ref{expansion J functional}
	and the fact $\nabla\Ga(z)\in \cw_z$, we obtain (integrating by parts)
	\begin{eqnarray}\label{nabla Ga}
		\inp{\nabla\Ga(z)}{v}&=&\real\int_{\R^m}\frac{\tr\ih}2\Big[ \frac12\big[(D_{\ig_{\R^m}}z,v)_{\ig_{\R^m}}+(z,D_{\ig_{\R^m}}v)_{\ig_{\R^m}}\big]-|z|^{2^*-2}_{\ig_{\R^m}}(z,v)_{\ig_{\R^m}} \Big] \nonumber \\[0.5em]
		& & -\frac14\sum_i\ih_{ii}\Big( (\pa_i\cdot_{\ig_{\R^m}}\nabla_{\pa_i} z,v)_{\ig_{\R^m}} +(\pa_i\cdot_{\ig_{\R^m}}\nabla_{\pa_i}v,z)_{\ig_{\R^m}}\Big) d\vol_{\ig_{\R^m}} \nonumber \\[0.5em]
		&=&\real\int_{\R^m}\frac{\tr\ih}2\Big[ (D_{\ig_{\R^m}}z,v)_{\ig_{\R^m}}-|z|^{2^*-2}_{\ig_{\R^m}}(z,v)_{\ig_{\R^m}} \Big]d\vol_{\ig_{\R^m}} \nonumber \\[0.5em]
		& & +\frac14\real\int_{\R^m}(\nabla(\tr\ih)\cdot_{\ig_{\R^m}}z,v)_{\ig_{\R^m}}d\vol_{\ig_{\R^m}} \nonumber \\[0.5em]
		& & -\frac12\sum_i\real\int_{\R^m}\ih_{ii}(\pa_i\cdot_{\ig_{\R^m}}\nabla_{\pa_i} z,v)_{\ig_{\R^m}}d\vol_{\ig_{\R^m}}  \nonumber \\[0.5em]
		& &
		-\frac14\sum_i\real\int_{\R^m}\pa_i\ih_{ii}(\pa_i\cdot_{\ig_{\R^m}}z,v)_{\ig_{\R^m}}d\vol_{\ig_{\R^m}}
	\end{eqnarray}
	for all $v\in \cw_z$, where we have used the facts
	\[
	\aligned
	\real\int_{\R^m}(\tr\ih)(z, D_{\ig_{\R^m}}v)_{\ig_{\R^m}}d\vol_{\ig_{\R^m}}
	&=\real\int_{\R^m}(\tr\ih)(D_{\ig_{\R^m}}z, v)_{\ig_{\R^m}}d\vol_{\ig_{\R^m}}  \\[0.5em]
	&\quad
	+\real\int_{\R^m}(\nabla(\tr\ih)\cdot_{\ig_{\R^m}}z,v)_{\ig_{\R^m}}d\vol_{\ig_{\R^m}}
	\endaligned
	\]
	and
	\[
	\aligned
	\pa_i\real(\pa_i\cdot_{\ig_{\R^m}}z,v)
	=\real(\pa_i\cdot_{\ig_{\R^m}}\nabla_{\pa_i}z, v)-\real(\pa_i\cdot_{\ig_{\R^m}}\nabla_{\pa_i} v,z).
	\endaligned
	\]
	Note that $z\in\cm$ satisfies the equation
	\[
	D_{\ig_{\R^m}}z=|z|^{2^*-2}_{\ig_{\R^m}}z \quad \text{on } \R^m,
	\]
	we have
	\[
	\nabla\Ga(z)=-\frac12\sum_i\ih_{ii}\pa_i\cdot_{\ig_{\R^m}}\nabla_{\pa_i} z+\frac14\nabla(\tr\ih)\cdot_{\ig_{\R^m}}z-\frac14\sum_i\pa_i\ih_{ii}\pa_i\cdot_{\ig_{\R^m}}z
	\]
	for all $z\in\cm$. Moreover, we have $\bar w_z$ solves the equation
	\begin{\equ}\label{X1}
		\nabla^2\cj_0(z)[\bar w_z]=-\nabla\Ga(z).
	\end{\equ}
	where
	\[
	\nabla^2\cj_0(z)[\bar w_z]=D_{\ig_{\R^m}}\bar w_z-|z|^{2^*-2}_{\ig_{\R^m}}\bar w_z-(2^*-2)|z|^{2^*-4}_{\ig_{\R^m}}\real(z,\bar w_z)z.
	\]
	
	Next, we will explicitly substitute $z=\psi_{\lm,\xi,\ga}$ into \eqref{X1}. For this purpose, let us set $\bar w_{\lm,\xi,\ga}=\bar w_{\psi_{\lm,\xi,\ga}}$ and $w_{\lm,\xi,\ga}^*(x)=\lm^{\frac{m-1}2}\bar w_{\lm,\xi,\ga}(\lm x+\xi)$. Then, by change of variables, we have $w_{\lm,\xi,\ga}^*$ solves
	\begin{eqnarray}\label{X0}
		\nabla^2\cj_0(\psi_{1,0,\ga})[w_{\lm,\xi,\ga}^*]&=&\frac12\sum_i\ih_{ii}(\lm x+\xi)\pa_i\cdot_{\ig_{\R^m}}\nabla_{\pa_i} \psi_{1,0,\ga}-\frac{\lm}4\nabla(\tr\ih)(\lm x+\xi)\cdot_{\ig_{\R^m}}\psi_{1,0,\ga}   \nonumber \\
		& &+\frac{\lm}4\sum_i\pa_i\ih_{ii}(\lm x+\xi)\pa_i\cdot_{\ig_{\R^m}}\psi_{1,0,\ga}
	\end{eqnarray}
	Taking to the limit $\lm\to0$, by \eqref{E1}, \eqref{E5}, \eqref{E6} and the continuity of $K_{\psi_{1,0,\ga}}=\nabla^2\cj_0(\psi_{1,0,\ga})^{-1}$, we see that $w_{\lm,\xi,\ga}^*\to w_{0,\xi,\ga}^*$ in $\msd^{\frac12}(\R^m,\mbs(\R^m))$ where
	$w_{0,\xi,\ga}^*$ solves the equation
	\begin{\equ}\label{X2}
		\nabla^2\cj_0(\psi_{1,0,\ga})[w_{0,\xi,\ga}^*]=\frac12\sum_i\ih_{ii}(\xi)\pa_i\cdot_{\ig_{\R^m}}\nabla_{\pa_i} \psi_{1,0,\ga}.
	\end{\equ}
	
	The solution of Eq. \eqref{X2} is not unique. It is uniquely determined up to the addition of elements of $\ker\nabla^2\cj_0(\psi_{1,0,\ga})=T_{\psi_{1,0,\ga}}\cm$. However, by changing of variables again, we find the quantity
	\[
	\lim_{\lm\to0}\inp{\nabla\Ga(\psi_{\lm,\xi,\ga})}{\bar w_{\lm,\xi,\ga}}
	=-\frac12\sum_i\ih_{ii}(\xi)\real\int_{\R^m}(\pa_i\cdot_{\ig_{\R^m}}\nabla_{\pa_i} \psi_{1,0,\ga},w_{0,\xi,\ga}^*)_{\ig_{\R^m}}d\vol_{\ig_{\R^m}}
	\]
	does not depend on the projection of $w_{0,\xi,\ga}^*$ in $T_{\psi_{1,0,\ga}}\cm$ (since $\pa_i\cdot_{\ig_{\R^m}}\nabla_{\pa_i} \psi_{1,0,\ga}\in T_{\psi_{1,0,\ga}}\cm^\bot$). Hence, it is enough for us to construct one solution to \eqref{X2}. Observing that Eq. \eqref{X2} is linear in $\ih$, one checks easily that
	\[
	w_{0,\xi,\ga}^*(x)=\frac12\sum_i\ih_{ii}(\xi)x_i\nabla_{\pa_i}\psi_{1,0,\ga}(x)
	\]
	is indeed a solution. Moreover, following from\eqref{d psi}, \eqref{Cd psi} and the rules of Clifford multiplication, we see that
	\[
	\real(\pa_i\cdot_{\ig_{\R^m}}\nabla_{\pa_i} \psi_{1,0,\ga},x_i\nabla_{\pa_i}\psi_{1,0,\ga})=0
	\]
	for $i=1,\dots,m$ and
	\[
	\aligned
	&\real(\pa_i\cdot_{\ig_{\R^m}}\nabla_{\pa_i} \psi_{1,0,\ga},x_j\nabla_{\pa_j}\psi_{1,0,\ga})\\
	&\qquad
	=\frac{m^mx_ix_j}{\big(1+|x|^2 \big)^{m+1}}\real(\pa_i\cdot_{\ig_{\R^m}}(1-x)\cdot_{\ig_{\R^m}}\ga,\pa_j\cdot_{\ig_{\R^m}}\ga)
	-\frac{m^mx_j^2}{\big(1+|x|^2 \big)^{m+1}} \\
	&\qquad
	=\frac{m^mx_ix_j}{\big(1+|x|^2 \big)^{m+1}}\real(\pa_i\cdot_{\ig_{\R^m}}\ga,\pa_j\cdot_{\ig_{\R^m}}\ga) \\
	&\qquad \quad
	-\sum_{\substack{k=1 \\k\neq i,\ k\neq j}}^m \frac{m^mx_ix_jx_k}{\big(1+|x|^2 \big)^{m+1}}\real(\pa_i\cdot_{\ig_{\R^m}}\pa_k\cdot_{\ig_{\R^m}}\ga,\pa_j\cdot_{\ig_{\R^m}}\ga)
	-\frac{m^mx_j^2}{\big(1+|x|^2 \big)^{m+1}}
	\endaligned
	\]
	for $i\neq j$. Since $\real(\pa_i\cdot_{\ig_{\R^m}}\ga,\pa_j\cdot_{\ig_{\R^m}}\ga)\equiv 0$ and
	$\real(\pa_i\cdot_{\ig_{\R^m}}\pa_k\cdot_{\ig_{\R^m}}\ga,\pa_j\cdot_{\ig_{\R^m}}\ga)$ are constants, we deduce
	\begin{eqnarray*}
		\lim_{\lm\to0}\inp{\nabla\Ga(\psi_{\lm,\xi,\ga})}{\bar w_{\lm,\xi,\ga}}
		&=&-\frac12\sum_i\ih_{ii}(\xi)\real\int_{\R^m}(\pa_i\cdot_{\ig_{\R^m}}\nabla_{\pa_i} \psi_{1,0,\ga},w_{0,\xi,\ga}^*)_{\ig_{\R^m}}d\vol_{\ig_{\R^m}} \\
		&=&\frac14\sum_{\substack{i,j \\i\neq j}}\ih_{ii}(\xi)\ih_{jj}(\xi)\int_{\R^m}\frac{m^mx_j^2}{\big(1+|x|^2 \big)^{m+1}} d\vol_{\ig_{\R^m}} \\
		&=&\frac14\int_{\R^m}\frac{m^mx_1^2}{\big(1+|x|^2 \big)^{m+1}} d\vol_{\ig_{\R^m}} \big((\tr\ih)^2 - \tr(\ih^2)\big)(\xi)
	\end{eqnarray*}
	Notice that $|x|^2=x_1^2+\cdots +x_m^2$ and
	\[
	\int_{\R^m}\frac{x_1^2}{\big(1+|x|^2 \big)^{m+1}} d\vol_{\ig_{\R^m}}=\cdots=\int_{\R^m}\frac{x_m^2}{\big(1+|x|^2 \big)^{m+1}} d\vol_{\ig_{\R^m}}
	\]
	we have
	\[
	\lim_{\lm\to0}\inp{\nabla\Ga(\psi_{\lm,\xi,\ga})}{\bar w_{\lm,\xi,\ga}}
	=\frac{m^{m-1}}4\int_{\R^m}\frac{|x|^2}{\big(1+|x|^2 \big)^{m+1}} d\vol_{\ig_{\R^m}} \big((\tr\ih)^2 - \tr(\ih^2)\big)(\xi)
	\]
	as desired.
\end{proof}

\medskip

\begin{proof}[Proof of Proposition \ref{behavior near 0}]
	The procedure is to find the expansion of $\hat\Phi(\psi_{\lm,\xi,\ga})$ with respect to $\lm$  in a neighborhood of $0$. In order to do this, we need to find the expansions for both of the functionals $\Phi(\psi_{\lm,\xi,\ga})$ and $\inp{K_{\psi_{\lm,\xi,\ga}}(\nabla\Ga(\psi_{\lm,\xi,\ga}))}{\nabla\Ga(\psi_{\lm,\xi,\ga})}$.

	For $m\geq3$, we can assume without loss of generality that $\ih=\diag(\ih_{11},\dots,\ih_{m-1 m-1},0)$ is $(m,p)$-elementary around the point $\xi\in\R^m$ with $\pa_m\ih_{ii}(\xi)\equiv c_m\neq0$, for $1\leq i\leq m-1$. The other cases can be done in a very similar manner. 
	
	By Definition \ref{def k-elementary} and by the change of variables $x\mapsto \lm x+\xi$, we find
	\[
	\ih_{ii}(\lm x+\xi)=\ih_{ii}(\xi)+\lm(c_ix_i+c_mx_m)+o(|\lm x|^p)
	\]	
	for $i=1,\dots,m-1$, and hence
	\[
	\aligned
	\big(\tr(\ih^2)-(\tr\ih)^2\big)(\lm x+\xi)&=-\sum_{\substack{i,j \\i\neq j}}(\ih_{ii}\ih_{jj})(\lm x+\xi) \\
	& \hspace{-3em}  =-\sum_{\substack{i,j \\i\neq j}}\Big[\ih_{ii}(\xi)\ih_{jj}(\xi)  +\lm\ih_{ii}(\xi)(c_jx_j+c_mx_m) +\lm\ih_{jj}(\xi)(c_ix_i+c_mx_m) \\
	& + \lm^2\big(c_ic_jx_ix_j+c_ic_mx_ix_m+c_jc_mx_jx_m+c_m^2x_m^2\big)+o(|\lm x|^p)\Big].
	\endaligned
	\]
	Terms of order $\lm$ in $\Phi(\psi_{\lm,\xi,\ga})$	have coefficients like
	\[
	\int_{\R^m}\frac{x_i}{\big( 1+|x|^2 \big)^m}d\vol_{\ig_{\R^m}}
	\]
	which are all zero. And the coefficients of $\lm^2$ are of the type
	\[
	\int_{\R^m}\frac{x_ix_j}{\big( 1+|x|^2 \big)^m}d\vol_{\ig_{\R^m}} 
	\]
	and such integral are non-zero only for $i=j$. Hence, by Lemma \ref{lemma Phi1} and $p\geq2$, we deduce
	\begin{\equ}\label{expand Phi}
		\Phi(\psi_{\lm,\xi,\ga})=C_0\big( \tr(\ih^2)-(\tr\ih)^2 \big)(\xi) -A c_m^2  \lm^2 + o(\lm^2)
	\end{\equ}
	where
	\[
	A=\frac{m^{m-2}(m-1)(m-2)}{16}\int_{\R^m}\frac{|x|^2}{\big( 1+|x|^2 \big)^m}d\vol_{\ig_{\R^m}} .
	\]
	
	It remains to find an expansion for $\inp{K_{\psi_{\lm,\xi,\ga}}(\nabla\Ga(\psi_{\lm,\xi,\ga}))}{\nabla\Ga(\psi_{\lm,\xi,\ga})}$. By Eq. \eqref{X0} and the local expansion of $\ih_{ii}$, we want to develop $w_{\lm,\xi,\ga}^*$ in powers of $\lm$.
	
	We first assume $m\geq6$. Then, by the computations \eqref{E1}-\eqref{E7} in the Appendix, we have 
	\[
	w_{\lm,\xi,\ga}^*=w_{0,\xi,\ga}^*+\lm\nu_{0,\xi,\ga}^*+\lm^2\mu_{0,\xi,\ga}^*+o(\lm^2) \quad \text{in } \msd^{\frac12}(\R^m,\mbs(\R^m))
	\]
	where $w_{0,\xi,\ga}^*$, $\nu_{0,\xi,\ga}^*$ and $\mu_{0,\xi,\ga}^*$ satisfy the following equations
	\begin{\equ}\label{EEE1}
		\nabla^2\cj_0(\psi_{1,0,\ga})[w_{0,\xi,\ga}^*]=\frac12\sum_{i=1}^{m-1}\ih_{ii}(\xi)\pa_i\cdot_{\ig_{\R^m}}\nabla_{\pa_i}\psi_{1,0,\ga},
	\end{\equ}
	\begin{\equ}\label{EEE2}
		\nabla^2\cj_0(\psi_{1,0,\ga})[\nu_{0,\xi,\ga}^*]=\frac12\sum_{i=1}^{m-1}(c_ix_i+c_mx_m)\pa_i\cdot_{\ig_{\R^m}}\nabla_{\pa_i}\psi_{1,0,\ga}-\frac{m-1}4c_m\pa_m\cdot_{\ig_{\R^m}}\psi_{1,0,\ga}
	\end{\equ}
	and (since the Hessian of $\ih_{ii}$ at $\xi$ vanishes identically)
	\begin{\equ}\label{EEE3}
		\nabla^2\cj_0(\psi_{1,0,\ga})[\mu_{0,\xi,\ga}^*]=0.
	\end{\equ}
	And now we have
	\[
	\aligned
	&\inp{K_{\psi_{\lm,\xi,\ga}}(\nabla\Ga(\psi_{\lm,\xi,\ga}))}{\nabla\Ga(\psi_{\lm,\xi,\ga})}
	=\nabla^2\cj_0(\psi_{1,0,\ga})[w_{\lm,\xi,\ga}^*,w_{\lm,\xi,\ga}^*] \\
	&\qquad =\nabla^2\cj_0(\psi_{1,0,\ga})[w_{0,\xi,\ga}^*,w_{0,\xi,\ga}^*]
	+2\nabla^2\cj_0(\psi_{1,0,\ga})[w_{0,\xi,\ga}^*,\nu_{0,\xi,\ga}^*] \lm\\
	&\quad \qquad +\big[ 2\nabla^2\cj_0(\psi_{1,0,\ga})[w_{0,\xi,\ga}^*,\mu_{0,\xi,\ga}^*]
	+ \nabla^2\cj_0(\psi_{1,0,\ga})[\nu_{0,\xi,\ga}^*,\nu_{0,\xi,\ga}^*]\big]\lm^2+o(\lm^2).
	\endaligned
	\]
	Next, we will calculate the coefficients in front of each power of $\lm$:

	\medskip
	
	1. The zero order term
	
	As was shown by Lemma \ref{lemma Phi2}, we have
	\[
	\nabla^2\cj_0(\psi_{1,0,\ga})[w_{0,\xi,\ga}^*,w_{0,\xi,\ga}^*]=C_1\big( \tr(\ih^2)-(\tr\ih)^2 \big)(\xi).
	\]
	
	\medskip
	
	2. The first order term
	
	By using the equation \eqref{EEE2}, we find
	\[
	\aligned
	\nabla^2\cj_0(\psi_{1,0,\ga})[w_{0,\xi,\ga}^*,\nu_{0,\xi,\ga}^*]
	&=\frac12\sum_{i=1}^{m-1}\real\int_{\R^m}(c_ix_i+c_mx_m) (\pa_i\cdot_{\ig_{\R^m}}\nabla_{\pa_i}\psi_{1,0,\ga},w_{0,\xi,\ga}^*)_{\ig_{\R^m}}d\vol_{\ig_{\R^m}} \\[0.2em]
	&\qquad -\frac{(m-1)c_m}4\real\int_{\R^m}(\pa_m\cdot_{\ig_{\R^m}}\psi_{1,0,\ga},w_{0,\xi,\ga}^*)_{\ig_{\R^m}}d\vol_{\ig_{\R^m}}.
	\endaligned
	\]
	Note that $w_{0,\xi,\ga}^*$ can be explicitly formulated as
	\[
	w_{0,\xi,\ga}^*=\frac12\sum_{i=1}^{m-1}\ih_{ii}(\xi)x_i\nabla_{\pa_i}\psi_{1,0,\ga},
	\]
	we find the integrals in $\nabla^2\cj_0(\psi_{1,0,\ga})[w_{0,\xi,\ga}^*,\nu_{0,\xi,\ga}^*]$ are of the type
	\[
	\int_{\R^m}\frac{x_ix_j}{\big(1+|x|^2\big)^{m}}d\vol_{\ig_{\R^m}}, \quad 
	\int_{\R^m}\frac{x_i^2x_j}{\big(1+|x|^2\big)^{m+1}}d\vol_{\ig_{\R^m}}, \quad 
	\int_{\R^m}\frac{x_i^2x_kx_j}{\big(1+|x|^2\big)^{m+1}}d\vol_{\ig_{\R^m}},
	\]
	\[
	\int_{\R^m}\frac{x_i^3}{\big(1+|x|^2\big)^{m}}d\vol_{\ig_{\R^m}} \quad \text{and} \quad \int_{\R^m}\frac{x_ix_jx_kx_l}{\big(1+|x|^2\big)^{m}}d\vol_{\ig_{\R^m}}
	\]
	with $i\neq j $, $i\neq k$, $j\neq k$ and $k\neq l$, which are all zero. Hence the first order term vanishes identically.
	
	\medskip
	
	3. The second order term
	
	By virtue of \eqref{EEE3}, we only need to evaluate $\nabla^2\cj_0(\psi_{1,0,\ga})[\nu_{0,\xi,\ga}^*,\nu_{0,\xi,\ga}^*]$.	To this end, it sufficient to solve Eq. \eqref{EEE2}. Analogous to solving Eq. \eqref{X2}, we only need to construct one solution to \eqref{EEE2}. Particularly, one checks that
	\[
	\nu_{0,\xi,\ga}^*=\frac14\Big(\sum_{i=1}^{m-1} c_ix_i^2\nabla_{\pa_i}\psi_{1,0,\ga}-c_mx_m^2\nabla_{\pa_m}\psi_{1,0,\ga} \Big)-\frac{m-1}4c_mx_m\psi_{1,0,\ga}
	\]
	solves Eq. \eqref{EEE2}. And hence, by the explicit expression of $\psi_{1,0,\ga}$, we have
	\[
	\aligned
	\nabla^2\cj_0(\psi_{1,0,\ga})[\nu_{0,\xi,\ga}^*,\nu_{0,\xi,\ga}^*]
	&=\frac12\sum_{i=1}^{m-1}\real\int_{\R^m}(c_ix_i+c_mx_m)(\pa_i\cdot_{\ig_{\R^m}}\nabla_{\pa_i}\psi_{1,0,\ga},\nu_{0,\xi,\ga}^*)_{\ig_{\R^m}} d\vol_{\ig_{\R^m}}\\
	&\qquad
	-\frac{(m-1)c_m}4\real\int_{\R^m} (\pa_m\cdot_{\ig_{\R^m}}\psi_{1,0,\ga}, \nu_{0,\xi,\ga}^*)_{\ig_{\R^m}} d\vol_{\ig_{\R^m}} \\
	&=\frac12\sum_{i=1}^{m-1}\int_{\R^m}\bigg( \frac{m^mc_m^2x_m^4}{4(1+|x|^2)^{m+1}}
	-\frac{m^{m-1}(m-1)c_m^2x_m^2}{4(1+|x|^2)^m} \bigg) d\vol_{\ig_{\R^m}}\\
	&\qquad 
	-\frac{(m-1)c_m^2}4\int_{\R^m}\frac{m^{m-1}x_m^2}{4(1+|x|^2)^m}d\vol_{\ig_{\R^m}} \\[0.3em]
	&=-\frac{5m^{m-2}(m-1)(m-2)c_m^2}{64}\int_{\R^m}\frac{|x|^2}{(1+|x|^2)^m}d\vol_{\ig_{\R^m}},
	\endaligned
	\]
	where the last equality follows from the fact
	\[
	\int_{\R^m}\frac{x_m^4}{(1+|x|^2)^{m+1}}d\vol_{\ig_{\R^m}}=\frac{3(m+2)}{8m}\int_{\R^m}\frac{x_m^2}{(1+|x|^2)^{m}}d\vol_{\ig_{\R^m}}  \text{ for } m\geq3.
	\]
	
	Now, together with \eqref{expand Phi}, we find
	\[
	\hat\Phi(\psi_{\lm,\xi,\ga})=-\frac{3m^{m-2}(m-1)(m-2)c_m^2}{128} \lm^2\int_{\R^m}\frac{|x|^2}{(1+|x|^2)^m}d\vol_{\ig_{\R^m}}+o(\lm^2)
	\]
	as desired.
	
	\medskip
	
	For $m=5$, by \eqref{E1}-\eqref{E9}, we find
	\begin{\equ}\label{expansion 5D}
		w_{\lm,\xi,\ga}^*=w_{0,\xi,\ga}^*+\lm \nu_{0,\xi,\ga}^*+ \mu_{\lm,\xi,\ga}^* + \rho_{\lm,\xi,\ga}^* \quad \text{in } \msd^{\frac12}(\R^5,\mbs(\R^5))
	\end{\equ}
	where
	\[
	\nabla^2\cj_0(\psi_{1,0,\ga})[w_{0,\xi,\ga}^*]=\frac12\sum_{i=1}^{4}\ih_{ii}(\xi)\pa_i\cdot_{\ig_{\R^5}}\nabla_{\pa_i}\psi_{1,0,\ga}
	\]
	\[	\nabla^2\cj_0(\psi_{1,0,\ga})[\nu_{0,\xi,\ga}^*]=\frac12\sum_{i=1}^{4}(c_ix_i+c_5x_5)\pa_i\cdot_{\ig_{\R^5}}\nabla_{\pa_i}\psi_{1,0,\ga}-c_5\pa_5\cdot_{\ig_{\R^5}}\psi_{1,0,\ga} 
	\]
	\[
	\aligned
	\nabla^2\cj_0(\psi_{1,0,\ga})[\mu_{\lm,\xi,\ga}^*]&=(1-\eta_\lm)\Bigg[\frac12\sum_{i=1}^{5}\ih_{ii}(\lm x+\xi)\pa_i\cdot_{\ig_{\R^5}}\nabla_{\pa_i}\psi_{1,0,\ga} \\
	&\quad  -\frac{\lm}4(\nabla\tr\ih)(\lm x+\xi)\cdot_{\ig_{\R^5}}\psi_{1,0,\ga} + \frac{\lm}4\sum_{i=1}^5\pa_i\ih_{ii}(\lm x+\xi) \pa_i\cdot_{\ig_{\R^5}}\psi_{1,0,\ga} \\ &\quad -\frac{1}2\sum_{i=1}^{4}\Big[\ih_{ii}(\xi)+\lm( c_ix_i+ c_5x_5)\Big]\pa_i\cdot_{\ig_{\R^5}}\nabla_{\pa_i}\psi_{1,0,\ga} + \lm c_5\pa_5\cdot_{\ig_{\R^5}}\psi_{1,0,\ga} \Bigg]
	\endaligned
	\]
	and 
	\[
	\nabla^2\cj_0(\psi_{1,0,\ga})[\rho_{\lm,\xi,\ga}^*]= o(\lm^p)\eta_\lm\big(|x|^p|\nabla_{\pa_i}\psi_{1,0,\ga}|+ |x|^{p-1}|\psi_{1,0,\ga}|\big).
	\]
	In the above expressions, $\eta_\lm$ is a smooth cut-off function such that $\eta_\lm(x)\equiv1$ for $|x|\leq\de/\lm$ and $\eta_\lm(x)\equiv0$ for $|x|>{2\de}/\lm$ with $\de>0$ being fixed suitably small. We emphasize that, by \eqref{E3}, \eqref{E6}, \eqref{E8} and \eqref{E9}, we have $\nabla^2\cj_0(\psi_{1,0,\ga})[\mu_{\lm,\xi,\ga}^*]=O(\lm^2)$ and $\nabla^2\cj_0(\psi_{1,0,\ga})[\rho_{\lm,\xi,\ga}^*]=o(\lm^2)$.

	By using the expansion \eqref{expansion 5D}, we have
	\[
	\aligned
	&\inp{K_{\psi_{\lm,\xi,\ga}}(\nabla\Ga(\psi_{\lm,\xi,\ga}))}{\nabla\Ga(\psi_{\lm,\xi,\ga})}
	=\nabla^2\cj_0(\psi_{1,0,\ga})[w_{\lm,\xi,\ga}^*,w_{\lm,\xi,\ga}^*] \\
	&\qquad =\nabla^2\cj_0(\psi_{1,0,\ga})[w_{0,\xi,\ga}^*,w_{0,\xi,\ga}^*]
	+2\nabla^2\cj_0(\psi_{1,0,\ga})[w_{0,\xi,\ga}^*,\nu_{0,\xi,\ga}^*] \lm\\
	&\quad \qquad + 2\nabla^2\cj_0(\psi_{1,0,\ga})[w_{0,\xi,\ga}^*,\mu_{\lm,\xi,\ga}^*]
	+ \nabla^2\cj_0(\psi_{1,0,\ga})[\nu_{0,\xi,\ga}^*,\nu_{0,\xi,\ga}^*]\lm^2+o(\lm^2).
	\endaligned
	\]
	Since $\nabla^2\cj_0(\psi_{1,0,\ga})[\mu_{\lm,\xi,\ga}^*]=O(\lm^2)$, $2^*=\frac52$
	and $\int_{|x|>\frac\de\lm}|w_{0,\xi,\ga}^*|^{\frac{5}{2}}d\vol_{\ig_{\R^5}}\to0$ as $\lm\to0$, we find
	\[
	\nabla^2\cj_0(\psi_{1,0,\ga})[w_{0,\xi,\ga}^*,\mu_{\lm,\xi,\ga}^*]=\nabla^2\cj_0(\psi_{1,0,\ga})[\mu_{\lm,\xi,\ga}^*,w_{0,\xi,\ga}^*]=o(\lm^2)
	\]
	Then we see that the coefficient of the first order term vanishes as in the case $m\geq6$ and the coefficient in front of $\lm^2$ is $\nabla^2\cj_0(\psi_{1,0,\ga})[\nu_{0,\xi,\ga}^*,\nu_{0,\xi,\ga}^*]$. Therefore we can proceed as in the  case $m\geq6$ to obtain our conclusion.
	
	\medskip

	For $m=4$, by using the stronger asymptotic condition on $\ih$,  we find
	\begin{\equ}\label{expansion 4D}
		w_{\lm,\xi,\ga}^*=w_{0,\xi,\ga}^*+\lm \nu_{0,\xi,\ga}^*+ \mu_{\lm,\xi,\ga}^* \quad \text{in } \msd^{\frac12}(\R^4,\mbs(\R^4))
	\end{\equ}
	where 
	\[
	\nabla^2\cj_0(\psi_{1,0,\ga})[w_{0,\xi,\ga}^*]=\frac12\sum_{i=1}^{3}\ih_{ii}(\xi)\pa_i\cdot_{\ig_{\R^4}}\nabla_{\pa_i}\psi_{1,0,\ga}
	\]
	\[	\nabla^2\cj_0(\psi_{1,0,\ga})[\nu_{0,\xi,\ga}^*]=\frac12\sum_{i=1}^{3}(c_ix_i+c_4x_4)\pa_i\cdot_{\ig_{\R^4}}\nabla_{\pa_i}\psi_{1,0,\ga}-\frac34c_4\pa_4\cdot_{\ig_{\R^4}}\psi_{1,0,\ga} 
	\]
	\[
	\aligned
	\nabla^2\cj_0(\psi_{1,0,\ga})[\mu_{\lm,\xi,\ga}^*]&=(1-\eta_\lm)\Bigg[\frac12\sum_{i=1}^4\ih_{ii}(\lm x+\xi)\pa_i\cdot_{\ig_{\R^4}}\nabla_{\pa_i}\psi_{1,0,\ga} \\
	&\quad -\frac\lm4\nabla\tr(\ih)(\lm x+\xi)\cdot_{\ig_{\R^4}}\psi_{1,0,\ga}  +\frac\lm4\sum_{i=1}^4\pa_i\ih_{ii}(\lm x+\xi)\pa_i\cdot_{\ig_{\R^4}}\psi_{1,0,\ga} \\
	&\quad -\frac12\sum_{i=1}^3\Big[ \ih_{ii}(\xi)+\lm (c_ix_i+ c_4x_4) \Big]\pa_i\cdot_{\ig_{\R^4}}\nabla_{\pa_i}\psi_{1,0,\ga}  + \frac{3\lm c_4}4 \pa_4\cdot_{\ig_{\R^4}}\psi_{1,0,\ga} \Bigg],
	\endaligned
	\]
	and $\eta_\lm$ is a smooth cut-off function as before. Following from \eqref{E1}, \eqref{E3} and \eqref{E6}, we can see that $\nabla^2\cj_0(\psi_{1,0,\ga})[\mu_{\lm,\xi,\ga}^*]=O(\lm^{\frac32})$. Then, by using the expansion \eqref{expansion 4D}, we have
	\[
	\aligned
	&\inp{K_{\psi_{\lm,\xi,\ga}}(\nabla\Ga(\psi_{\lm,\xi,\ga}))}{\nabla\Ga(\psi_{\lm,\xi,\ga})}
	=\nabla^2\cj_0(\psi_{1,0,\ga})[w_{\lm,\xi,\ga}^*,w_{\lm,\xi,\ga}^*] \\
	&\qquad =\nabla^2\cj_0(\psi_{1,0,\ga})[w_{0,\xi,\ga}^*,w_{0,\xi,\ga}^*]
	+2\nabla^2\cj_0(\psi_{1,0,\ga})[w_{0,\xi,\ga}^*,\nu_{0,\xi,\ga}^*] \lm\\
	&\quad \qquad + 2\nabla^2\cj_0(\psi_{1,0,\ga})[w_{0,\xi,\ga}^*,\mu_{\lm,\xi,\ga}^*]
	+ \nabla^2\cj_0(\psi_{1,0,\ga})[\nu_{0,\xi,\ga}^*,\nu_{0,\xi,\ga}^*]\lm^2+o(\lm^2).
	\endaligned
	\]
	Since $\nabla^2\cj_0(\psi_{1,0,\ga})[\mu_{\lm,\xi,\ga}^*]=O(\lm^{\frac32})$, $2^*=\frac83$
	and $\big(\int_{|x|>\frac\de\lm}|w_{0,\xi,\ga}^*|^{\frac{8}{3}}d\vol_{\ig_{\R^4}}\big)^{\frac38}= O(\lm^{\frac32})$ as $\lm\to0$, we find
	\[
	\nabla^2\cj_0(\psi_{1,0,\ga})[w_{0,\xi,\ga}^*,\mu_{\lm,\xi,\ga}^*]=\nabla^2\cj_0(\psi_{1,0,\ga})[\mu_{\lm,\xi,\ga}^*,w_{0,\xi,\ga}^*]=o(\lm^2)
	\]
	Hence, apart from the zero order term, the first non-vanishing coefficient is still in front of $\lm^2$ and equals to $\nabla^2\cj_0(\psi_{1,0,\ga})[\nu_{0,\xi,\ga}^*,\nu_{0,\xi,\ga}^*]$. Then we can proceed as in the higher dimensional cases. This completes the whole proof. 
\end{proof}

\begin{Rem}%\label{behavior near 0 remark}
	In $3$-dimension, as is computed in \eqref{E2} and \eqref{E5} in the Appendix, the $L^{\frac32}$-norm of the functions $x_j\pa_i\cdot_{\ig_{\R^3}}\nabla_{\pa_i}\psi_{1,0,\ga}$ and $\pa_k\cdot_{\ig_{\R^3}}\psi_{1,0,\ga}$ become infinite. This causes that $w_{\lm,\xi,\ga}^*$ has no more of order $\lm$. Precisely, we see that the equation 
	\[	\nabla^2\cj_0(\psi_{1,0,\ga})[\nu_{0,\xi,\ga}^*]=\frac12\sum_{i=1}^{2}(c_ix_i+c_3x_3)\pa_i\cdot_{\ig_{\R^3}}\nabla_{\pa_i}\psi_{1,0,\ga}-\frac12c_3\pa_3\cdot_{\ig_{\R^3}}\psi_{1,0,\ga} 
	\]
	is no longer defined in the dual space of $\msd^{1/2}(\R^3,\mbs(\R^3))$ and we can only expand $w_{\lm,\xi,\ga}^*$ in the form
	\[
	w_{\lm,\xi,\ga}^*=w_{0,\xi,\ga}^*+\nu_{\lm,\xi,\ga}^*
	\]
	with $\|\nu_{\lm,\xi,\ga}^*\|=O(\lm|\ln\lm|^{\frac23})$. And hence $\nabla^2\cj_0(\psi_{1,0,\ga})[w_{\lm,\xi,\ga}^*,w_{\lm,\xi,\ga}^*]$ has no more of order $\lm^2$.
	Comparing with the asymptotic expansion of $w_{\lm,\xi,\ga}^*$ for $m\geq4$, it becomes more complicated to determine the exact asymptotic expansion of $w_{\lm,\xi,\ga}^*$ in dimension $3$. Our approach in this regard up to now have failed.
\end{Rem}

\end{document}